\documentclass[12pt,english]{amsart}
\usepackage[latin9]{inputenc}
\usepackage{babel}
\usepackage{amsbsy}
\usepackage{amstext}
\usepackage{amsthm}
\usepackage{amssymb}
\usepackage[unicode=true,pdfusetitle,
 bookmarks=true,bookmarksnumbered=false,bookmarksopen=false,
 breaklinks=false,pdfborder={0 0 1},backref=false,colorlinks=false]
 {hyperref}

\makeatletter
\numberwithin{equation}{section}
\numberwithin{figure}{section}
\theoremstyle{plain}
\newtheorem{thm}{\protect\theoremname}[section]
\theoremstyle{plain}
\newtheorem{cor}[thm]{\protect\corollaryname}
\theoremstyle{plain}
\newtheorem{conjecture}[thm]{\protect\conjecturename}
\theoremstyle{plain}
\newtheorem{lem}[thm]{\protect\lemmaname}
\theoremstyle{definition}
\newtheorem{defn}[thm]{\protect\definitionname}
\theoremstyle{plain}
\newtheorem{prop}[thm]{\protect\propositionname}
\theoremstyle{remark}
\newtheorem{rem}[thm]{\protect\remarkname}
\theoremstyle{definition}
\newtheorem*{example*}{\protect\examplename}

\usepackage{amsmath}
\usepackage{geometry}
\usepackage{enumitem}
\makeatother

\providecommand{\conjecturename}{Conjecture}
\providecommand{\corollaryname}{Corollary}
\providecommand{\definitionname}{Definition}
\providecommand{\examplename}{Example}
\providecommand{\lemmaname}{Lemma}
\providecommand{\propositionname}{Proposition}
\providecommand{\remarkname}{Remark}
\providecommand{\theoremname}{Theorem}

\begin{document}
\title{Moments of unramified 2-group extensions of quadratic fields}
\author{Jack Klys}
\begin{abstract}
Let $f\left(K\right)$ be the number of unramified extensions $L/K$
of a quadratic number field $K$ with $\mathrm{Gal}\left(L/K\right)=H$
and $\mathrm{Gal}\left(L/\mathbb{Q}\right)=G$ where $G$ is a central
extension of $\mathbb{F}_{2}^{n}$ by $\mathbb{F}_{2}$. We find a
function $g\left(K\right)$ such that $f/g$ has finite moments and
a distribution on its values. We show this distribution is a point
mass when $H$ is non-abelian and the Cohen-Lenstra distribution when
$H$ is abelian, despite the fact that the set of values of $f/g$
do not form a discrete set. We prove an explicit formula for $f$
as well as a refined counting function with local conditions. We also
determine correlations of such counting functions for different groups
$G$. Lastly we formulate a conjecture about moments and correlations
for any pair of 2-groups $\left(G,H\right)$.
\end{abstract}

\maketitle

\section{Introduction}

\subsection{Background}

The Cohen-Lenstra heuristics are a well known series of conjectures
and heuristics for the distributions of class groups of number fields.
There has been much recent interest and progress on these conjectures
and their variants, including non-abelian versions. In this paper
we make progress towards a potential (currently not yet formulated)
non-abelian version of the Cohen-Lenstra conjectures for 2-groups
in the setting of quadratic fields by proving distributional results
about unramified 2-extensions of quadratic number fields.

Let $\mathcal{D}_{X}^{\pm}$ denote the set of positive (resp. negative)
quadratic discriminants of absolute value less than $X$. For any
function $f\left(K\right)$ defined on the set of quadratic fields
let
\[
\boldsymbol{E}_{k}^{\pm}\left(f\right)=\lim_{X\longrightarrow\infty}\frac{\sum_{K,D_{K}\in\mathcal{D}_{X}^{\pm}}f^{k}\left(K\right)}{\left|\mathcal{D}_{X}^{\pm}\right|}
\]
 be the average of $f^{k}$ over positive (resp. negative) discriminants.

For any pair of groups $H\le G$ with $\left[G:H\right]=2$ define
\[
f\left(K\right)=\left|\left\{ L/K\text{ unramified}\mid\exists\phi:G\left(L/\mathbb{Q}\right)\cong G,\phi\left(G\left(L/K\right)\right)=H\right\} \right|.
\]
We call any such extension as in the definition of $f$ a $\left(G,H\right)$-extension.

The Cohen-Lenstra heuristics predict $\boldsymbol{E}_{k}^{\pm}\left(f\right)$
for all $k$ when $H$ is an abelian $p$-group with $p$ odd and
$G=H\rtimes C_{2}$ with $C_{2}$ acting by inversion \cite{cohenlenstra}.
In this case $f\left(K\right)=\left|\mathrm{Surj}\left(Cl_{K},H\right)\right|\cdot\left|\mathrm{Aut}\left(H\right)\right|^{-1}$.
In fact it can be shown that the Cohen-Lenstra conjectures about the
distribution of $Cl_{K,p}$ are equivalent to the statement that $\boldsymbol{E}_{1}^{\pm}(\left|\mathrm{Surj}\left(Cl_{K},H\right)\right|)=m_{\pm}$
for all finite abelian $p$-groups $H$, where $m_{\pm}$ equals $1/\mathrm{Aut}\left(H\right)$
and $1$ in the postive and negative cases respectively.

If we instead restrict to asking about the distribution of $Cl_{K,p}\left[p\right]$
this is equivalent to asking for a distribution of the values of the
function $\left|\mathrm{Surj}\left(Cl_{K},H\right)\right|$ and in
turn is equivalent to computing $\boldsymbol{E}_{k}^{\pm}(\left|\mathrm{Surj}\left(Cl_{K},\mathbb{Z}/p\mathbb{Z}\right)\right|)$
for all $k\in\mathbb{Z}_{\ge1}$. Thus computation of the moments
$\boldsymbol{E}_{k}^{\pm}\left(f\right)$ for any pair of groups $\left(G,H\right)$
as above is a natural generalization of the abelian Cohen-Lenstra
conjectures.

For any even group $G$ Wood conjectured that $\boldsymbol{E}_{1}^{\pm}\left(f\right)$
is finite when there is a unique conjugacy class of order 2 elements
not in $H$ and infinite otherwise \cite{Wood}. Alberts and the author
investigated the latter case for the pair $\left(G,H\right)=\left(Q_{8}\rtimes C_{2},C_{2}\right)$
by computing a finite value for $\boldsymbol{E}_{k}^{\pm}\left(f/3^{\omega\left(D_{k}\right)}\right)$
for all $k$ (here $Q_{8}\rtimes C_{2}$ is the central product of
$D_{8}$ and $C_{4}$ along $C_{2}$) where $f/3^{\omega\left(D_{k}\right)}$
is normalized precisely to make the moments finite \cite{alberts,albertsklys}
($\omega\left(n\right)$ denotes the number of unique prime divisors
of $n$). We refer the reader there for a more detailed exposition
of all the above results.

In this paper we consider the generalization of this to all 2-groups
$\left(G,H\right)$. In particular when $G$ is a central extension
of $\mathbb{F}_{2}^{n}$ by $\mathbb{F}_{2}$ we determine a value
$c$ such that the function $f/c^{\omega\left(D_{k}\right)}$ determines
a distribution, and we compute this distribution and its moments.
Furthermore we conjecture this value of $c$ works for all pairs of
finite 2-groups $\left(G,H\right)$.

\subsection{\label{subsec:Main-results}Main results}

Let $G$ be a finite $2$-group with $H\le G$ a subgroup such that
$\left[G:H\right]=2$. We will call the pair $\left(G,H\right)$ admissible
if there exists a $\left(G,H\right)$-extension. This implies $G$
can be generated by order 2 elements not contained in $H$. For example
any pair $\left(G,H\right)$ with $G$ an abelian 2-group of exponent
greater than 2 is not admissible.

Let $\Phi\left(G\right)$ be the Frattini subgroup of $G$. Let $T_{0}\subset G/\Phi\left(G\right)$
be the subset of all elements whose lifts to $G$ have order 2 and
are not contained in $H$. We will refer to $T_{0}$ as the maximal
admissible generating set. Let $c=c_{T_{0}}$ be the number of conjugacy
classes of order 2 elements in $G$ which are not contained in $H$.
This can be viewed as the number of orbits of the set of lifts of
$T_{0}$ acting on itself by conjugation. Let $\mathrm{Aut}_{H}(G,\Phi\left(G\right))$
be the subgroup of $\mathrm{Aut}\left(G/\Phi\left(G\right)\right)$
preserving the set $H$ whose elements lift to $\mathrm{Aut}\left(G\right)$.

Let $\mu_{CL}^{\pm}\left(H\right)$ be the Cohen-Lenstra probability
measure on finite abelian 2-groups $H$ which is proportional to $1/(\left|\mathrm{Aut}H\right|\cdot\left|H\right|^{a})$
where $a=0,1$ in the positive (resp. negative) case. Let $P_{CL}^{\pm}\left(i\right)$
be the probability with respect to the Cohen-Lenstra measure that
a finite abelian 2-group has rank $i$.

Finally for any real valued function $h$ defined on quadratic discriminants
we will denote by $\mu_{h}$ the distribution on $\mathbb{R}$ (with
the Borel $\sigma$-algebra) determined by its values (see Section
\ref{subsec:Preliminaries-for-distributions} for a precise definition
of $\mu_{h}$).

Our main result is the following.
\begin{thm}
\label{thm:main thm}Let $G$ be a central extension of $\mathbb{F}_{2}^{n}$
by $\mathbb{F}_{2}$ with $\left(G,H\right)$ is admissible. Let $h\left(d\right)=f\left(d\right)/c^{\omega\left(d\right)}$
and $a=\left|\mathrm{Aut}_{H}\left(G,\Phi(G)\right)\right|$.

If $H$ is not abelian then the distribution of $h\left(d\right)$
is $\mu_{h}=\delta_{P}$ where $\delta_{P}$ is the point-mass at
$P\in\mathbb{Q}$ and $P\neq0$.

If $H$ is abelian then $h\left(d\right)$ has distribution $\mu_{h}$
supported on the set $\left\{ \left(2^{i}-1\right)/a\mid i\in\mathbb{Z}_{\ge0}\right\} $
given by
\[
\mu_{h}\left(\left(2^{i}-1\right)/2^{n-1}a\right)=P_{CL}^{\pm}\left(i\right)
\]
 over the set of real (respectively imaginary) quadratic fields.
\end{thm}

See (Theorem \ref{thm:k_moments pointmass}) for a precise formula
for $P$. Considering the function $f\left(d\right)/c^{\omega\left(d\right)}$
rather than $f\left(d\right)$ is crucial. The latter has infinite
moments of all orders and there is no distribution on its values.

We note that even though in the case when $H$ is abelian and $\mu_{h}$
is supported on a discrete subset of $\mathbb{Q}$, the values of
$h$ themselves do not form a discrete subset. In fact $h$ takes
values not in the support of $\mu_{h}$ with positive density in the
set of quadratic fields. The proof of Theorem \ref{thm:main thm}
shows that in this case $h$ can be partially expressed in terms of
$\left|Cl_{K}\left[4\right]/Cl_{K}\left[2\right]\right|$ and this
term is what controls the distribution of $h$ and results in the
occurence of the distribution $P_{CL}^{\pm}\left(i\right)$.

We can additionally deduce a result about moments of the above function
as well as correlations of functions for several pairs $\left(G_{i},H_{i}\right)$
. Let $N\left(k\right)$ denote the number of subspaces of $\mathbb{F}_{2}^{k}$.
\begin{thm}
Let $\left(G_{1},H_{1}\right),\ldots,\left(G_{k},H_{k}\right)$ be
a sequence of admissible pairs. Let $T_{i,0}$ be the maximal admissible
generating set corresponding to $\left(G_{i},H_{i}\right)$. Let $f_{i}$
be the function counting $\left(G_{i},H_{i}\right)$-extensions. Let
$Y\subset\left[k\right]$ be the set of indices $i$ for which $H_{i}$
is abelian. Then
\[
\lim_{X\longrightarrow\infty}\frac{\sum_{d\in\mathcal{D}_{X}^{\pm}}\prod_{i=1}^{k}\left(\frac{f_{i}\left(d\right)}{c_{T_{i,0}}^{\omega\left(d\right)}}\right)}{\sum_{d\in\mathcal{D}_{X}^{\pm}}1}=\left(\prod_{j=1}^{k}q_{j}^{-1}\right)\left(\prod_{j\in\left[k\right]\backslash Y}P_{j}\right)\left(\sum_{i=0}^{\left|Y\right|}\left(-1\right)^{i}\binom{\left|Y\right|}{i}M^{\pm}\left(\left|Y\right|-i\right)\right)
\]
 where
\begin{itemize}
\item $q_{j}=2^{n_{j}-1}\left|\mathrm{Aut}_{H_{j}}(G_{j},\Phi\left(G_{j}\right))\right|$
and $P_{j}\in\mathbb{Q}$
\item $M^{-}\left(j\right)=N\left(j\right)$ and $M^{+}\left(j\right)=2^{-j}\left(N\left(j+1\right)-N\left(j\right)\right)$
\end{itemize}
\end{thm}

See (Theorem \ref{thm:correlations}) for an explicit formula for
$P_{j}$. Note when $\left(G_{i},H_{i}\right)$ are all equal this
theorem gives the $k$th moment of $f\left(d\right)/c_{T_{0}}^{\omega\left(d\right)}$.

We have the following corollary of Theorem \ref{thm:main thm} on
the density of quotients of the Galois group of the maximal unramified
extension of quadratic fields.
\begin{cor}
\label{cor:main thm corollary}Let $\left(G_{1},H_{1},T_{0,1}\right),\ldots,\left(G_{k},H_{k},T_{0,k}\right)$
be a set of admissible tuples. Let $\mathcal{K}_{0}\subset\mathcal{K}$
be the set of all quadratic fields $K$ whose maximal unramified extension
$K^{un}/\mathbb{Q}$ contains a $\left(G_{i},H_{i}\right)$ extension
for all $i$.

If $H_{i}$ is non-abelian for all $i$ then $\mathcal{K}_{0}$ has
density 1. Otherwise $\mathcal{K}_{0}$ has density $1-P_{CL}\left(0\right)$.
\end{cor}

The first step to proving the above results is obtaining an explicit
formula for $f$ expressed as a sum of Legendre symbols. In fact we
obtain a formula for a refined counting function. Let $T\subset G/\Phi\left(G\right)$
be any subset. We will call $T$ admissible for $\left(G,H\right)$,
or say that $\left(G,H,T\right)$ is admissible if $T$ lifts to a
generating set of order 2 elements of $G$ not contained in $H$.

For any $L/K$ a $\left(G,H\right)$-extension there is a canonical
generating set $U$ for\\
 $G\left(L/\mathbb{Q}\right)/\Phi\left(G\left(L/\mathbb{Q}\right)\right)$
which is not contained in $G\left(L/K\right)$ and lifts to a generating
set of order 2 elements ($U$ is given by projecting the inertia groups).
We define
\[
f_{T}\left(K\right)=\left|\left\{ L/K\text{ unramified}\mid\exists\phi:\left(L/\mathbb{Q}\right)\cong G,\phi\left(G\left(L/K\right)\right)=H,\phi\left(U\right)=T\right\} \right|
\]
 and call such an extension a $\left(G,H,T\right)$-extension.

Let $\mathrm{Aut}_{H,T}(G,\Phi\left(G\right))$ be the subgroup of
$\mathrm{Aut}\left(G/\Phi\left(G\right)\right)$ preserving the sets
$H$ and $T$ whose elements lift to $\mathrm{Aut}\left(G\right)$.
Let $c_{T}$ denote the number of conjugacy classes of lifts of $T$.
This can be interpreted as the number of orbits of the set of lifts
of $T$ acting on itself by conjugation.
\begin{thm}
\label{thm:Let--and}Suppose $\left(G,H,T\right)$ is admissible.
Let $T=\left\{ t_{1},\ldots,t_{r}\right\} $ and let $\left\{ x_{1},\ldots,x_{r}\right\} \subset G$
be any set of lifts. Let $S_{i}=\left\{ 1\le j\le r\mid\left[x_{i},x_{j}\right]\neq1\right\} $.
Then
\begin{align*}
f_{T}\left(d\right) & =\frac{1}{2^{n}\left|\mathrm{Aut}_{H,T}(G,\Phi\left(G\right))\right|}\sum_{\left(d_{1},\ldots,d_{r}\right)}\prod_{i=1}^{r}\prod_{p\mid d_{i}}\left(1+\left(\frac{\prod_{j\in S_{i}}d_{j}}{p}\right)\right)
\end{align*}
where the sum is over tuples of coprime fundamental discriminants
such that $d=\prod_{i=1}^{r}d_{i}$.
\end{thm}

In the proof we utilize the Embedding Theorem \cite{cfields} which
gives a condition for the existence of the type of extensions we are
interested in. This formula generalizes the one in \cite{albertsklys}
used to count $Q_{8}$-extensions.

We define a graph $\mathcal{G}\left(T_{0}\right)$ with vertex set
$T_{0}$, and say $u,v\in T_{0}$ are connected by an edge if their
lifts do not commute. We show $H$ being abelian is equivalent to
$\mathcal{G}\left(T_{0}\right)$ being complete and bipartite and
split the proof up based on this property. We then proceeds by determining
the $k$th moments and hence the distribution of $f_{T_{0}}$ in the
case when $\mathcal{G}\left(T_{0}\right)$ is not complete bipartite,
and building on this information to obtain the distribution directly
in the complete bipartite case by re-expressing $f_{T_{0}}$ in terms
of the class group and other functions of the form $f_{T}$ and applying
measure-theoretic arguments.

The theorems of Fouvry and Kl\"uners on character sum cancellation
\cite{fk1} can be applied to $f_{T}$ to show that the computation
of the $k$th moments of $f_{T}$ depends on classifying the maximal
disconnected subgraphs of a graph $\mathcal{G}^{*}\left(T^{k}\right)$
(see Section \ref{sec:The graph} for the definition). In particular
the sizes of these subgraphs determine the asymptotics of $f_{T}$,
whereas their particular form is needed to compute the coefficient
of the main term.

While for $k=1$ $\mathcal{G}^{*}\left(T\right)$ is closely related
to $\mathcal{G}\left(T\right)$ above which is easy to handle, the
structure of $\mathcal{G}^{*}\left(T^{k}\right)$ is more complicated
and necessitates additional tools. Since $\mathcal{G}^{*}\left(T^{k}\right)$
depends on relations in $G$ we encode these using a bilinear form
in characteristic 2 which can be used to study the graph by employing
the theory of bilinear forms and linear algebra in characteristic
2 (see Section \ref{sec:isotropic subsets quad forms}).

\subsection{Related work}

There are some result known regarding the values $\boldsymbol{E}_{k}^{\pm}\left(f\right)$
for various pairs of groups $\left(G,H\right)$ and variants of $f$.

The well known theorem of Davenport-Heilbronn gives $\boldsymbol{E}_{1}^{\pm}\left(f\right)$
for $H\cong\mathbb{Z}/3\mathbb{Z}$ \cite{davenportHeilbronn}. Bhargava
computed $\boldsymbol{E}_{1}^{\pm}\left(f\right)$ for $\left(G,H\right)=\left(S_{n},A_{n}\right)$
for $n=3,4,5$ \cite{BhargavaNonabelian}. The work of Fouvry and
Kl\"uners on $4$-ranks of class groups of quadratic fields \cite{fk1}
(an extension of the Cohen-Lenstra heuristics to $p=2$) can be rephrased
as computing $\boldsymbol{E}_{k}^{\pm}\left(f/2^{\omega\left(D_{k}\right)}\right)$
for $\left(D_{8},C_{4}\right)$ for all $k$ (see Section \ref{subsec:Comparison-with-Fouvry-Klners}).

The Embedding Theorem \cite{cfields} used in the proof of Theorem
\ref{thm:Let--and} has been generalized by Alberts \cite{alberts2}
to a larger class of extensions, and consequently can be used to obtain
more general counting functions, the study of which we expect to form
the basis of future work.

\subsection{Conjectures}

We state some conjectures which are suggested by our main theorems.
\begin{conjecture}
\label{conj:Let--be}Let $G$ be a finite $2$-group. Let $H\le G$
be a subgroup with $\left[G:H\right]=2$ such that $\left(G,H\right)$
is admissible. Then
\begin{enumerate}
\item 
\[
0<\boldsymbol{E}_{k}^{\pm}\left(\frac{f}{c^{\omega\left(D_{K}\right)}}\right)<\infty
\]
and these moments determine a distribution on the values of $f$.
\item There exists a constant $C\left(G,H\right)$ such that 
\[
\sum_{K,0<\pm D_{K}<X}f^{k}\left(K\right)\sim C\left(G,H\right)X\left(\log X\right)^{c^{k}-1}.
\]
\end{enumerate}
\end{conjecture}

The suggestion that $c$ controls the asymptotics for even groups
appeared previously in the work of Wood. She has shown that a suitable
modification of the Malle-Bhargava field counting heuristics predicts
the above asymptotic for $k=1$ for any even group. She also proved
this asymptotic for $G$ elementary abelian and obtained a lower bound
in a function field analogue (see Section 6 and Theorem 1.2 in \cite{Wood}).

In Section \ref{subsec:Malle-Bhargava-heuristics} we restate her
heuristic argument in the language of this paper and show that it
can be modified to apply to the following refined conjecture involving
the counting function $f_{T}$.
\begin{conjecture}
\label{conj:Let--be-1}Let $\left(G_{1},H_{1},T_{1}\right),\ldots,\left(G_{k},H_{k},T_{k}\right)$
be a sequence of admissible tuples of finite 2-groups and generating
sets. Then
\begin{enumerate}
\item For all positive integers $k$
\[
0<\boldsymbol{E}_{1}^{\pm}\left(\frac{\prod_{i=1}^{k}f_{T_{i}}}{\prod_{i=1}^{k}c_{T_{i}}^{\omega\left(D_{K}\right)}}\right)<\infty.
\]
 If $\left(G_{i},H_{i},T_{i}\right)=\left(G,H,T\right)$ for all $i$
these moments determine a distribution on the values of $f_{T}$.
\item There exists a constant $C\left(G,H\right)$ such that 
\[
\sum_{K,0<\pm D_{K}<X}\prod_{i=1}^{k}f_{T_{i}}\left(K\right)\sim C\left(G,H,T\right)X\left(\log X\right)^{\prod_{i=1}^{k}c_{T_{i}}^{k}-1}.
\]
\end{enumerate}
\end{conjecture}

It is easy to show that Conjecture \ref{conj:Let--be} follows from
Conjecture \ref{conj:Let--be-1} using the fact that $c_{T}<c_{T_{0}}$
for any admissible generating set $T\subsetneq T_{0}$ and $f=\sum_{T\subseteq T_{0}}f_{T}$.

As a consequence of our main theorems we have
\begin{cor}
\label{cor:conj is true}Let $G$ be a central extension of $\mathbb{F}_{2}^{n}$
by $\mathbb{F}_{2}$ and let $\left(G,H,T\right)$ be admissible.
Suppose $G$ is not elementary abelian. Then Conjecture \ref{conj:Let--be}
(1) is true.

If either of the following conditions is true:
\begin{enumerate}
\item $T_{i}=T_{i,0}$ is the maximal admissible generating set for all
$i=1,\ldots,k$
\item $k=1$ and $T_{1}$ is any admissible generating set for $\left(G_{1},H_{1}\right)$
\end{enumerate}
then Conjecture \ref{conj:Let--be-1} (1) true.
\end{cor}

The proofs of our main theorems imply that in fact part (2) of Conjectures
\ref{conj:Let--be} and \ref{conj:Let--be-1} is true but for the
function counting extensions unramified away from infinity (indeed
throughout the proof we work with this counting function, and show
that when normalized it is asymptotic to $f_{T}/c^{\omega}$). This
at least gives a bound for the quantity in part (2) of the above conjectures.

\subsection{Plan of the proof}

In Section \ref{sec:Counting 2-extensions} we prove a formula for
$f$ the number of $\left(G,H,T\right)$-extensions of any quadratic
field. In Sections \ref{sec:isotropic subsets quad forms} we prove
necessary technical results about quadratic forms in characteristic
2. We define the graphs we will use in our computation of the moments
of $f$ in Section \ref{sec:The graph} and classify the disconnected
subgraphs in Section \ref{sec:Maximal disconnected subgraphs}.

Applying these results we compute the asymptotics and moments of $f$
in Sections \ref{sec:Asymptotic-analysis-of} and \ref{sec:Computing-the-moments}.
Finally we compute the distributions and moments and correlations
in Sections \ref{sec:Determining-the-distribution} and \ref{sec:Moments-and-Correlations}.

\section{Counting 2-extensions\label{sec:Counting 2-extensions}}

The main goal of this section is to obtain a formula for the function
$f_{T}\left(K\right)$ which gives the number of $\left(G,H,T\right)$-extensions
(defined below) of a quadratic field $K$ for any fixed admissible
$\left(G,H,T\right)$.

\subsection{\label{subsec:Definitions-and-preliminaries}Preliminaries}

Let $K=\mathbb{Q}\left(\sqrt{d}\right)$. Let $G$ be a central extension
of $\mathbb{F}_{2}^{n}$ by $\mathbb{F}_{2}$ and let $L$ be an unramified
extension of $K$ with $\phi:G\left(L/\mathbb{Q}\right)\cong G$ such
that $\phi\left(G\left(L/K\right)\right)=H$.

Let $L^{g}=L\cap K^{gen}$ where $K^{gen}=\mathbb{Q}\left(\sqrt{q_{1}},\ldots,\sqrt{q_{r}}\right)$
and the $q_{i}$ are prime discriminants with $\prod q_{i}=d$. It
is easy to show that $L^{g}=\mathbb{Q}\left(\sqrt{d_{1}'},\ldots,\sqrt{d_{n}'}\right)$
where $d_{i}'\mid d$ and are independent modulo $\left(\mathbb{Q}^{*}\right)^{2}$.

It is also clear that $G\left(L^{g}/\mathbb{Q}\right)\cong C_{2}^{n}$
and we can pick the isomorphism such that the standard basis element
$e_{i}\in C_{2}^{n}$ projects nontrivially onto $G_{i}=G\left(K(\sqrt{d_{i}'})/\mathbb{Q}\right)$.
Let $y_{i}$ be any chosen lift of $e_{i}$ in $G\left(L/\mathbb{Q}\right)$.
 Let $\left\langle a\right\rangle $ be the distinguished central
subgroup of $G$ of order 2. Note this implies there are only two
possibilities for $y_{i}$ and they differ by a multiple of $a$.
\begin{lem}
\label{lem:y generate}The set of $y_{i}$ generates $G\left(L/\mathbb{Q}\right)$.
\end{lem}

\begin{proof}
The $e_{i}$ generate $C_{2}^{n}\cong G\left(L/\mathbb{Q}\right)/\left\langle a\right\rangle $
so the $y_{i}$ generate at least $2^{n}$ elements in $G\left(L/\mathbb{Q}\right)\backslash\left\{ a\right\} $
and hence $\left\langle y_{i}\right\rangle _{i=1}^{n}\cup\left\{ a\right\} $
generates $G\left(L/\mathbb{Q}\right)$. Since $a\in Z\left(G\right)$
and $G\left(L/\mathbb{Q}\right)$ is not abelian there must exist
$y_{i}$ and $y_{j}$ which don't commute, so that $\left[y_{i},y_{j}\right]=a$.
\end{proof}
For any such extension $L/\mathbb{Q}$ let $U\subset G\left(L/\mathbb{Q}\right)/\Phi\left(G\left(L/\mathbb{Q}\right)\right)$
be the projection of the generators of all the inertia subgroups of
$G\left(L/\mathbb{Q}\right)$.
\begin{defn}
\label{def:ght-extension}We call $L/K$ an unramified (resp. unramified
away from infinity) $\left(G,H,T\right)$-extension if it is unramified
everywhere (resp. away from infinity) and there exists an isomorphism
$\phi:G\left(L/\mathbb{Q}\right)\cong G$ such that $\phi\left(G\left(L/K\right)\right)=H$
and $\overline{\phi}\left(U\right)=T$. We call $\phi$ the associated
isomorphism of the $\left(G,H,T\right)$-extension (note $\phi$ is
not unique in general).
\end{defn}

Until Section \ref{sec:Determining-the-distribution} we will work
with unramified away from infinity $\left(G,H,T\right)$-extensions,
and use $f_{T}$ to denote the number of these. Note this differs
from the definition used up to this point. To prove our theorems for
unramified everywhere $\left(G,H,T\right)$-extensions we will apply
Lemma \ref{lem:counting unram everywhere}. Thus until further notice
we will simply say $\left(G,H,T\right)$-extensions to refer to those
unramified away from infinity.

\subsection{\label{subsec:Counting-extensions-with}Extensions with a fixed genus
subfield}

Throughout this subsection we will fix an unramified away from infinity
extension $L'/K$ such that $G\left(L'/K\right)\cong C_{2}^{n-1}$
and $G\left(L'/\mathbb{Q}\right)\cong C_{2}^{n}$. We will count $\left(G,H,T\right)$-extensions
$L/K$ with $L^{g}=L'$. Fix an isomorphism $\phi':\mathrm{Gal}\left(L'/\mathbb{Q}\right)\cong G/\Phi\left(G\right)$
such that $\phi'\left(U\right)=T$ and $\phi'\left(\mathrm{Gal}\left(L'/K\right)\right)=\overline{H}$.

First we define a notion which will be helpful to this end, using
ideas of Lemmermeyer from \cite{lemmermeyerh8}. The field $L$ can
be written as $L=L^{g}\left(\sqrt{\mu}\right)$ for some $\mu\in L^{g}$.
Then since $L/L^{g}$ is a Kummer extension we see that $\mu^{e_{i}e_{j}}\stackrel{2}{=}\mu$
in $L^{g}$. Let $\alpha_{ij}^{2}=\mu^{e_{i}e_{j}-1}$. Then $\left(\alpha_{ij}^{2}\right)^{1+e_{i}e_{j}}=\mu^{\left(e_{i}e_{j}-1\right)\left(e_{i}e_{j}+1\right)}=1$.
Hence $\alpha_{ij}^{1+e_{i}e_{j}}=\pm1$. Define $S\left(\mu\right)\in\mathbb{F}_{2}^{\binom{n}{2}}$
by 
\[
S\left(\mu\right)_{ij}=\begin{cases}
1 & \text{if }\alpha_{ij}^{1+e_{i}e_{j}}=-1\\
0 & \text{if }\alpha_{ij}^{1+e_{i}e_{j}}=+1.
\end{cases}
\]

\begin{lem}
$y_{i}$ and $y_{j}$ commute in $G$ if and only if $S\left(\mu\right)_{ij}=0$.
\end{lem}

\begin{proof}
Since the $y_{i}$ have order 2 we have $\left[y_{i},y_{j}\right]=\left(y_{i}y_{j}\right)^{2}$.
We have $\sqrt{\mu}^{y_{i}y_{j}}=\alpha_{ij}\sqrt{\mu}$ and hence
$\left(a+b\sqrt{\mu}\right)^{\left(y_{i}y_{j}\right)^{2}}=a+\alpha_{ij}^{1+e_{i}e_{j}}b\sqrt{\mu}$
for $a,b\in L^{g}$. Thus $y_{i}y_{j}$ has order 2 if and only if
$\alpha_{ij}^{1+e_{i}e_{j}}=1$ and the result follows.
\end{proof}
This lemma in particular shows that the group $G$ is determined by
the choice of $e_{i}$'s along with $S\left(\mu\right)$ (up to permutation
of the entries) since by Lemma \ref{lem:y generate} $G$ is generated
by the $y_{i}$ which by definition have order 2, and $S\left(\mu\right)$
encodes the relations between them. Then it is clear that $S\left(\mu\right)=0$
in $\mathbb{F}_{2}^{\binom{n}{2}}$ if and only if $G=C_{2}^{n+1}$.
\begin{lem}
\label{lem:Suppose--and}Let $L_{1}=L_{1}^{g}\left(\sqrt{\mu}\right)$
and $L_{2}=L_{2}^{g}\left(\sqrt{\nu}\right)$ be $\left(G,H,T\right)$-extensions
with $L_{i}^{g}=L'$ and associated isomorphisms $\phi_{1}$ and $\phi_{2}$.
Assume $\overline{\phi_{1}}=\overline{\phi_{2}}=\phi'$.

Then $\mu\overset{2}{=}\nu\delta$ for some $\delta\in\mathbb{Z}$
a fundamental discriminant such that $\delta\mid d$. Conversely for
any fundamental discriminant $\delta\in\mathbb{Z}$ such that $\delta\mid d$
we have that $L_{1}^{g}\left(\sqrt{\delta\mu}\right)$ is a $\left(G,H,T\right)$-extension
with associated isomorphism that projects to $\phi'$.
\end{lem}

\begin{proof}
Let $M=L_{1}L_{2}$ and let $L_{3}$ be the third quadratic extension
of $L_{1}^{g}$ contained in $M$, so $L_{3}=L_{1}^{g}\left(\sqrt{\mu\nu}\right)$.
Let $\beta_{ij}^{2}=\nu^{e_{i}e_{j}-1}$. Then $\left(\alpha_{ij}\beta_{ij}\right)^{2}=\left(\mu\nu\right)^{e_{i}e_{j}-1}$.
Hence $S\left(\mu\nu\right)_{ij}=S\left(\mu\right)_{ij}\cdot S\left(\nu\right)_{ij}$.

Let $\varphi=\phi_{2}^{-1}\circ\phi_{1}$. Then $\varphi:G\left(L_{1}/\mathbb{Q}\right)\cong G\left(L_{2}/\mathbb{Q}\right)$
such that the morphism $\overline{\varphi}$ induced on the quotient
$G\left(L'/\mathbb{Q}\right)$ is the identity. From the existence
of $\varphi$ we see that $S\left(\mu\right)_{ij}=S\left(\nu\right)_{ij}$
for all $i,j$, and thus $G\left(L_{3}/\mathbb{Q}\right)\cong C_{2}^{n+1}$.
This implies $\mu\nu\overset{2}{=}\delta\in\mathbb{Z}$ and $\delta$
can be chosen to be a fundamental discriminant. Since $L_{1}$ and
$L_{2}$ are unramified over $K=\mathbb{Q}(\sqrt{d})$ so is $L_{3}$,
and hence $\delta\mid d$.

Now suppose $\delta\in\mathbb{Z}$ is a fundamental discriminant such
that $\delta\mid d$. From the definitions it is clear that $S\left(\mu\right)=S\left(\delta\mu\right)$
and hence there is an isomorphism $\varphi:G\left(L_{1}^{g}\left(\sqrt{\delta\mu}\right)/\mathbb{Q}\right)\cong G\left(L_{1}^{g}\left(\sqrt{\mu}\right)/\mathbb{Q}\right)$
such that $\overline{\varphi}=id$ on $G\left(L_{1}^{g}/\mathbb{Q}\right)$.
Additionally $L_{1}^{g}\left(\sqrt{\delta\mu}\right)/K$ is unramified
since it is a subfield of the composite of $L_{1}^{g}\left(\sqrt{\mu}\right)/K$
and $K\left(\sqrt{\delta}\right)/K$, both of which are unramified
(the latter because it is contained in $K^{gen}$). This implies $L_{1}^{g}\left(\sqrt{\delta\mu}\right)$
is a $\left(G,H,T\right)$-extension with associated isomorphism $\phi_{1}\circ\varphi$.
\end{proof}
\begin{prop}
\label{prop:If-there-exists} Let $L'=K(\sqrt{d_{1}'},\ldots\sqrt{d_{n}'})$.
If there exists a $\left(G,H,T\right)$-extension $L/\mathbb{Q}$
such that $L^{g}=L'$ with associated isomorphism $\phi$ satisfying
$\overline{\phi}=\phi'$ then there are exactly $2^{\omega\left(d\right)-n}$
such extensions.
\end{prop}

\begin{proof}
If $L=K(\sqrt{d_{1}'},\ldots\sqrt{d_{n}'})\left(\sqrt{\mu}\right)$
is a $\left(G,H,T\right)$-extension with associated isomorphism $\phi'$
then by Lemma \ref{lem:Suppose--and} every other such extension is
of the form $L^{g}\left(\sqrt{\delta\mu}\right)$ for some fundamental
discriminant $\delta\mid d$. Define the $\mathbb{F}_{2}$ vector
space $V=\left\langle q_{1},\ldots,q_{\omega\left(d\right)}\right\rangle /\left\langle q_{1}^{2},\ldots,q_{\omega\left(d\right)}^{2}\right\rangle $
where the $q_{i}$ are the divisors of $d$ which are prime fundamental
discriminants or one of $\left\{ -4,\pm8\right\} $. Then the number
of such extensions which are distinct is the size of the vector space
$V/\left\langle d_{1}',\ldots,d_{n}'\right\rangle $, which is clearly
$2^{\omega\left(d\right)-n}$ (recall the $d_{i}'$ are all independent
mod $\mathbb{Q}^{*2}$).

\end{proof}
For $a,b\in\mathbb{Q}$ let $\left(a,b\right)$ denote the cyclic
algebra in $\mathrm{Br}\left(\mathbb{Q}\right)$ defined by 
\[
\left(a,b\right)=K\left\langle u,v\right\rangle /\left\langle u^{2}=a,v^{2}=b,uv=-vu\right\rangle .
\]
 Recall that there is an injection 
\[
inv:\mathrm{Br}\left(\mathbb{Q}\right)\hookrightarrow\prod_{v}\mathrm{Br}\left(\mathbb{Q}_{v}\right)\cong\prod_{v}\mathbb{Q}/\mathbb{Z}.
\]
Since $\left(a,b\right)$ has order 2 we can view each factor of the
image of $inv$ as lying in $\frac{1}{2}\mathbb{Z}/\mathbb{Z}$. Identify
this group with $\left\langle \pm1\right\rangle $. Denote the composition
of $inv$ with the projection onto the factor corresponding to $v$
by $inv_{v}$. Denote by $\left(a,b\right)_{v}$ the quadratic Hilbert
symbol over the field $\mathbb{Q}_{v}$. Then the crucial property
we will need is 
\[
inv_{v}\left(\left(a,b\right)\right)=\left(a,b\right)_{v}.
\]

We will use the following theorem, stated as Theorem 1.2 in \cite{cfields}.
\begin{thm}[Embedding Criterion]
\label{thm:-Let-,} Let $L'=F(\sqrt{d_{1}'},\ldots,\sqrt{d_{n}'})$,
where the $d_{i}'$ are elements of $F^{*}$ independent modulo $F^{*2}$.
Let $G=\mathrm{Gal}\left(L'/F\right)$, and consider a non-split central
extension $G^{*}$ of $\mathbb{F}_{2}$ by $G$. Let $e_{1},\ldots,e_{n}$
generate $G$, where $e_{i}\sqrt{d_{j}'}=\left(-1\right)^{\delta\left(i,j\right)}\sqrt{d_{j}'}$
and let $x_{1},\ldots,x_{n}$ be any set of preimages of $e_{1},\ldots,e_{n}$
in $G^{*}$. Define $c_{ij}=1$ if $\left[x_{i},x_{j}\right]\neq1$
and 0 otherwise for $i\neq j$ and $c_{ii}=1$ if $x_{i}^{2}\neq1$.

Then there exists a Galois extension $L/F$ with $L'\subset L$ such
that $\mathrm{Gal}\left(L/F\right)\cong G^{*}$ and the surjection
$G^{*}\longrightarrow G$ is the natural surjection of Galois groups,
if and only if 
\[
\prod_{i\le j}\left(d_{i}',d_{j}'\right)^{c_{ij}}=1\in\mathrm{Br}\left(F\right).
\]
\end{thm}

\subsection{Extensions of a fixed quadratic field}

For the next proposition we define some notation. Let $\overline{H}$
be the projection of $H$ to $G/\Phi\left(G\right)$. Let $\mathcal{K}_{T}$
be the set of pairs $\left(L',\phi'\right)$ where $L'/K$ is a subfield
of $K^{gen}$ with isomorphism $\phi':\mathrm{Gal}\left(L'/\mathbb{Q}\right)\cong G/\Phi\left(G\right)$
such that $\phi'\left(U\right)=T$ and $\phi'\left(\mathrm{Gal}\left(L'/K\right)\right)=\overline{H}$.
Let $\mathcal{D}$ be the set of $r$-tuples $\left(d_{1},\ldots,d_{r}\right)$
of coprime fundamental discriminants with each $d_{i}\neq1$ and such
that $d=\prod_{i=1}^{r}d_{i}$. Define $U_{i}\subset\left[n\right]$
by the expression $t_{i}=\sum_{j\in U_{i}}e_{j}$ for each $i=1,\ldots,r$.
\begin{prop}
\label{prop:There-is-a}There is a bijective correspondence between
$\mathcal{K}_{T}$ and $\mathcal{D}$.
\end{prop}

\begin{proof}
Suppose $\left(L',\phi'\right)\in\mathcal{K}_{T}$. Define $d_{i}$
to be the product of all prime fundamental discriminants $p$ (if
$p=2$ use $\pm2^{\mathrm{ord}_{2}d}$) such that $\phi'\left(I_{\mathfrak{p}}\right)=t_{i}$
for some (equivalently all) $\mathfrak{p}\mid p$ in $L'$. Then all
the $d_{i}$ are coprime and $d=\prod_{i=1}^{r}d_{i}$ since each
$t_{i}$ is in the image of some inertia group by definition.

Conversely given a factorization $d=\prod_{i=1}^{r}d_{i}$ in $\mathcal{D}$
define $d_{j}'=\prod_{i:j\in U_{i}}d_{i}$ for each $j=1,\ldots,n$.
Let $L'=\mathbb{Q}(\sqrt{d_{1}'},\ldots,\sqrt{d_{n}'})$. For each
$j=1,\ldots,n$ let $\phi'_{j}:\mathrm{Gal}\left(L'/\mathbb{Q}\right)\longrightarrow\mathrm{Gal}\left(\mathbb{Q}\left(\sqrt{d_{j}'}\right)\right)$
be the projection. Then $\phi'=\left(\phi'_{1},\ldots,\phi'_{n}\right)$
is a homomorphism $\phi':\mathrm{Gal}\left(L'/\mathbb{Q}\right)\longrightarrow G/\Phi\left(G\right)$
where $G/\Phi\left(G\right)$ has basis $\left\{ e_{1},\ldots,e_{n}\right\} $.

We will show $\phi'$ is an isomorphism. The relation $d_{j}'=\prod_{i:j\in U_{i}}d_{i}$
implies that if $p\mid d_{i}$ then $p\mid d_{j}'$ exactly when $j\in U_{i}$.
This implies that for any prime $\mathfrak{p}\mid p$ in $L'$ if
$I_{\mathfrak{p}}$ is the inertia group then $\phi'_{j}\left(I_{\mathfrak{p}}\right)=1$
exactly when $j\in U_{i}$. Thus $\phi'\left(I_{\mathfrak{p}}\right)=t_{i}$.
Since the $\left\{ t_{1},\ldots,t_{n}\right\} $ generate $G/\Phi\left(G\right)$
this implies $\phi'$ is surjective. Since $\mathrm{Gal}\left(L'/\mathbb{Q}\right)$
is a quotient of $\mathbb{F}_{2}^{n}$ this implies $\phi'$ is an
isomorphism.

Pick a basis $\left\{ e_{1}',\ldots,e_{n}'\right\} $ for $\mathrm{Gal}\left(L'/\mathbb{Q}\right)$
such that $e_{i}'\sqrt{d_{j}'}=\left(-1\right)^{\delta_{ij}}\sqrt{d_{j}'}$.
Then clearly $\phi'\left(e_{j}'\right)=e_{j}$. Let $t_{i}'=\sum_{j\in U_{i}}e_{i}'$.
It follows from the above argument that $\phi'\left(U\right)=T$.
Each $t_{i}'$ projects nontrivially to $\mathrm{Gal}\left(K/\mathbb{Q}\right)=\mathbb{F}_{2}$
so $\mathrm{Gal}\left(L'/K\right)$ is equal to the trace 0 subspace
of $\mathrm{Gal}\left(L'/\mathbb{Q}\right)$ with respect to the generating
set $\left\{ t_{1}',\ldots,t_{n}'\right\} $. The same is true for
$\overline{H}\subset G/\Phi\left(G\right)$ with $\left\{ t_{1},\ldots,t_{n}\right\} $.
Thus $\phi'\left(\mathrm{Gal}\left(L'/K\right)\right)=\overline{H}$.

Next we will show the above maps are inverses of each other.

Let $\left(L',\phi'\right)\in\mathcal{K}_{T}$. We can write $L'=\mathbb{Q}(\sqrt{d_{1}'},\ldots,\sqrt{d_{n}'})$
where $\mathbb{Q}(\sqrt{d_{j}'})$ is the subfield corresponding to
$\phi'^{-1}\left(e_{j}\right)$. Let $d=\prod_{i=1}^{r}d_{i}$ be
the corresponding factorization.

We claim that $d_{j}'=\prod_{i,j\in U_{i}}d_{i}$ for each $j=1,\ldots,n$.
Fix $p\mid d$. Then $p\mid d_{j}'$ if and only if $I_{\mathfrak{p}}$
projects non-trivially to $\mathbb{Q}(\sqrt{d_{j}'})$. This is equivalent
to $\phi'\left(I_{\mathfrak{p}}\right)=t_{i}$ for some $i$ satisfying
$j\in U_{i}$ which by definition of the correspondence is equivalent
to $p\mid d_{i}$ for some $i$ satisfying $j\in U_{i}$. Furthermore
it is clear that $\phi'=\left(\phi'_{1},\ldots,\phi'_{n}\right)$
as defined above. This proves one direction.

Let $d=\prod_{i=1}^{r}d_{i}$ in $\mathcal{D}$. It is shown above
that if $p\mid d_{i}$ then $\phi'\left(I_{\mathfrak{p}}\right)=t_{i}$
for any $\mathfrak{p}\mid p$ in $L'$ where $\left(L',\phi'\right)$
is the corresponding pair. Thus the maps are inverse to each other,
which completes the proof.
\end{proof}
Recall $\left\{ e_{1},\ldots,e_{n}\right\} $ is a fixed basis for
$G/\Phi\left(G\right)$. For each $i=1,\ldots,n$ let $x_{i}'\in G$
be any lift of $e_{i}$. Let $S_{i}'=\left\{ 1\le j\le n\mid j\neq i,\left[x_{i}',x_{j}'\right]\neq1\right\} \cup\left\{ i\mid\mathrm{ord}\left(x_{i}'\right)=4\right\} $.
\begin{lem}
\label{lem:asd}Let $\left(L',\phi'\right)\in\mathcal{K}_{T}$ with
$L'=K(\sqrt{d'_{1}},\ldots\sqrt{d'_{n}})$. There exists a $\left(G,H,T\right)$-extension
$L/L'/\mathbb{Q}$ with isomorphism $\phi$ such that $\phi'^{-1}\circ\overline{\phi}$
is the natural surjection of Galois groups if and only if 
\[
\prod_{p\mid d}\left(1+\left(-1\right)^{\chi\left(p\right)}\left(\frac{\prod_{j\le i,j\in S_{i}'}\left(d'_{j}/p^{\alpha_{j}\left(p\right)}\right)^{\alpha_{i}\left(p\right)}\left(d'_{i}/p^{\alpha_{i}\left(p\right)}\right)^{\alpha_{j}\left(p\right)}}{p}\right)\right)\neq0
\]
 where 
\[
\chi\left(p\right)=\begin{cases}
\frac{\left(p-1\right)}{2}\sum_{j\le i,j\in S_{i}'}\alpha_{i,j}\left(p\right) & p\text{ odd}\\
\sum_{j\le i,j\in S_{i}'}\frac{\left(d_{j}'/2^{\alpha_{j}\left(2\right)}\right)-1}{2}\frac{\left(d_{i}'/2^{\alpha_{i}\left(2\right)}\right)-1}{2} & p=2.
\end{cases}
\]
\end{lem}

\begin{proof}
We identify $\mathrm{Gal}\left(L'/\mathbb{Q}\right)$ with $G/\Phi\left(G\right)$
using $\phi'$. Then by Theorem \ref{thm:-Let-,} such an extension
exists if and only if $\prod_{j\le i}\left(d_{i}',d_{j}'\right)^{c_{ij}}=1\in\mathrm{Br}\left(\mathbb{Q}\right)$
(it is easy to show that when $F=\mathbb{Q}$ and $L'/K$ is unramified
at the finite places, the extension $L$ given by Theorem \ref{thm:-Let-,}
is always also unramified above $K$ at the finite places). Clearly
$c_{ij}=1$ if and only if $j\in S_{i}'$.

By the discussion preceding the theorem, since $inv_{v}$ is an injection
such an extension exists if and only if

\begin{align*}
1 & =inv_{v}\left(\prod_{j\le i,j\in S_{i}'}\left(d_{i}',d_{j}'\right)\right)\\
 & =\prod_{j\le i,j\in S_{i}'}\left(inv_{v}\left(d_{i}',d_{j}'\right)\right)\\
 & =\prod_{j\le i,j\in S_{i}'}\left(d_{i}',d_{j}'\right)_{v}.
\end{align*}
 at all places $v$ of $\mathbb{Q}$. By the product formula for the
Hilbert symbol this condition only needs to be checked at the finite
primes. This is trivially satisfied for $v\nmid d$.

To simplify notation let $\alpha_{i}\left(p\right)=\mathrm{ord}_{p}d_{i}'$
and $\alpha_{i,j}\left(p\right)=\alpha_{i}\left(p\right)\alpha_{j}\left(p\right)$.
Note the property of the quadratic Hilbert symbol $\left(d'_{i},d'_{j}\right)_{v}=\left(d'_{j},d'_{i}\right)_{v}$.
Then the condition at odd $p$ becomes
\[
\left(-1\right)^{\frac{\left(p-1\right)}{2}\sum_{j\le i,j\in S_{i}'}\alpha_{i,j}\left(p\right)}\left(\frac{\prod_{j\le i,j\in S_{i}'}\left(d'_{j}/p^{\alpha_{j}\left(p\right)}\right)^{\alpha_{i}\left(p\right)}\left(d'_{i}/p^{\alpha_{i}\left(p\right)}\right)^{\alpha_{j}\left(p\right)}}{p}\right)=1
\]
 The condition at $p=2$ becomes
\[
-1^{\sum_{j\le i,j\in S_{i}'}\frac{\left(d_{j}'/2^{\alpha_{j}\left(2\right)}\right)-1}{2}\frac{\left(d_{i}'/2^{\alpha_{i}\left(2\right)}\right)-1}{2}+\alpha_{i}\left(2\right)\frac{\left(d_{j}'/2^{\alpha_{j}\left(2\right)}\right)^{2}-1}{8}+\alpha_{j}\left(2\right)\frac{\left(d_{i}'/2^{\alpha_{i}\left(2\right)}\right)^{2}-1}{8}}=1
\]
 which can be rewritten as 
\[
\left(-1\right)^{\sum_{j\le i,j\in S_{i}'}\frac{\left(d_{j}'/2^{\alpha_{j}\left(2\right)}\right)-1}{2}\frac{\left(d_{i}'/2^{\alpha_{i}\left(2\right)}\right)-1}{2}}\left(\frac{\prod_{j\le i,j\in S_{i}'}\left(d'_{j}/p^{\alpha_{j}\left(2\right)}\right)^{\alpha_{i}\left(2\right)}\left(d'_{i}/p^{\alpha_{i}\left(2\right)}\right)^{\alpha_{j}\left(2\right)}}{2}\right).
\]
The result follows.
\end{proof}
For each $i=1,\ldots,r$ let $x_{i}\in G$ be any lift of $t_{i}$.
Let $S_{i}=\left\{ 1\le j\le r\mid\left[x_{i},x_{j}\right]\neq1\right\} $.
\begin{lem}
\label{lem:extension_exists_above_di}Let $\left(L',\phi'\right)\in\mathcal{K}_{T}$
and $d=\prod_{i=1}^{r}d_{i}\in\mathcal{D}$ the corresponding factorization
(see Proposition \ref{prop:There-is-a}). There exists a $\left(G,H,T\right)$-extension
$L/\mathbb{Q}$ with isomorphism $\phi$ such that $L'\subset L$
and $\phi'^{-1}\circ\overline{\phi}$ is the natural surjection of
Galois groups if and only if 
\begin{equation}
\prod_{i=1}^{r}\prod_{p\mid d_{i}}\left(1+\left(\frac{\prod_{j\in S_{i}}d_{j}}{p}\right)^{\mathrm{ord}_{p}\left(d\right)}\right)\ne0.\label{eq:extension_exists_above_di}
\end{equation}
\end{lem}

\begin{proof}
By Lemma \ref{lem:asd} there exists a $\left(G,H,T\right)$-extension
$L/\mathbb{Q}$ with isomorphism $\phi$ such that $L'\subset L$
and $\phi'^{-1}\circ\overline{\phi}$ is the natural surjection of
Galois groups if and only if 
\begin{equation}
\prod_{p\mid d}\left(1+\left(-1\right)^{\chi\left(p\right)}\left(\frac{\prod_{j\le i,j\in S_{i}'}\left(d'_{j}/p^{\alpha_{j}\left(p\right)}\right)^{\alpha_{i}\left(p\right)}\left(d'_{i}/p^{\alpha_{i}\left(p\right)}\right)^{\alpha_{j}\left(p\right)}}{p}\right)\right)\neq0\label{eq:firstcondition}
\end{equation}
 where 
\[
\chi\left(p\right)=\begin{cases}
\frac{\left(p-1\right)}{2}\sum_{j\le i,j\in S_{i}'}\alpha_{i,j}\left(p\right) & p\text{ odd}\\
\sum_{j\le i,j\in S_{i}'}\frac{\left(d_{j}'/2^{\alpha_{j}\left(2\right)}\right)-1}{2}\frac{\left(d_{i}'/2^{\alpha_{i}\left(2\right)}\right)-1}{2} & p=2.
\end{cases}
\]

Since each $x_{i}$ has order 2 this implies that the size of the
set 
\[
\left\{ \left(j_{1},j_{2}\right)\mid1\le j_{1}\le n,j_{2}\le j_{1},j_{2}\in S_{j_{1}}',j_{1},j_{2}\in U_{i}\right\} 
\]
 is even for all $1\le i\le r$. Note that for $1\le j\le n$ and
any $p$, $\alpha_{j}\left(p\right)\ge1$ if and only if $p\mid d_{i}$
and $j\in U_{i}$ for some $1\le i\le r$.

Hence for $p$ odd
\begin{align*}
\chi\left(p\right) & =\frac{\left(p-1\right)}{2}\sum_{j_{1}\le j_{2},j_{1}\in S_{j_{2}}';j_{1},j_{2}\in U_{i}}1\\
 & =\frac{\left(p-1\right)}{2}\left|\left\{ \left(j_{1},j_{2}\right)\mid1\le j_{1}\le n,j_{2}\le j_{1},j_{2}\in S_{j_{1}}',j_{1},j_{2}\in U_{i}\right\} \right|
\end{align*}
 for some $i$. Thus $\chi\left(p\right)$ is even for $p$ odd. Since
every odd fundamental discriminant is congruent to 1 mod4, $\frac{\left(d_{i}'/2^{\alpha_{i}\left(2\right)}\right)-1}{2}$
is odd exactly when $\alpha_{i}\left(2\right)\ge1$. This similarly
implies that $\chi\left(2\right)$ is even.

Now we partition the product over $p\mid d$ in $\left(\ref{eq:firstcondition}\right)$by
the $d_{i}$. Fix $1\le i\le r$ and $p\mid d_{i}$. Recall this
implies $\alpha_{j}\left(p\right)\ge1$ exactly when $j\in U_{i}$.
Fix $1\le j\le n$. Then 
\begin{align*}
\prod_{k\in S_{j}'}\left(d'_{j}/p^{\alpha_{j}\left(p\right)}\right)^{\alpha_{k}\left(p\right)}\left(d'_{k}/p^{\alpha_{k}\left(p\right)}\right)^{\alpha_{j}\left(p\right)} & =\prod_{k\in S_{j}'\cap U_{i}}\left(d'_{j}/p^{\alpha_{j}\left(p\right)}\right)^{\alpha_{k}\left(p\right)}\left(d'_{k}/p^{\alpha_{k}\left(p\right)}\right)^{\alpha_{j}\left(p\right)}\\
 & \times\prod_{k\in S_{j}'\backslash U_{i}}\left(d'_{k}\right)^{\alpha_{j}\left(p\right)}\\
 & =\left(d_{j}'/p^{\alpha_{j}\left(p\right)}\right)^{\mathrm{ord}_{p}\left(d\right)\left|S_{j}'\cap U_{i}\right|}\cdot a
\end{align*}
for some $a$. Hence considering the contribution from each $j$ we
see that for any $1\le l\le r$ with $l\neq i$ the power of $d_{l}$
in the factor corresponding to $d_{i}$ in $\left(\ref{eq:firstcondition}\right)$
will be $\mathrm{ord}_{p}\left(d\right)\sum_{j\in U_{l}}\left|S_{j}'\cap U_{i}\right|$
and $\sum_{j\in U_{l}}\left|S_{j}'\cap U_{i}\right|$ is odd exactly
when $x_{l}$ and $x_{i}$ do not commute. Note division by $p$ only
occurs in $\left(\ref{eq:firstcondition}\right)$ when $\alpha_{j,k}\left(p\right)\ge1$
hence the numerator of the legendre symbol will in fact be a product
of fundamental discriminants (negative signs remaining after any division
operation cancel). Thus this factor is
\[
\prod_{p\mid d_{i}}\left(1+\left(\frac{\prod_{j\in S_{i}}d_{j}}{p}\right)^{\mathrm{ord}_{p}\left(d\right)}\right).
\]
\end{proof}
Let $\mathrm{Aut}_{H,T}(G/\Phi\left(G\right))$ be the subgroup of
$\mathrm{Aut}(G/\Phi\left(G\right))$ preserving the set $\overline{H}$
and the set $T$. Let $\mathrm{Aut}_{H,T}(G,\Phi\left(G\right))$
be the subgroup of $\mathrm{Aut}_{H,T}(G/\Phi\left(G\right))$ consisting
of elements which lift to $\mathrm{Aut}\left(G\right)$.
\begin{thm}
\label{thm:f_T_formula}Suppose $\left(G,H,T\right)$ is admissible.
Then
\begin{align}
f_{T}\left(d\right) & =\frac{1}{2^{n}\left|\mathrm{Aut}_{H,T}(G,\Phi\left(G\right))\right|}\sum_{\left(d_{1},\ldots,d_{r}\right)\in\mathcal{D}}\prod_{i=1}^{r}\prod_{p\mid d_{i}}\left(1+\left(\frac{\prod_{j\in S_{i}}d_{j}}{p}\right)^{\mathrm{ord}_{p}\left(d\right)}\right).\label{eq:formulathmstatement}
\end{align}
\end{thm}

\begin{proof}
Let $L'=\mathbb{Q}(\sqrt{d_{1}'},\ldots,\sqrt{d_{n}'})$ and suppose
$\left(L',\phi'\right)\in\mathcal{K}_{T}$ with corresponding tuple
$\left(d_{1},\ldots,d_{r}\right)\in\mathcal{D}$ given by Proposition
\ref{prop:There-is-a}.  By Lemma \ref{lem:extension_exists_above_di}
there exists a $\left(G,H,T\right)$-extension $L/\mathbb{Q}$ with
$L^{g}=L'$ and associated isomorphism $\phi$ such that $\phi'^{-1}\circ\overline{\phi}$
is the natural surjection of Galois groups (equivalently $\phi'=\overline{\phi}$
) if and only if 
\begin{equation}
\prod_{i=1}^{r}\prod_{p\mid d_{i}}\left(1+\left(\frac{\prod_{j\in S_{i}}d_{j}}{p}\right)^{\mathrm{ord}_{p}\left(d\right)}\right)\ne0.\label{eq:formulaproof}
\end{equation}

This function takes on the values 0 or $2^{\omega\left(d\right)}$.
By Proposition \ref{prop:If-there-exists} if a $\left(G,H,T\right)$-extension
$L/\mathbb{Q}$ with $L^{g}=L'$ and associated isomorphism $\phi$
such that $\phi'=\overline{\phi}$ exists, there are exactly $2^{\omega\left(d\right)-n}$
such extensions. Hence the formula $\left(\ref{eq:formulaproof}\right)$
divided by $2^{n}$ gives the number of $\left(G,H,T\right)$-extensions
$L/\mathbb{Q}$ with $L^{g}=L'$ and associated isomorphism $\phi$
such that $\phi'=\overline{\phi}$.

For fixed $L'$ the group $\mathrm{Aut}_{H,T}(G/\Phi\left(G\right))$
acts freely transitively on the set $I=\left\{ \left(L',\phi'\right)\in\mathcal{K}_{T}\right\} $
by composition with $\phi'$. The set of $\left(G,H,T\right)$-extensions
$L/\mathbb{Q}$ with $L^{g}=L'$ and associated isomorphism $\phi$
such that $\phi'=\overline{\phi}$ is equal for every element in the
orbit of $\phi'$ under $\mathrm{Aut}_{H,T}(G,\Phi\left(G\right))$.
Hence summing over $\mathcal{D}$ overcounts by a factor of $\left|\mathrm{Aut}_{H,T}(G,\Phi\left(G\right))\right|$.
This completes the proof.
\end{proof}
\begin{rem}
We will eventually look at the growth of the function \ref{eq:formulathmstatement}
as $d$ ranges over all fundamental discriminants. We can make some
heuristic considerations to see what kind of behaviour to expect.
If we expect the legendre symbol to take each value $\pm1$ with probability
$1/2$ then the expected value of $\prod_{i,S_{i}\neq\varnothing}\prod_{p\mid d_{i}}\left(1+\left(\frac{\prod_{j\in S_{i}}d_{j}}{p}\right)\right)$
is 
\[
\frac{1}{2^{\omega\left(\prod_{i,S_{i}\neq\varnothing}d_{i}\right)}}2^{\omega\left(\prod_{i,S_{i}\neq\varnothing}d_{i}\right)}+\left(1-\frac{1}{2^{\omega\left(\prod_{i,S_{i}\neq\varnothing}d_{i}\right)}}\right)0=1.
\]
Let $r_{1}=\left|\left\{ i\mid S_{i}\neq\varnothing\right\} \right|$
and $r_{2}=r-r_{1}$. Hence on average we expect 
\begin{align*}
2^{n}\left|\mathrm{Aut}_{H,T}(G,\Phi\left(G\right))\right|f_{T}\left(d\right) & =\sum_{d=d_{1}\cdots d_{r}}2^{\omega\left(\prod_{i,S_{i}=\varnothing}d_{i}\right)}\\
 & =\sum_{i=0}^{\omega\left(d\right)}2^{i}\binom{\omega\left(d\right)}{i}r_{2}^{i}r_{1}^{\omega\left(d\right)-i}\\
 & =r_{1}^{\omega\left(d\right)}\sum_{i=0}^{\omega\left(d\right)}\left(\frac{2r_{2}}{r_{1}}\right)^{i}\binom{\omega\left(d\right)}{i}\\
 & =r_{1}^{\omega\left(d\right)}\left(1+\frac{2r_{2}}{r_{1}}\right)^{\omega\left(d\right)}\\
 & =\left(r_{1}+2r_{2}\right)^{\omega\left(d\right)}.
\end{align*}
Thus we expect $\sum_{d<X}f_{T}\left(d\right)\sim C\left(G,H,T\right)X\left(\log X\right)^{r_{1}+2r_{2}-1}$
for some constant $C\left(G,H,T\right)$. We will show that this is
indeed the case.

We can interpret the quantity $r_{1}+2r_{2}$ as the number of conjugacy
classes of lifts of $T$ in $G$. Recall $a$ is the distinguished
central element of order 2 in $G$. If $S_{i}=\varnothing$ then $x_{i}$
commutes with all the $x_{j}$. Hence $x_{i}\in Z\left(G\right)$
and there are the two conjugacy classes $\left\{ x_{i}\right\} $
and $\left\{ x_{i}a\right\} $. If $S_{i}\neq\varnothing$ then there
is the conjugacy class $\left\{ x_{i},x_{i}a\right\} $.
\end{rem}

\section{\label{sec:isotropic subsets quad forms}Isotropic subsets for quadratic
forms}

In this section we prove a result about maximal isotropic subsets
of quadratic forms in characteristic 2. We will use this result to
deduce statements about maximal disconnected subgraphs of the graph
$\mathcal{G}^{*}\left(T\right)$ on the generating set $T$ (see \ref{sec:The graph})
by encoding the graph structure using a particular quadratic form
defined in the next section.

We collect some facts regarding quadratic forms in characteristic
2. Let $Q:\mathbb{F}_{2}^{n}\longrightarrow\mathbb{F}_{2}$ be a quadratic
form and let $B\left(u,v\right)=Q\left(u+v\right)+Q\left(u\right)+Q\left(v\right)$
be the corresponding bilinear form.

For any subspace $W\subset\mathbb{F}_{2}^{n}$ we define $\mathrm{rad}\left(W\right)=W\cap W^{\perp}$
(where $W^{\perp}$ is taken with respect to $B$). We can decompose
$\mathbb{F}_{2}^{n}=V\oplus\mathrm{rad}(\mathbb{F}_{2}^{n})$, and
$\mathrm{rad}(\mathbb{F}_{2}^{n})=R\oplus R_{0}$ such that $Q\left(R_{0}\right)=0$.
We say $Q$ is non-degenerate if $\mathrm{rad}(\mathbb{F}_{2}^{n})=0$.
Hence in the above notation $Q$ is non-degenerate on $V$ and it
is easy to see that $\dim R=0$ or 1.

We say any $v\in\mathbb{F}_{2}^{n}$ is singular if $v\neq0$ and
$\mathbb{F}_{2}^{n}$. We say a subspace $W\subset\mathbb{F}_{2}^{n}$
is totally singular if $Q\left(v\right)=0$ for all $v\in W$.

The following result is Proposition 14.6 from \cite{grovesbook}.
\begin{lem}
\label{lem:if--is}If $Q$ is nondegenerate, given any totally singular
space there exists another one disjoint from it.
\end{lem}

We extend this result to the case when $Q$ is degenerate and $\ker Q\mid_{\mathbb{F}_{2}^{n\perp}}=0$.
In the above notation this means $\mathbb{F}_{2}^{n}=V\oplus R$.
Let $p_{V},p_{R}$ be the projections onto $V$ and $R$.
\begin{prop}
\label{lem:Let--be}Let $Q$ be a quadratic form on $\mathbb{F}_{2}^{n}$
and let $R=\mathrm{rad}(\mathbb{F}_{2}^{n})$ be non-trivial. Suppose
$\ker(Q\mid_{R})=0$. Given any totally singular subspace $W$ there
exists another totally singular subspace $W'$ of the same size such
that $W\cap W'\subset\left\langle u\right\rangle $ for some $u\in\mathbb{F}_{2}^{n}$
with $p_{R}\left(u\right)\neq0$.
\end{prop}

\begin{proof}
As noted above we have $\mathbb{F}_{2}^{n}=V\oplus R$. Note the subspace
$W_{0}=W\cap V$ has index $\le2$ in $W$.

We can apply Proposition 14.6 from \cite{grovesbook} to $W_{0}$
to obtain a totally singular subspace $W_{1}\subset V$ of the same
size such that $W_{0}\cap W_{1}=0$. Furthermore $W_{0}$ has basis
$\left\{ u_{1},\ldots,u_{k}\right\} $ and $W_{1}$ has basis $\left\{ v_{1},\ldots,v_{k}\right\} $
and $H_{i}=\left\langle u_{i},v_{i}\right\rangle $ is hyperbolic,
meaning $B\left(u_{i},v_{i}\right)=1$, and $B\left(H_{i},H_{j}\right)=0$
for $i\neq j$.

If $W_{0}=W$ then we let $W'=W_{1}$. Then $W'$ satisfies the desired
properties and we are done.  Otherwise let $u\in W\backslash W_{0}$.
Since $Q$ is non-degenerate on $V$ we can write $V=W_{0}\oplus W_{1}\oplus W_{2}$
with $W_{2}=\left(W_{0}\oplus W_{1}\right)^{\perp}$ (taken in $V$).

Write $p_{V}\left(u\right)$ in this basis and let $v$ be the element
obtained by applying to $p_{V}\left(u\right)$ the transposition exchanging
$u_{i}$ and $v_{i}$ for all $i$ (note that $u=v$ is possible).

First we will show $Q\left(v\right)=0$. Write $u=w_{0}+w_{1}+w_{2}+r$
with $w_{i}\in W_{i}$ and $r\in R$. Recall $B\left(u,v\right)=Q\left(u+v\right)+Q\left(u\right)+Q\left(v\right)$.
We have 
\begin{align*}
Q\left(u\right) & =Q\left(w_{0}+w_{1}+w_{2}+r\right)\\
 & =Q\left(w_{0}+w_{1}+w_{2}\right)+Q\left(r\right)\\
 & =Q\left(w_{0}+w_{1}\right)+Q\left(w_{2}\right)+Q\left(r\right)\\
 & =B\left(w_{0},w_{1}\right)+Q\left(w_{2}\right)+Q\left(r\right)
\end{align*}
where we are using $B\left(w_{0}+w_{1}+w_{2},r\right)=0$, $B\left(w_{0}+w_{1},w_{2}\right)=0$
and the fact that $W_{0}$ and $W_{1}$ are totally singular. This
implies in particular that $Q\left(u+w_{0}'\right)=B\left(w_{0}+w_{0}',w_{1}\right)+Q\left(w_{2}\right)+Q\left(r\right)$
for any $w_{0}'\in W_{0}$. Since $Q\left(u+w_{0}'\right)=0$ for
all $w_{0}'\in W_{0}$, by the properties of $B$ on hyperbolic subspaces
listed above this implies $w_{1}=0$. Thus $Q\left(u\right)=Q\left(w_{2}\right)+Q\left(r\right)$.
By symmetry also $Q\left(v\right)=0$.

Then also by symmetry $v\in W_{1}^{\perp}$ (taken in $\mathbb{F}_{2}^{n}$)
and hence $W'=\left\langle W_{1},v\right\rangle $ is totally singular.
Clearly $p_{R}\left(v\right)\neq0$ and $\left|W'\right|=\left|W\right|$.

It remains to show $W\cap W'\subset\left\langle u\right\rangle $.
Suppose $w_{0}+u=w_{1}+v$ with $w_{i}\in W_{i}$. As shown above
$u$ is not supported on $W_{1}$ and $v$ is not supported on $W_{0}$.
Thus we have equality if and only if $u=v$ and $w_{i}=0$. This completes
the proof.
\end{proof}
We now prove some facts about maximal isotropic subsets of certain
generating sets of $\mathbb{F}_{2}^{n}$. We start with the case when
$\ker(Q\mid_{\mathbb{F}_{2}^{n\perp}})=0$.
\begin{lem}
\label{lem:perp_sets_bound1}Suppose $\ker(Q\mid_{\mathbb{F}_{2}^{n\perp}})=0$.
Let $H\subset\mathbb{F}_{2}^{n}$ be a codimension 2 subspace. Let
$S=(\ker Q)\backslash H$ and suppose $S$ generates $\mathbb{F}_{2}^{n}$.
Then for any subset $T\subset S$ 
\[
\left|T^{\perp}\cap T\right|\le\left|S\right|-\left|T\right|.
\]
\end{lem}

\begin{proof}
Decompose $\mathbb{F}_{2}^{n}=V\oplus R$, where $R=\mathbb{F}_{2}^{n\perp}$
and $\ker(Q\mid_{R})=0$. To simplify notation let $\widehat{T}=T^{\perp}\cap T$
and $r=\left|S\right|$. Suppose $T=\left\{ w_{1},\ldots,w_{r'}\right\} \subset S$
and $\left\{ w_{1},\ldots,w_{m}\right\} =\widehat{T}$. We want to
show $m\le r-r'$. Assume $m\ge1$ otherwise this is trivial.

By definition $B\left(w_{i},w_{j}\right)=0$ and $Q\left(w_{i}\right)=0$
for all $w_{i},w_{j}\in\widehat{T}$. The set $\widehat{T}$ can be
extended to a maximal totally singular subspace $W\subset\mathbb{F}_{2}^{n}$
containing $\widehat{T}$.

If $\dim(V\oplus R)$ is even then $R=0$. If $\dim(V\oplus R)$ is
odd then $R=\left\langle z\right\rangle $ and $Q\left(z\right)=1$.
Let $p_{R}$ be the projection onto $R$.  By Lemma \ref{lem:Let--be}
there exists a totally singular subspace $W'\subset Y$ of the same
size as $W$ such that $W\cap W'\subset\left\langle w\right\rangle $
for some $w$ with $p_{R}\left(w\right)\neq0$.

First suppose $w\notin\widehat{T}$, so that $W\cap W'=\left\{ 0\right\} $.
Let $W_{0}\subset W'$ be the set of all elements $w_{0}$ such that
$B\left(w_{i},w_{0}\right)=1$ for some $i=1,\ldots,m$. Suppose
towards a contradiction that $|W_{0}\cap S|<m$. Then since $\hat{T}\cap W'=\varnothing$
we see that $W_{1}=\left\langle \left(W'\backslash W_{0}\right)\cup\widehat{T}\right\rangle $
is a totally singular subspace containing $\widehat{T}$ which has
strictly more elements not in $H$ than $W'$. Since $\dim(W_{1}/(W_{1}\cap H))=1$,
that is half the elements in $W_{1}$ are in $H$, this implies $W_{1}$
is a totally singular space containing $\widehat{T}$ which is strictly
larger than $W'$ and hence $W$. Thus it must be that $|W_{0}\cap S|\ge m$.

If $w\in\widehat{T}$ then by the same argument we get $|W_{0}\cap S|\ge m-1$.
Note $B\left(w,w_{i}\right)=0$ for $i=1,\ldots,r'$. Since $\ker(Q\mid_{R})=0$
it follows that $w\notin R$. If $B\left(w,u\right)=0$ for all $u\in S$
then this would imply $w\in R$ since $S$ is a generating set for
$Y$. Thus there exists some $w'\in S$ such that $B\left(w,w'\right)=1$.
Since $w\in W'$ this implies $w'\notin W'$. Thus $\left|(W_{0}\cup\left\{ w'\right\} )\cap S\right|\ge m$.

From the definition of $W_{0}$ and $w'$ it follows that $\left((W_{0}\cup\left\{ w'\right\} )\cap S\right)\cap\left\{ w_{1},\ldots,w_{r'}\right\} =\varnothing$.
Thus $m\le r-r'$ as desired.
\end{proof}
\begin{rem}
\label{rem:A-simpler-version}A simpler version of the above proof
gives the same result when $S=(\ker Q)\backslash\left\{ 0\right\} $
and $S$ generates $\mathbb{F}_{2}^{n}$. We will require this fact
in the proof of the next proposition.
\end{rem}

\begin{prop}
\label{prop:perp_sets_bound2}Let $H\subset\mathbb{F}_{2}^{n}$ be
a codimension 2 subspace. Let $S=(\ker Q)\backslash H$ and suppose
$S$ generates $\mathbb{F}_{2}^{n}$. Then for any subset $T\subset S$
satisfying $T\cap\ker(Q\mid_{\mathbb{F}_{2}^{n\perp}})=\varnothing$
we have

\[
\left|T^{\perp}\cap T\right|\le\left|S\right|-\left|T\right|.
\]
\end{prop}

\begin{proof}
By Lemma \ref{lem:perp_sets_bound1} the result holds when $\ker(Q\mid_{\mathbb{F}_{2}^{n\perp}})=0$,
so assume $\ker(Q\mid_{\mathbb{F}_{2}^{n\perp}})\neq0$. Decompose
$\mathbb{F}_{2}^{n}=V\oplus R\oplus R_{0}$, where $R\oplus R_{0}=\mathbb{F}_{2}^{n\perp}$
and $R_{0}=\ker(Q\mid_{R\oplus R_{0}})$. Let $Y=V\oplus R$.

Let $T_{z}=T\cap\left(Y+z\right)$ for any $z\in R_{0}$. By translating
back each of these cosets we obtain the subset $T'=\cup_{z\in R_{0}}(T_{z}+z)$
of $Y$. Note $T'^{\perp}=T^{\perp}$. We have
\begin{align*}
\left|T^{\perp}\cap T\right| & =\sum_{z\in R_{0}}\left|T^{\perp}\cap T_{z}\right|\\
 & \le\sum_{z\in R_{0}\backslash H}\left|T'^{\perp}\cap\left(T'\cap H\right)\right|+\sum_{z\in R_{0}\cap H}\left|T'^{\perp}\cap\left(T'\backslash H\right)\right|\\
 & =\frac{\left|R_{0}\right|}{\left[R_{0}:R_{0}\cap H\right]}\left|T'^{\perp}\cap T'\right|.
\end{align*}

We consider two cases. 

\textit{Case 1:} If $R_{0}\subset H$ then since $T_{z}\cap H=\varnothing$
this implies $T'\subset Y\backslash H$ so by Lemma \ref{lem:perp_sets_bound1}
$\left|T'^{\perp}\cap T'\right|\le\left|(\ker Q\mid_{Y})\backslash H\right|-\left|T'\right|.$

If $x\in(\ker Q\mid_{Y})\backslash(H\cup T')$ then clearly $x+z\in(\ker Q\mid_{Y+z})\backslash(H\cup T_{z})$
for any $z\in R_{0}$. Hence
\begin{align*}
\left|R_{0}\right|\left(\left|(\ker Q\mid_{Y})\backslash H\right|-\left|T'\right|\right) & \le\left|R_{0}\right|\min_{z\in R_{0}}\left(\left|(\ker Q\mid_{Y+z})\backslash H\right|-\left|T_{z}\right|\right)\\
 & \le\sum_{z\in R_{0}}\left|(\ker Q\mid_{Y+z})\backslash H\right|-\left|T_{z}\right|\\
 & =\left|S\right|-\left|T\right|.
\end{align*}

\textit{Case 2:} If $R_{0}\nsubseteq H$ then by Remark \ref{rem:A-simpler-version}
(note that the condition $T\cap\ker(Q\mid_{\mathbb{F}_{2}^{n\perp}})=\varnothing$
implies $T'\subset(\ker Q\mid_{Y})\backslash\left\{ 0\right\} $)
we have $\left|T'^{\perp}\cap T'\right|\le\left|(\ker Q\mid_{Y})\backslash\left\{ 0\right\} \right|-\left|T'\right|.$
Then 
\begin{align*}
\left|T^{\perp}\cap T\right| & \le\frac{\left|R_{0}\right|}{2}\left(\left|(\ker Q\mid_{Y})\backslash\left\{ 0\right\} \right|-\left|T'\right|\right).
\end{align*}
 Let $U_{1}=(\ker(Q\mid_{Y})\backslash(\left\{ 0\right\} \cup T'))\cap H$
and $U_{2}=\ker(Q\mid_{Y})\backslash(\left\{ 0\right\} \cup T'\cup H)$.
If $x\in U_{1}$ then $x+z\in(\ker Q\mid_{Y+z})\backslash(H\cup T_{z})$
for $z\in R_{0}\backslash H$ and similarly if $x\in U_{2}$ and $z\in R_{0}\cap H$.
Thus 
\begin{align*}
\frac{\left|R_{0}\right|}{2}\left(\left|(\ker Q\mid_{Y})\backslash\left\{ 0\right\} \right|-\left|T'\right|\right) & =\sum_{z\in R_{0}\backslash H}\left|U_{1}\right|+\sum_{z\in R_{0}\cap H}\left|U_{2}\right|\\
 & \le\sum_{z\in R_{0}}\left|(\ker Q\mid_{Y+z})\backslash H\right|-\left|T_{z}\right|\\
 & =\left|S\right|-\left|T\right|.
\end{align*}
 This completes the proof.
\end{proof}

\section{\label{sec:The graph}The graph $\mathcal{G}^{*}$}

In this section we define the graph whose maximal disconnected subgraphs
we will be interested in classifying.

Recall we have a fixed isomorphism $G/\Phi\left(G\right)\cong\mathbb{F}_{2}^{n}$.
Define the graph $\mathcal{G}\left(\mathbb{F}_{2}^{n}\right)$ with
vertex set $\mathbb{F}_{2}^{n}$ and let $u,v\in\mathbb{F}_{2}^{n}$
be connected by an edge if their lifts in $G$ do not commute.

Define a set of elements $F=\left\{ u_{w},v_{w}\mid w\in\mathbb{F}_{2}^{n}\right\} $.
We say $u\in F$ is supported on $w\in\mathbb{F}_{2}^{n}$ if $u$
is of the form $u_{w}$ or $v_{w}$. Let $s:F\longrightarrow\mathbb{F}_{2}^{n}$
be the function sending an element to its support. Additionally for
any subsets $F_{1}\subset F$ and $V\subset\mathbb{F}_{2}^{n}$ we
say $F_{1}$ is supported on $V$ if $s\left(F_{1}\right)\subset V$.
We say that $F_{1}$ has full support on $V$ if $s\left(F_{1}\right)=V$.
For any subset $T\subset F$ we let $F\left(T\right)=\left\{ u\in F\mid s\left(u\right)\in T\right\} $.

We define $u,v\in F$ to be connected by an edge if $u=u_{w}$ and
$v=v_{z}$ and $\left(w,z\right)$ is an edge in $\mathcal{G}\left(\mathbb{F}_{2}^{n}\right)$.
Define $L\left(u,v\right)=1$ if $u,v$ are connected and $0$ if
not.
\begin{defn}
Let $T_{1},\ldots,T_{k}$ be subsets of $\mathbb{F}_{2}^{n}$ and
let $\widehat{T}=\prod_{i=1}^{k}T_{i}$. Define the graph $\mathcal{G}^{*}(\widehat{T})$
with vertex set $F(\widehat{T})=\prod_{i=1}^{k}F\left(T_{i}\right)$
and an edge between $u,v\in F(\widehat{T})$ if $\sum_{i=1}^{k}L\left(u_{i},v_{i}\right)=1$.
\end{defn}

We explain the relationship $\mathcal{G}^{*}(\widehat{T})$ holds
to the counting function (\ref{eq:formulathmstatement}).

For each $i=1,\ldots,r$ let $d_{i}=D_{2i-1}D_{2i}$. Let $U=\left[2r\right]$.
Note that the expression in the summation in formula (\ref{eq:formulathmstatement})
can be expanded as 
\begin{align*}
\sum_{\left(d_{1},\ldots,d_{r}\right)\in\mathcal{D}}\prod_{i=1}^{r}\prod_{p\mid d_{i}}\left(1+\left(\frac{\prod_{j\in S_{i}}d_{j}}{p}\right)^{\mathrm{ord}_{p}\left(d\right)}\right) & =\sum_{d=\prod D_{u}}\prod_{u,v\in U}\left(\frac{D_{u}}{D_{v}}\right)^{\Phi\left(u,v\right)}
\end{align*}
where $\Phi\left(u,v\right)\in\mathbb{F}_{2}$. That is $\Phi\left(u,v\right)=1$
if the symbol $(\frac{D_{u}}{D_{v}})$ appears in the expression and
$0$ otherwise.

Let $U_{1},\ldots,U_{k}$ be $k$ sets with $U_{i}=\left[2r_{i}\right]$.
For $u,v\in\prod_{i=1}^{k}U_{i}$ define $\Phi_{k}\left(u,v\right)=\sum_{i=1}^{k}\Phi\left(u_{i},v_{i}\right)$.
Then we define $u,v\in\prod_{i=1}^{k}U_{i}$ to be unlinked if 
\[
\Delta_{k}\left(u,v\right)=\Phi_{k}\left(u,v\right)+\Phi_{k}\left(v,u\right)=0.
\]

For any $1\le i\le r$, if $x_{i}=\prod_{j=1}^{n}x_{j}^{w_{j}}$ define
a bijective mapping $\varphi=\left(\varphi_{1},\ldots,\varphi_{k}\right):\prod_{i=1}^{k}U_{i}\longrightarrow F(\widehat{T})$
coordinate-wise with $\varphi_{i}$ given by
\begin{align*}
 & 2i-1\mapsto u_{w}\\
 & 2i\mapsto v_{w}.
\end{align*}
Then it is easy to see that $\Delta_{k}=L_{k}\circ\varphi$, and we
will use this fact to view $\Delta_{k}$ as defined on $\mathcal{G}^{*}(\widehat{T})$.

Let $T\subset\mathbb{F}_{2}^{n}$ be an admissible generating set
of $G/\Phi\left(G\right)$. Define a quadratic form on $\mathbb{F}_{2}^{n}$
by 
\[
Q\left(u\right)=\sum_{i\le j}a_{ij}u_{i}u_{j}
\]
 where $a_{ij}=1$ if $\left[x_{i},x_{j}\right]\neq1$ for $i\neq j$
and $a_{ii}=1$ if $\mathrm{ord}x_{i}=4$ and 0 otherwise. It has
an associated symmetric bilinear form $B\left(u,v\right)=Q\left(u+v\right)+Q\left(u\right)+Q\left(v\right)$
on $\mathbb{F}_{2}^{n}\times\mathbb{F}_{2}^{n}$. Alternately $B\left(u,v\right)=u^{T}Av$
where $A$ is the symmetric matrix given by $A_{ij}=a_{ij}$ as defined
above.
\begin{lem}
\label{lem:ker Q is order 2}Let $x=\prod_{i=1}^{n}x_{i}^{u_{i}}$
be an element of $G$ with $u\in\mathbb{F}_{2}^{n}$. Then $x$ has
order 2 if and only if $Q\left(u\right)=0$.
\end{lem}

\begin{proof}
If $x_{i}$ and $x_{j}$ do not commute then $\left[x_{i},x_{j}\right]=a$
where $a$ is the distinguished central element of $G$. If $\mathrm{ord}x_{i}=4$
then $x_{i}^{2}=a$. Explicitly writing out $x^{2}$ as a product
and interchanging the basis elements gives 
\begin{align*}
x^{2} & =\left(\prod_{i=1}^{n}x_{i}^{u_{i}}\right)^{2}\\
 & =\prod_{i=1}^{n}x_{i}^{2u_{i}}a^{\sum_{i<j}a_{ij}u_{i}u_{j}}\\
 & =a^{\sum_{i\le j}a_{ij}u_{i}u_{j}}\\
 & =a^{Q\left(u\right)}.
\end{align*}
 The result follows.
\end{proof}
\begin{lem}
\label{lem:B detects edges}Let $x=\prod_{i=1}^{n}x_{i}^{u_{i}}$
and $y=\prod_{i=1}^{n}x_{i}^{v_{i}}$. Then $\left[x,y\right]=1$
if and only if $B\left(u,v\right)=0$.
\end{lem}

\begin{proof}
Let $c\left(x,y\right)=1$ if $\left[x,y\right]\neq1$ and 0 otherwise.
By an argument analogous to the above we have 
\begin{align*}
\left(xy\right)^{2} & =x^{2}y^{2}\left[x,y\right]\\
 & =a^{Q\left(u\right)+Q\left(v\right)+c\left(x,y\right)}
\end{align*}
Note that, up to a multiple by $a$, $xy$ equals $\prod_{i=1}^{n}x_{i}^{u_{i}+v_{i}}$.
Since $\mathrm{ord}\left(xy\right)=\mathrm{ord}\left(xya\right)$
we have $\left(xy\right)^{2}=a^{Q\left(u+v\right)}$. Thus $B\left(u,v\right)=c\left(x,y\right)$
which proves the result.

\end{proof}

\section{\label{sec:Maximal disconnected subgraphs}Maximal disconnected subgraphs}

Fix a set $T\subset G/\Phi\left(G\right)$ given by $T=\left\{ t_{1},\ldots t_{r}\right\} $.
Let $\left\{ x_{1},\ldots,x_{r}\right\} $ be any choice of lifts
of $\left\{ t_{1},\ldots,t_{r}\right\} $. Recall we defined $S_{i}=\left\{ 1\le j\le r\mid\left[x_{i},x_{j}\right]\neq1\right\} $.
This is the set of indices $j$ such that $t_{j}$ has an edge to
$t_{i}$ in the graph $\mathcal{G}\left(T\right)$.

Let $C_{1},\ldots,C_{s}$ be the connected components of $\mathcal{G}\left(T\right)$.
Without loss of generality we can assume $i\in\left\{ 1,\ldots,r_{1}\right\} $
satisfies $S_{i}\neq\varnothing$.

For a fixed pair $\left(G,H\right)$ we will let $T_{0}\subset G/\Phi\left(G\right)$
be the set of all elements which lift to order 2 elements in $G\backslash H$
and call $T_{0}$ the maximal admissible generating set. The graph
$\mathcal{G}\left(T_{0}\right)$ determines $\left(G,H\right)$ as
follows.
\begin{lem}
\label{lem:complete bipartite abelian}The group $H$ is abelian if
and only if $\mathcal{G}\left(T_{0}\right)$ is a complete and bipartite
graph.
\end{lem}

\begin{proof}
Suppose $H$ is abelian. Then $G\cong H\rtimes C_{2}$ with $C_{2}$
acting by inversion. Hence every element of $G$ can be written as
$x\sigma$ with $x\in H$ and $\sigma\in C_{2}$. Every element of
$G-H$ has order 2 since $x\sigma x\sigma=xx^{-1}\sigma^{2}=1$ since
$\sigma$ acts by inversion. Thus $T_{0}=\overline{G-H}\subset G/\Phi\left(G\right)$.

If $y\in G$ has order 4 then $y^{2}=a$ where $\left\langle a\right\rangle =\Phi\left(G\right)$
and $a$ is central of order 2. Hence $y^{-1}=ya$. Let $x_{1}\sigma,x_{2}\sigma\in G$.
Then 
\[
x_{1}\sigma x_{2}\sigma=x_{2}^{-1}\sigma x_{1}^{-1}\sigma.
\]
Thus $x_{1}\sigma$ and $x_{2}\sigma$ commute if and only if $x_{i}$
both have order 2, or both have order 4 in $H$.

Combining the above we see that $\mathcal{G}\left(T_{0}\right)=\left(W_{1},W_{2}\right)$
is complete and bipartite where $W_{1}=\left\{ \overline{x\sigma}\mid\mathrm{ord}\left(x\right)=2\right\} $
and $W_{2}=\left\{ \overline{x\sigma}\mid\mathrm{ord}\left(x\right)=4\right\} $.

Now suppose $\mathcal{G}\left(T_{0}\right)=\left(W_{1},W_{2}\right)$
is complete and bipartite. Since $H$ has index 2 in $G$, clearly
$H=\left\{ \widehat{t_{i}t_{j}}\mid t_{i},t_{j}\in T_{0}\right\} $
where $\widehat{t_{i}t_{j}}$ denotes the lift from $G/\Phi\left(G\right)$
to $G$. Suppose $x_{1}=\widehat{t_{1}t_{2}}$ and $x_{2}=\widehat{s_{1}s_{2}}$
with $t_{i},s_{i}\in T_{0}$. By checking the cases $t_{i}\in W_{k}$,
$s_{j}\in W_{l},$ with $i,j,k,l=1,2$ we see that $x_{1}$ and $x_{2}$
commute, hence $H$ is abelian.
\end{proof}
Throughout this paper we will primarily work with the sets $T_{0}\subset G/\Phi\left(G\right)$
and hence refer to the condition on $\mathcal{G}\left(T_{0}\right)$
rather than that on $\left(G,H\right)$ from the above lemma. The
lemma will be applied at the end to translate the theorem statements
to ones about $\left(G,H\right)$.

For any $T\subset G/\Phi\left(G\right)$ define $c_{T}$ to be the
number of orbits of the set of lifts of $T$ acting on itself by conjugation.
If $T$ is a generating set this is the number of conjugacy classes
of the lifts of $T$ in $G$.

We will write $c=c_{T_{0}}$.

The numbers $c_{T}$ will be crucial to determining the maximal disconnected
subgraphs of $\mathcal{G}^{*}$.
\begin{lem}
\label{lem:c_T_bound}Let $T\subseteq T_{0}$. Then $c_{T}\le c$.
Furthermore if $T$ is a generating set then equality holds if and
only if $T=T_{0}$.
\end{lem}

\begin{proof}
Note $T\subset T_{0}=\ker Q\backslash H$ by definition and Lemma
\ref{lem:ker Q is order 2}.

By definition of the quadratic form $Q$ and its associated bilinear
form $B$ and Lemma \ref{lem:B detects edges}, $T^{\perp}\cap T$
is the subset of $T$ whose lifts have trivial action. There are two
lifts of any $t\in T$, and hence they form one orbit if $t\notin T^{\perp}\cap T$
and two distinct orbits if $t\in T^{\perp}\cap T$.

Let $K=\ker(Q\mid_{\mathbb{F}_{2}^{n\perp}})$. Let $T_{0}'=T_{0}\backslash K$
and $T_{0}''=T_{0}\cap K$. Similarly let $T'=T\cap T_{0}'$ and $T''=T\cap T_{0}''$.
The above implies that the lifts of each element of $T_{0}'$ form
one orbit (under the action of $T_{0}$) and the lifts of each element
of $T_{0}''$ form two orbits.

Then 
\[
c_{T}\le c-(\left|T_{0}'\right|-\left|T'\right|)-2\left(\left|T_{0}''\right|-\left|T''\right|\right)+\left|T'^{\perp}\cap T'\right|.
\]
It follows from Proposition \ref{prop:perp_sets_bound2} that 
\[
\left|T'^{\perp}\cap T'\right|\le\left|T'_{0}\right|-\left|T'\right|.
\]
 Thus we have $c_{T}\le c$ for any $T\subseteq T_{0}$.

Suppose $c_{T}=c$ and $T$ is a generating set. If $T\neq T_{0}$
then there is some $x\in T$ which is fixed by the conjugation action
of $T$ but not $T_{0}$. Since any set of lifts of $T$ generates
$G$ this implies any lift $x'\in Z\left(G\right)$. This is a contradiction
since $x$ is not fixed by $T_{0}$. Thus $T=T_{0}$.
\end{proof}
If $T$ is a generating set then the proof of the above result does
not require Proposition \ref{prop:perp_sets_bound2}. However the
proofs in Section \ref{sec:Maximal disconnected subgraphs} will necessitate
dealing with sets $T$ which are not generating sets.

\subsection{Maximal unlinked sets}

We will now define certain subsets of $\mathcal{G}^{*}(\widehat{T})$
called maximal unlinked sets, which we will show are equal to maximal
disconnected subgraphs of $\mathcal{G}^{*}(\widehat{T})$ of maximum
size.

As above fix a set $T\subset G/\Phi\left(G\right)$ given by $T=\left\{ t_{1},\ldots t_{r}\right\} $.
Throughout this section we will use the identification made in Section
\ref{sec:The graph} between the vertices of $\mathcal{G}^{*}\left(T\right)$
and $U=\left\{ 1,\ldots,2r\right\} $. We partition the vertices of
$\mathcal{G}^{*}\left(T\right)$ as follows.

For $1\le j\le s$ we let
\begin{align*}
A_{j} & =\left\{ 2i-1\mid1\le i\le r_{1},i\in C_{j}\right\} ,\\
B_{j} & =\left\{ 2i\mid1\le i\le r_{1},i\in C_{j}\right\} 
\end{align*}
 Also we let
\[
C=\left\{ 2i-1,2i\mid i=r_{1}+1,\ldots,r\right\} .
\]
Let $E_{j}=A_{j}\cup B_{j}$ so that $U=\cup_{j=1}^{s}E_{j}\cup C$.

It is easy to verify the following description of $\mathcal{G}^{*}\left(T\right)$.
If $i_{1}\in C_{j}$ and $u=2i_{1}-1$ (that is $u\in A_{j}$) then
$u$ has an edge to every $2i_{2}$ with $i_{2}\in C_{j}$ and $i_{2}\neq i_{1}$.
It has no edge to every other element of $U$. Similarly $u=2i_{1}$
(that is $u\in B_{j}$) has an edge to every $2i_{2}-1$ with $i_{2}\in C_{j}$
and $i_{2}\neq i_{1}$ and has no edge to everything else. Every element
of $C$ is disconnected.

Consider the set of subsets of $U$
\begin{equation}
\mathcal{T}=\left\{ \cup_{j=1}^{s}W_{j}\cup C\mid W_{j}=A_{j},B_{j}\right\} .\label{eq:type2_sets}
\end{equation}

Let $\left(G,H\right)=\left(D_{4},C_{4}\right)$ and $S\subset\mathbb{F}_{2}^{2}$
be the maximal adimissible generating set for this pair. For the purpose
of the next subsection we will identify $\mathcal{G}^{*}\left(S^{k}\right)$
with $\mathbb{F}_{2}^{2k}$ and refer to disconnected subgraphs of
$\mathbb{F}_{2}^{2k}$.

Suppose $T$ is a generating set, such that $T=\left(W_{1},W_{2}\right)$
is a complete bipartite graph. Partition the vertices into four disjoint
sets 
\begin{align*}
A_{0} & =\left\{ 2u-1\mid u\in W_{1}\right\} ,\\
A_{3} & =\left\{ 2u-1\mid u\in W_{2}\right\} ,\\
A_{1} & =\left\{ 2u\mid u\in W_{1}\right\} ,\\
A_{2} & =\left\{ 2u\mid u\in W_{2}\right\} .
\end{align*}

Let $T_{1},\ldots,T_{k}$ be a sequence of such generating sets. Define
a bijection between $\left\{ A_{0},A_{1},A_{2},A_{3}\right\} ^{k}$
and elements of $\mathbb{F}_{2}^{2k}$ which acts coordinatewise by
sending $A_{i}$ to the binary representation of $i$ in $\mathbb{F}_{2}^{2}$.
We will implicitly use this identification in what follows. For any
$\mathbf{q}\in\left(\mathbb{F}_{2}^{2}\right)^{k}$ let
\[
\mathcal{R}_{\mathbf{q}}=\left\{ x\in\prod_{i=1}^{k}U\mid x_{i}\in\mathbf{q}_{i}\right\} .
\]

Let $\mathcal{D}$ be the set of maximal disconnected subgraphs of
$\mathcal{G}^{*}\left(S\right)$ with full support on $S$ (see Section
\ref{sec:The graph} for the definition of the support). For any $V\subset\left(\mathbb{F}_{2}^{2}\right)^{k}$
define $\mathcal{R}_{V}=\bigcup_{\mathbf{q}\in V}\mathcal{R}_{\mathbf{q}}$.
Then define

\begin{equation}
\mathcal{R}=\left\{ R_{V}\mid V\in\mathcal{D}\right\} .\label{eq:type1_sets}
\end{equation}

Now we put together the above definitions of $\mathcal{T}$ and $\mathcal{R}$.
For generating sets $T_{1},\ldots,T_{k}$ let $Y$ be the subset of
indices $i=1,\ldots,k$ for which $T_{i}$ is a complete bipartite
graph. Let $\mathcal{T}_{i}$ correspond to $T_{i}$ for $i\notin Y$
and let $\mathcal{R}$ correspond to the set of $\left\{ T_{i}\right\} _{i\in Y}$
as defined above. Then let 
\[
\mathcal{M}(\widehat{T})=\mathcal{R}\times\prod_{j\notin Y}\mathcal{T}_{j}.
\]
 We will denote maximal unlinked sets in $\mathcal{R}$ to be of \textit{$\textit{type 1}$
}and those in $\prod_{j}\mathcal{T}_{j}$ of $\textit{type 2}$.

\subsection{Classification when $T=T_{0}$}
\begin{defn}
For any $\widehat{T}=\prod_{i=1}^{k}T_{i}$ let $\mathbf{U}(\widehat{T})$
be the set of largest maximal disconnected subgraphs of $\mathcal{G}^{*}(\widehat{T})$
with full support on $\widehat{T}$ (see Section \ref{sec:The graph}
for the definition of the support of $\widehat{T}$).
\end{defn}

We will show that when $\widehat{T}=T_{0}^{k}$ we have the equality
of sets $\mathcal{M}(\widehat{T})=\mathbf{U}(\widehat{T})$.

Throughout this section we will say $u,v\in\mathcal{G}^{*}(\mathbb{F}_{2}^{n})$
are unlinked if there is no edge between them.
\begin{lem}
\label{lem:bounded_size}Let $t\in\mathcal{M}(\widehat{T})$. Then
$\left|t\right|\le c^{k}$. Furthermore, if $T$ is a product of generating
sets then equality holds if and only if $T_{i}=T_{0}$ for all $i$.
\end{lem}

\begin{proof}
Write $t=R_{V}\times\prod_{j\notin Y}t_{j}$. Clearly $\left|t_{j}\right|=c_{T_{j}}$
and by Lemma \ref{lem:c_T_bound} $c_{T_{j}}\le c$. For $i\in Y$
write $T_{i}=\left(W_{i,1},W_{i,2}\right)$ and without loss of generality
suppose $\left|W_{i,1}\right|\ge\left|W_{i,2}\right|$ for all $i$.
Note $2\left|W_{i,j}\right|=c_{W_{i,j}}\le c$ for all $i,j$. Since
$\left|R_{\mathbf{q}}\right|\le\prod_{i\in Y}\left|W_{i,1}\right|^{r}$
for all $\mathbf{q}\in V$ we have 
\begin{align*}
\left|R_{V}\right| & \le2^{r}\prod_{i\in Y}\left|W_{i,1}\right|^{r}\\
 & =\prod_{i\in Y}c_{W_{i,1}}\\
 & \le c^{r}.
\end{align*}
 Suppose $\left|W_{i,1}\right|>\left|W_{i,2}\right|$ for some $i$.
Since $t$ has full support on $T$ there is some $\mathbf{q}\in V$
such that for some $i$ we have $\mathbf{q}_{i}\in\left\{ \left(0,1\right),\left(1,1\right)\right\} $,
that is, for $j=1$ or 3, every $x\in R_{\mathbf{q}}$ satisfies $x_{i}\in A_{j}$.
Hence 
\begin{align*}
\left|R_{\mathbf{q}}\right| & \le\left|W_{i,2}\right|\prod_{i'\neq i}\left|W_{i',1}\right|^{r}\\
 & <\prod_{i\in Y}\left|W_{i,1}\right|^{r}.
\end{align*}
Thus in this case equality holds if and only if $\left|W_{i,1}\right|=\left|W_{i,2}\right|$
(which implies $\left|R_{\mathbf{q}}\right|=(c_{T_{i}}/2)^{r}$) and
$c_{T_{i}}=c$ (which for generating sets only holds when $T_{i}=T_{0}$)
for all $i$.

Since $\left|t\right|=\left|R_{V}\right|\cdot\prod_{j\notin Y}\left|t_{j}\right|$
this completes the proof.
\end{proof}
Note that for any $x\in\mathcal{R}_{\mathbf{q}_{1}}$ and $y\in\mathcal{R}_{\mathbf{q}_{2}}$
the value of $\Phi_{r}\left(x,y\right)$ and hence $\Delta_{r}\left(x,y\right)$
is determined by $\mathbf{q}_{1}$ and $\mathbf{q}_{2}$. This induces
the function $\Delta_{r}:\mathbb{F}_{2}^{2r}\times\mathbb{F}_{2}^{2r}\longrightarrow\mathbb{F}_{2}$.
This is the same function used to detect maximal disconnected subsets
in the sense of \cite[Section 5.2, p. 472]{fk1} (which they refer
to as maximal unlinked sets). In particular there exists a non-degenerate
quadratic form $P$ defined on $\mathbb{F}_{2}^{2r}$ such that $\Delta_{r}\left(u,v\right)=P\left(u+v\right)$.
It can be seen from the explicit formula for $P$ given there that
$P$ is non-degenerate.

We will show that when $\widehat{T}$ is a product of complete bipartite
generating sets the elements of $\mathbf{U}(\widehat{T})$ are indexed
by maximal unlinked sets in the sense of Fouvry-Kl\"uners. Recall
in the previous subsection we identified $\mathcal{G}^{*}\left(S^{k}\right)$
with $\mathbb{F}_{2}^{2k}$.
\begin{lem}[{\cite[Lemma 18, Section 5.5, p. 478]{fk1}}]
\label{lem:Let-k-=002265}Let $\mathcal{U}\subset\mathbb{F}_{2}^{2r}$
be a disconnected subgraph. Then $\left|\mathcal{U}\right|\le2^{r}$
and for any $c\in\mathbb{F}_{2}^{2r}$, $c+\mathcal{U}$ is also a
disconnected subgraph. If $\left|\mathcal{U}\right|=2^{r}$, then
either $\mathcal{U}$ is a subspace of $\mathbb{F}_{2}^{2r}$ of dimension
$r$ or a coset of such a subspace.
\end{lem}

\begin{lem}
\label{lem:For--let}For $i=1,\ldots,k$ let $T_{i}\subseteq T_{0}$
(not necessarily generating sets). Let $\mathcal{T}_{i},\mathcal{R},$
be as defined above. Let $T_{i}'\subset T_{i}$ for each $i$ and
define similarly $\mathcal{T}_{i}',\mathcal{R}'$. Suppose $c_{T_{1}',\ldots,T_{k}'}=c_{T_{1},\ldots,T_{k}}=c^{k}$.

Let $x\in t$ for some $t\in\mathcal{M}(\widehat{T})$. Then for any
$t'\in\mathcal{M}\left(T'\right)$ there exists some $y\in t'$ such
that $x$ and $y$ are unlinked.
\end{lem}

\begin{proof}
First note that since $c_{T_{i}'},c_{T_{i}}\le c$ for each $i$ by
Lemma \ref{lem:c_T_bound}, this implies that $c=c_{T_{i}'}=c_{T_{i}}$
for each $i$. Let $Y'$ be the set of indices for which $T_{i}'$
is a complete bipartite graph. Let $t=t_{Y'}\times t_{0}$ and similarly
for $t'$. Let $\overline{x}$ be the projection of $x$ onto the
coordinates in $Y'$. We construct an element $y$ such that $\overline{y}$
is unlinked with $\overline{x}$ and $y_{i}$ is unlinked with $x_{i}$
for all $i\notin Y'$.

First suppose $i\notin Y'$. If $supp\left(x_{i}\right)\in T_{i}'$
then let $y_{i}$ be the element of $t_{i}'$ with $supp\left(y_{i}\right)=supp\left(x_{i}\right)$.
Now suppose $supp\left(x_{i}\right)\notin T_{i}'$. Since $c_{T_{i}'}=c_{T_{i}}$
this implies there exists some $u\in T_{i}'$ which is a fixed point
for the action of $T_{i}'$ but is not fixed for the action of $T_{i}$.
Hence $2u-1,2u\in t_{i}'$ so we let $y_{i}\in\left\{ 2u-1,2u\right\} $
have the same pairity as $x_{i}$.

Now consider all $i\in Y'$. If $T_{i}'$ is strictly contained in
$T_{i}$ with $c_{T_{i}'}=c_{T_{i}}$ then $T_{i}'$ must contain
at least one disconnected element, which contradicts that it is complete.
Thus $T_{i}'=T_{i}$. Suppose $\overline{x}\in\mathbf{q}$. Since
the function $\Delta_{r}$ is invariant under translation (simulatneously
in both coordinates) by Lemma \ref{lem:Let-k-=002265}, we can assume
$t_{Y'}'=R_{V}$ where $V$ is a subspace of $\mathbb{F}_{2}^{2r}$.
Thus we want to show that there exists $\mathbf{q}'\in V$ with $\Delta_{r}\left(\mathbf{q},\mathbf{q}'\right)=0$.

Since $V$ is an unlinked set containing 0 we have $P\left(V\right)=0$,
that is $V$ is a totally singular space for $P$. By Lemma \ref{lem:if--is}
there exists a subspace $W\subset\mathbb{F}_{2}^{2r}$ such that $P\left(W\right)=0$
and $W\cap V=\varnothing$ and $\left|W\right|=\left|V\right|=2^{r}$.
Thus $V\oplus W=\mathbb{F}_{2}^{2r}$. Write $\mathbf{q}=v+w$ for
some $v\in V$ and $w\in W$. Let $\mathbf{q}'=v$. Then 
\begin{align*}
\Delta_{r}\left(\mathbf{q},\mathbf{q}'\right) & =P\left(\mathbf{q}+\mathbf{q}'\right)\\
 & =P\left(w\right)\\
 & =0.
\end{align*}
 Thus we can let $\overline{y}\in\mathcal{R}_{\mathbf{q'}}$ be any
element and we get $\Delta_{r}\left(\overline{x},\overline{y}\right)=0$.

Combining the above for all coordinates we construct an element with
$\Delta_{k}\left(x,y\right)=0$.
\end{proof}
We need one more graph-theoretic lemma. For any edge $\left(u,v\right)\in E\left(T\right)$
consider the property $P\left(u,v\right)$: there exists some $w\in T$
such that either $\left(w,u\right),\left(w,v\right)\notin E\left(T\right)$
or $\left(w,u\right),\left(w,v\right)\in E\left(T\right)$.
\begin{lem}
\label{lem:The-condition-}The condition $P\left(u,v\right)$ is false
for all $\left(u,v\right)\in E\left(T\right)$ if and only if $T$
is a complete bipartite graph.
\end{lem}

\begin{proof}
Let $\left(u,v\right)\in E\left(T\right)$. If $P\left(u,v\right)$
is false then for every $w\in T$ either $\left(w,u\right)\in E\left(T\right)$
or $\left(w,v\right)\in E\left(T\right)$ but not both. Let $W_{u}=\left\{ w\in T\mid\left(w,v\right)\in E\left(T\right)\right\} $
and similarly for $W_{v}$. Then it is not hard to see that $\left(W_{u},W_{v}\right)$
is a complete partition of $T$. The other direction is trivial.
\end{proof}
We will make some notation which will be used in the following proofs.
Given any subset $\mathcal{U}\subseteq U^{k}$ we can write it as
a disjoint union $\mathcal{U}=\bigcup_{j=1}^{s+1}\mathcal{U}_{j}$
where we let
\begin{align*}
\mathcal{U}_{j} & =\left\{ u\in\mathcal{U}\mid u_{k}\in E_{j}\right\} \text{ for }1\le j\le s,\\
\mathcal{U}_{s+1} & =\left\{ u\in\mathcal{U}\mid u_{k}\in C\right\} .
\end{align*}
 Let $p$ be the projection onto the first $k-1$ coordinates. Let
$\mathcal{U}_{j,i}=\left\{ u\in\mathcal{U}_{j}\mid u_{k}=i\right\} $.
For simplicity we will write $V_{i}=p\left(\mathcal{U}_{j,i}\right)$.
\begin{lem}
\label{lem:unlinkedset_bound}Any disconnected subgraph $\mathcal{U}$
of $\mathcal{G}^{*}(\widehat{T})$ satisfies $\left|\mathcal{U}\right|\le c^{k}$.
\end{lem}

\begin{proof}
We proceed by induction on $k$. Let $\widehat{T}=\prod_{i=1}^{k}T_{i}$
and $\mathcal{U}\in\mathbf{U}(\widehat{T})$. Assume without loss
of generality that $T_{i}$ is bipartite complete exactly for $i=1,\ldots,r$.
The base case $k=r$ follows by Lemma \ref{lem:Let-k-=002265} which
says that the $\mathbf{U}(\widehat{T})=\mathcal{M}(\widehat{T})$.
Any element of $\mathcal{M}(\widehat{T})$ has size bounded by $c^{k}$
by Lemma \ref{lem:bounded_size}.

Suppose the statement is true for $k-1$. For any $u\in T_{k}$, $V_{2u-1}\cup V_{2u}$
is an unlinked set, hence by the induction hypothesis $\left|V_{2u-1}\cup V_{2u}\right|\le c^{k-1}$.
Suppose $\left\{ 2u-1,2u\right\} \subset E_{j}$ for some $1\le j\le s$,
so $\left(u,v\right)\in E\left(T_{k}\right)$ for some $v\in T_{k}$
and either $V_{2v-1}$ or $V_{2v}$ is non-empty. This implies that
$V_{2u-1}\cap V_{2u}=\varnothing$. Hence $|\mathcal{U}_{j,2u-1}|+|\mathcal{U}_{j,2u}|=|\mathcal{U}_{j,2u-1}\cup\mathcal{U}_{j,2u}|\le c^{k-1}$.
If $\left\{ 2u-1,2u\right\} \subset C$ then clearly $\left|\mathcal{U}_{j,2u-1}\right|,\left|\mathcal{U}_{j,2u}\right|\le c^{k-1}$.

Thus summing $\left|\mathcal{U}_{j,i}\right|$ over all $i\in U_{k}$
gives the desired result.
\end{proof}
\begin{thm}
\label{thm:The-maximal-unlinked}Let $\widehat{T}=\prod_{i=1}^{k}T_{i}$
where $T_{i}$ are all generating sets. Then any $\mathcal{U}\in\mathbf{U}(\widehat{T})$
satisfies $\left|\mathcal{U}\right|=c^{k}$ if and only if $T_{i}=T_{0}$
for all $i$. In this case $\mathbf{U}(\widehat{T})=\mathcal{M}(\widehat{T})$.
\end{thm}

\begin{proof}
It is clear that if $T_{i}=T_{0}$ for all $i$ then any $t\in\mathcal{M}(\widehat{T})$
is an unlinked set of size $c^{k}$ with full support. This gives
one direction of the first part of the theorem statement.

To prove the other direction and the second part we proceed by induction
on $k$. Let $\mathcal{U}\in\mathbf{U}(\widehat{T})$. Assume without
loss of generality that $T_{i}$ is bipartite complete exactly for
$i=1,\ldots,r$. The base case is $k=r$. As above, by Lemmas \ref{lem:Let-k-=002265}
and \ref{lem:bounded_size} $\mathbf{U}(\widehat{T})=\mathcal{M}(\widehat{T})$
and its elements have size bounded by $c^{r}$ with equality if and
only if $T_{i}=T_{0}$ for all $i$.

Suppose the statement is true for $k-1$. Suppose $\mathcal{U}$ is
an unlinked set with full support on $T$ such that $\left|\mathcal{U}\right|=c^{k}$.
We will show that either $p\left(\mathcal{U}_{j,i}\right)=\varnothing$
exactly for all $i\in A_{j}$ or for all $i\in B_{j}$. From the coordinatewise
definition of $\Delta_{k}$ it is clear that for any $\left(u,v\right)\in E\left(T_{k}\right)$
the elements of the set
\[
\mathcal{X}=\left\{ V_{2u-1}\cup V_{2u},V_{2u-1}\cup V_{2v-1},V_{2u}\cup V_{2v},V_{2v-1}\cup V_{2v}\right\} 
\]
 are unlinked sets. Furthermore every element of $V_{2u-1}$ is linked
with every element of $V_{2v}$, and similarly for $V_{2u}$ and $V_{2v-1}$.
Since $\mathcal{U}$ has full support, either $V_{2u-1}$ or $V_{2u}$
is non-empty.  Also note that $V_{2u-1}\cap V_{2u}=\varnothing$.
Lemma \ref{lem:unlinkedset_bound} implies $\left|\mathcal{U}_{j,2u-1}\cup\mathcal{U}_{j,2u}\right|\le c^{k-1}$.

By Lemma \ref{lem:The-condition-} there exists $\left(u,v\right)\in E\left(T_{k}\right)$
such that the condition $P\left(u,v\right)$ is true. We will use
this to rule out the case when $V_{2u-1},V_{2u},V_{2v-1},V_{2v}$
are all non-empty. Suppose towards a contradiction that this is true.
This implies $V_{2u-1},V_{2u},V_{2v-1},V_{2v}$ are all disjoint.

Lemma \ref{lem:unlinkedset_bound} implies $\left|\mathcal{U}_{2u}\cup\mathcal{U}_{2v}\right|\le c^{k-1}$.
Let $w\in T_{k}$ verify $P\left(u,v\right)$. First suppose that
$\left(w,u\right),\left(w,v\right)\in E\left(T_{k}\right)$ and without
loss of generality that $V_{2w}\neq\varnothing$. 

Let $T'=\prod T_{i}'$ be the support of the coordinates of the unlinked
set $V_{2u-1}\cup V_{2v-1}$. By the induction hypothesis, if $\left|V_{2u-1}\cup V_{2v-1}\right|=c^{k-1}$
then $T_{i}'=T_{i}=T_{0}$ and hence $c_{T_{i}'}=c_{T_{i}}=c$ for
all $1\le i\le k-1$. Any element of $V_{2w}$ is linked with every
element of $V_{2u-1}\cup V_{2v-1}$ and hence by Lemma \ref{lem:For--let}
$\left|V_{2u-1}\cup V_{2v-1}\right|<c^{k-1}$. Since $V_{2u-1}\cap V_{2v-1}=\varnothing$
this implies $\left|\mathcal{U}_{2u-1}\cup\mathcal{U}_{2v-1}\right|<c^{k-1}$.
Thus 
\[
\left|\mathcal{U}_{2u-1}\cup\mathcal{U}_{2v-1}\cup\mathcal{U}_{2u}\cup\mathcal{U}_{2v}\right|<2c^{k-1}
\]
 which implies $\left|\mathcal{U}_{j}\right|<\left|C_{j}\right|c^{k-1}$.

Now suppose that $\left(w,u\right),\left(w,v\right)\notin E\left(T_{k}\right)$
and again without loss of generality that $V_{2w}\neq\varnothing$.
Then any element of $V_{2w}$ is unlinked with every set in $\mathcal{X}$.
Since $V_{2u-1}\cup V_{2u}$ and $V_{2v-1}\cup V_{2v}$ are disjoint
this implies they cannot both be maximal so as above we get $\left|\mathcal{U}_{j}\right|<\left|C_{j}\right|c^{k-1}$.

Thus $V_{2u-1},V_{2u},V_{2v-1},V_{2v}$ cannot all be non-empty. Suppose
without loss of generality that $V_{2u-1}=\varnothing$ which implies
$V_{2u}\neq\varnothing$. 

If $V_{2v-1}\neq\varnothing$ and $x\in V_{2v-1}$ then every element
of $V_{2u}$ is connected to $x$. We will apply Lemma \ref{lem:For--let}
to show that $V_{2u}$ cannot be a maximal diconnected subgraph. Let
$T'=\prod T_{i}'$ be the support of the coordinates of the unlinked
set $V_{2u}$. If $\left|V_{2u}\right|=c^{k-1}$ by the same argument
as above by Lemma \ref{lem:For--let} there is an element of $V_{2u}$
which is unlinked with $x$, a contradiction. Thus $\left|V_{2u}\right|<c^{k-1}$.
If in addition $V_{2v}=\varnothing$ then this implies $\left|\mathcal{U}_{j}\right|<\left|C_{j}\right|c^{k-1}$.

Thus we have shown that if $\mathcal{U}\in\mathbf{U}(\widehat{T})$
such that $\left|\mathcal{U}\right|=c^{k}$ then $\mathcal{U}\in\mathcal{M}(\widehat{T})$.

Thus we have shown that either $p\left(\mathcal{U}_{j,i}\right)=\varnothing$
exactly for all $i\in A_{j}$ or exactly for all $i\in B_{j}$ as
claimed. Hence $\cup_{j,i}p\left(\mathcal{U}_{j,i}\right)$ is an
unlinked set, which implies $\mathcal{U}_{j}$ has maximal size when
$p\left(\mathcal{U}_{j,i}\right)$ is equal to a unique maximal unlinked
set fully supported on $\prod_{i=1}^{k-1}T_{i}$ for all $i$ and
$j$.

\end{proof}

\section{\label{sec:Asymptotic-analysis-of}Asymptotic analysis}

 Given a product of generating sets $\widehat{T}=\prod_{i}T_{i}$
we define $f_{\widehat{T}}=\prod_{i}f_{T_{i}}$. We will compute the
averages of functions of the form $f_{\widehat{T}}$, and use them
to obtain the $k$th moments of $f$ which will in turn determine
a distribution.

\subsection{Notational preliminaries}

We define some notation for the next theorem. Let $\mathrm{Aut}_{H,\widehat{T}}\left(G,\Phi(G)\right)=\prod_{i}\mathrm{Aut}_{H,T_{i}}\left(G,\Phi(G)\right)$.
Define $\mathcal{N}^{\pm}\left(T_{i}\right)$ be the set of subsets
$N\subset T_{i}$ of even (respectively odd) cardinality. Let $\mathcal{N^{\pm}}(\widehat{T})=\prod_{i}\mathcal{N}^{\pm}\left(T_{i}\right)$.

We define a condition on tuples of congruence classes which will be
used in the next theorem. Let $r\left(j\right)=\left|T_{j}\right|$
for $j=1,\ldots,k$. Let $\Upsilon\in\prod_{j=1}^{k}\left[r\left(j\right)\right]$,
$\alpha\in\left\{ 0,2,3\right\} $. Let $\widehat{U}=\prod_{i=1}^{k}U_{i}$
with $U_{i}=\left[2r\left(j\right)\right]$ (recall from Section \ref{sec:The graph}
there is a bijection $\widehat{U}\longrightarrow\mathcal{G}^{*}(\widehat{T})$).

Let $\left(h_{u}\right)_{u\in\widehat{U}}$ be a tuple of integers.
For each $1\le j\le k$ and $i\in\left[r\left(j\right)\right]$ define
the following conditions on $\left(h_{u}\right)$ modulo 4
\begin{equation}
\prod_{u,v}h_{u}h_{v}\equiv\begin{cases}
-1 & i\in N_{j},i\neq\Upsilon_{j}\\
1 & i\notin N_{j},i\neq\Upsilon_{j}\\
1 & i\in N_{j},i=\Upsilon_{j},\alpha=2\\
\pm1 & i\in N_{j},i=\Upsilon_{j},\alpha=3
\end{cases}\label{eq:general cong cond}
\end{equation}
 where the product is over all $u,v\in\widehat{U}$ with $u_{j}=2i-1,v_{j}=2i$.
For any positive integer $m$ with $4\mid m$ let 
\[
\Lambda_{m}\left(N,\Upsilon\right)=\left\{ \left(h_{u}\right)\in\left(\mathbb{Z}/m\mathbb{Z}\right)^{\widehat{U}}\mid\left(h_{u}\mathrm{mod}4\right)_{u}\text{ satisfies }(\ref{eq:general cong cond})\right\} ,
\]
\[
\Lambda_{m}\left(\mathcal{U},N,\Upsilon\right)=\left\{ \left(h_{u}\right)_{u\in\widehat{U}}\in\Lambda_{m}\left(N,\Upsilon\right)\mid h_{u}=1\text{ if }u\notin\mathcal{U}\right\} .
\]
 We will use $\Lambda_{m}\left(N,\Upsilon\right)$ to restrict the
congruence classes of certain factorizations of the discriminant in
the upcoming theorem.

Next we define some functions which will be used in the description
of the coefficient of the main term of $f_{\widehat{T}}$. For $N\in\mathcal{N^{\pm}}\left(T\right)$
and $u\in\mathcal{G}^{*}(\widehat{T})$ let
\[
\lambda^{N}\left(u\right)=\begin{cases}
1 & u=2i,\text{ }i\le r_{1},\text{ }\left|N\cap S_{i}\right|\text{ odd}\\
0 & \text{else.}
\end{cases}
\]
 Also recall we defined $\Phi_{k}\left(u,v\right)=\sum_{i=1}^{k}\Phi\left(u_{i},v_{i}\right)$
for $u,v\in\mathcal{G}^{*}\left(T\right)$, where the function $\Phi$
is defined at the start of Section \ref{sec:The graph}. Let $\Gamma\subseteq\left[k\right]$.
Then let

\begin{align}
\lambda_{k}^{N}\left(u\right) & =\sum_{i=1}^{k}\lambda^{N_{i}}\left(u_{i}\right),\label{eq:lambda_def}\\
\gamma_{\Upsilon,\Gamma}\left(u\right) & =\sum_{i\in\Gamma}\Phi\left(u_{i},2\Upsilon_{i}\right),\nonumber \\
\psi_{\Upsilon}\left(u\right) & =\sum_{i=1}^{k}\Phi\left(2\Upsilon_{i},u_{i}\right).\nonumber 
\end{align}

Finally for $(h_{u})\in\Lambda_{m}\left(N,\Upsilon\right)$ and $\mathcal{U}\subseteq\widehat{U}$
define 
\begin{align}
\chi_{1}\left(\left(h_{u}\right),\mathcal{U}\right) & =\prod_{\left\{ u,v\right\} \subset\mathcal{U}}\left(-1\right)^{\Phi_{k}\left(u,v\right)\frac{h_{u}-1}{2}\frac{h_{v}-1}{2}},\label{eq:chi def}\\
\chi_{2}\left(\left(h_{u}\right),\mathcal{U}\right) & =\prod_{u\in\mathcal{U}}\left(-1\right)^{\lambda_{k}^{N}\left(u\right)\frac{h_{u}-1}{2}},\nonumber \\
\chi_{3}\left(\left(h_{u}\right),\mathcal{U}\right) & =\prod_{u\in\mathcal{U}}\left(-1\right)^{\frac{h_{u}^{2}-1}{8}(\gamma_{\Upsilon,\Gamma}\left(u\right)+\psi_{\Upsilon}\left(u\right)\alpha)}.\nonumber 
\end{align}

Let $\mathcal{D}_{X,\alpha}^{\pm}$ to be the set of positive (resp.
negative) fundamental discriminants divisible by $2^{\alpha}$.

\subsection{Proof of the theorem}
\begin{thm}
\label{thm:asymptotic_analysis}Let $\widehat{T}=\prod_{i}T_{i}$
be a product of generating sets and let $\mathbf{c}=\max_{\mathcal{U}\in\mathbf{U}(\widehat{T})}\left|\mathcal{U}\right|$.
With the above notation, for any positive constant $a\in\mathbb{R}$
and $\alpha\in\left\{ 0,2,3\right\} $ the sum of $f_{\widehat{T}}\left(d\right)/a^{\omega\left(d\right)}$
over $\mathcal{D}_{X,\alpha}^{\pm}$ is

\begin{align*}
\sum_{d\in\mathcal{D}_{X,\alpha}^{\pm}}\frac{f_{\widehat{T}}\left(d\right)}{a^{\omega\left(d\right)}} & =\frac{\Gamma_{\mathbf{U}(\widehat{T})}}{4^{\mathbf{c}}2^{kn}\left|\mathrm{Aut}_{H,\widehat{T}}\left(G,\Phi(G)\right)\right|}\frac{4}{\pi^{2}}X(\log X)^{(\mathbf{c}/a)-1}+o\left(X(\log X)^{(\mathbf{c}/a)-1}\right)
\end{align*}
where
\begin{equation}
\Gamma{}_{\mathbf{U}(\widehat{T})}=\sum_{N\in\mathcal{N^{\pm}}(\widehat{T})}\sum_{\Gamma\subset\left[k\right]}\sum_{\Upsilon\in\prod_{j}\left[r\left(j\right)\right]}S_{N,\Gamma,\Upsilon}(\widehat{T}),\label{eq:Gamma def}
\end{equation}
 and
\begin{align}
S_{N,\Gamma,\Upsilon} & =\sum_{\substack{\mathcal{U}\in\mathbf{U}\left(\widehat{T}\right)\\
\left|\mathcal{U}\right|=\mathbf{c}
}
}\sum_{\left(h_{u}\right)\in\Lambda_{8}\left(\mathcal{U},N,\Upsilon\right)}\prod_{i=1}^{3}\chi_{i}\left(\left(h_{u}\right),\mathcal{U}\right).\label{eq:reduced_sum}
\end{align}
\end{thm}

\begin{proof}
Let $d\in\mathcal{D}_{X,\alpha}^{\pm}$. For any $T$ which is one
of the factors of $\widehat{T}$ we expand the formula from Theorem
(\ref{thm:f_T_formula}) as
\begin{align*}
f_{T}\left(d\right)= & \frac{1}{2^{n}}\frac{1}{\left|\mathrm{Aut}_{H,\widehat{T}}\left(G,\Phi(G)\right)\right|}\sum_{N\in\mathcal{N^{\pm}}(T)}\sum_{\Upsilon=1}^{r}\sum_{d=2^{\alpha}\prod D_{u}}\prod_{u,v}\left(\frac{D_{u}}{D_{v}}\right)^{\Phi\left(u,v\right)}\prod_{u}\left(\frac{-1}{D_{u}}\right)^{\lambda_{1}^{N}\left(u\right)}\\
 & \times\prod_{u}\left(\frac{2^{\alpha}}{D_{u}}\right)^{\psi_{\Upsilon}\left(u\right)}\left(1+\prod_{u}\left(\frac{D_{u}}{2}\right)^{\gamma_{\Upsilon,1}\left(u\right)}\right).
\end{align*}
where the second sum only appears when $\alpha\neq0$ and the third
sum is over factorizations of $d$ into positive integers $D_{u}$
such that $\left(D_{u}\right)\in\Lambda_{8}\left(N,l\right)$ and
the functions appearing in the exponents are defined in (\ref{eq:lambda_def})
with $k=1$. Note the $D_{u}$ were allowed to be negative and divisible
by 2 previously, hence we introduce $\lambda_{1}^{N}$ to keep track
of the negative signs and $\gamma_{\Upsilon,1},\psi_{\Upsilon}$ to
keep track of the factor $2^{\alpha}$.

For $\widehat{T}=\prod_{i=1}^{k}T_{i}$ we obtain, using the same
procedure as in \cite[p.470-472]{fk1}

\begin{align*}
f_{\widehat{T}}\left(d\right)= & \frac{1}{2^{kn}}\frac{1}{\left|\mathrm{Aut}_{H,\widehat{T}}\left(G,\Phi(G)\right)\right|}\sum_{\Gamma\subset\left\{ 1,\ldots,k\right\} }\sum_{N\in\mathcal{N^{\pm}}(\widehat{T})}\sum_{\Upsilon\in\prod_{j}\left[r\left(j\right)\right]}\sum_{d=2^{\alpha}\prod D_{u}}\prod_{u,v}\left(\frac{D_{u}}{D_{v}}\right)^{\Phi_{k}\left(u,v\right)}\\
 & \times\prod_{u}\left(\frac{-1}{D_{u}}\right)^{\lambda_{k}^{N}\left(u\right)}\left(\frac{D_{u}}{2}\right)^{\gamma_{\Upsilon,\Gamma}\left(u\right)}\left(\frac{2^{\alpha}}{D_{u}}\right)^{\psi_{\Upsilon}\left(u\right)}
\end{align*}
where $\left(D_{u}\right)\in\Lambda_{8}\left(N,\Upsilon\right)$ and
$\lambda_{k}^{N},\gamma_{\Upsilon,\Gamma},\psi_{\Upsilon}$ are defined
in the theorem statement. The sum over $\Gamma\subset\left[k\right]$
comes from expanding the product $\prod_{j=1}^{k}(1+\prod_{u}\left(\frac{D_{u}}{2}\right)^{\gamma_{\Upsilon}\left(u\right)})$.

An analysis of character sums and their cancellation, and a removal
of the congruence conditions on the $D_{u}$ as in \cite[p.473-478]{fk1}
gives
\begin{align}
\sum_{d\in\mathcal{D}_{X,\alpha}^{\pm}}\frac{f_{\widehat{T}}\left(d\right)}{a^{\omega\left(d\right)}}= & \frac{1}{2^{kn}}\frac{1}{4^{\mathbf{c}}}\frac{1}{\left|\mathrm{Aut}_{H,\widehat{T}}\left(G,\Phi(G)\right)\right|}\left[\sum_{N,\Gamma,\Upsilon}S_{N,\Gamma,\Upsilon}\right]\left[\sum_{\left(D_{u}\right)}\frac{\mu^{2}\left(2\prod D_{u}\right)}{a^{\omega\left(\prod D_{u}\right)}}\right]\label{eq:sum}\\
 & +o\left(X(\log X)^{(\mathbf{c}/a)-1}\right)\nonumber 
\end{align}
where the range of the last sum is over factorizations into $\mathbf{c}$-tuples
of coprime positive integers up to $X$ and $S_{N,\Gamma,\Upsilon}$
is as given by (\ref{eq:reduced_sum}) in the statement. The factor
of $4^{-\mathbf{c}}$ comes from removing the congruence conditions
on the remaining $\mathbf{c}$ variables in the last sum in (\ref{eq:sum})
as in the argument from \cite[p.480-482]{fk1} (Lemma 19 is modified
to give a factor of $1/4$ instead of $1/2$ since we are considering
odd congruence classes mod 8 instead of mod 4).

Since the number of ways of writing any integer $x<X$ as a product
of $\mathbf{c}$ integers is $\mathbf{c}^{\omega\left(x\right)}$
we have
\begin{align*}
\sum_{\left(D_{u}\right)}\frac{\mu^{2}\left(2\prod D_{u}\right)}{a^{\omega\left(\prod D_{u}\right)}} & =\left[\sum_{x<X}\mu^{2}\left(2x\right)\left(\frac{\mathbf{c}}{a}\right)^{\omega\left(x\right)}\right]+o\left(X(\log X)^{(\mathbf{c}/a)-1}\right)\\
 & =\frac{4}{\pi^{2}}X(\log X)^{(\mathbf{c}/a)-1}+o\left(X(\log X)^{(\mathbf{c}/a)-1}\right)
\end{align*}
 and the theorem follows.
\end{proof}

\section{\label{sec:Computing-the-moments}Computing the moments}

Throughout this section we will prove results about $\Gamma_{\mathcal{M}(\widehat{T})}$
(defined as in Theorem \ref{thm:asymptotic_analysis} with $\mathcal{M}(\widehat{T})$
in place of $\mathbf{U}(\widehat{T})$) when $\widehat{T}=\prod_{i=1}^{k}T_{i}$
with all $T_{i}$ not complete bipartite. We will apply them in the
case when $\mathbf{U}(\widehat{T})=\mathcal{M}(\widehat{T})$. Let
$c_{\widehat{T}}=\prod_{i=1}^{k}c_{T_{i}}$ be the cardinality of
any element of $\mathcal{M}(\widehat{T})$.

\subsection{The congruence conditions}

We will first reduce to a simpler set of congruence conditions. Let
$\widehat{T}=\prod_{i=1}^{k}T_{i}$ where all $T_{i}$ are not complete
bipartite. Let $N\in\mathcal{N^{\pm}}(\widehat{T})$ and $\mathcal{U}\in\mathcal{M}(\widehat{T})$.
Let $\Lambda\left(\mathcal{U},N\right)$ be the set of tuples of congruence
classes modulo 4 such that for each $1\le j\le k$ and $i\in\left[r\left(j\right)\right]$
\begin{equation}
\prod_{u,v}h_{u}h_{v}\equiv\begin{cases}
-1 & i\in N_{j}\\
1 & i\notin N_{j}
\end{cases}\label{eq:general cong cond-1}
\end{equation}
 where the product is over all $u,v\in\widehat{U}$ with $u_{j}=2i-1,v_{j}=2i$.
\begin{lem}
Suppose $\mathbf{U}(\widehat{T})=\mathcal{M}(\widehat{T})$. In the
notation of Theorem \ref{thm:asymptotic_analysis} we have
\[
\Gamma{}_{\mathcal{M}(\widehat{T})}=2^{c_{\widehat{T}}}\left|\widehat{T}\right|\sum_{N\in\mathcal{N^{\pm}}(\widehat{T})}\sum_{\mathcal{U}\in\mathcal{M}(\widehat{T})}\sum_{\left(h_{u}\right)\in\Lambda\left(\mathcal{U},N\right)}\chi_{1}\left(\left(h_{u}\right),\mathcal{U}\right)\chi_{2}\left(\left(h_{u}\right),\mathcal{U}\right)
\]
 where $\chi_{i}$ is defined in (\ref{eq:chi def}).
\end{lem}

\begin{proof}
Since $\chi_{1}$ and $\chi_{2}$ in (\ref{eq:reduced_sum}) only
depend on the class of $h_{u}$ modulo $4$, we can let $\left(h_{u}'\right)$
denote the reduction of $\left(h_{u}\right)$ modulo 4 and factor
the sum 
\begin{align*}
\sum_{\Gamma\subset\left[k\right]}\sum_{\Upsilon\in\prod_{j}\left[r\left(j\right)\right]}S_{N,\Gamma,\Upsilon} & =\sum_{\mathcal{U}\in\mathcal{M}(\widehat{T})}\sum_{\left(h_{u}'\right)}\chi_{1}\left(\left(h_{u}'\right),\mathcal{U}\right)\chi_{2}\left(\left(h_{u}'\right),\mathcal{U}\right)\\
 & \qquad\times\sum_{\Gamma\subset\left[k\right]}\sum_{\Upsilon\in\prod_{j}\left[r\left(j\right)\right]}\sum_{\left(h_{u}\right)}\chi_{3}\left(\left(h_{u}\right),\mathcal{U}\right).
\end{align*}
Note in the above sum $\left(h_{u}\right)$ depends on $\left(h_{u}'\right)$
since $h_{u}\equiv h_{u}'\mathrm{mod}4$ and hence there are two possible
choices for the class of $h_{u}$ modulo 8. If $h_{u}'\equiv1\left(4\right)$
then $h_{u}\equiv1$or 5 modulo 8. If $h_{u}\equiv1\left(8\right)$
then $\frac{h_{u}^{2}-1}{8}$ is even, and if $h_{u}\equiv5$ then
$\frac{h_{u}^{2}-1}{8}$ is odd. Hence both pairities are possible.
Similarly when $h_{u}'\equiv3\left(4\right)$.

Fix $\Upsilon$ and let 
\[
\kappa\left(\Upsilon\right)=\sum_{\Gamma\subset\left[k\right]}\sum_{(h_{u})}\prod_{u\in\mathcal{U}}\left(-1\right)^{\frac{h_{u}^{2}-1}{8}(\gamma_{\Upsilon,\Gamma}\left(u\right)+\psi_{\Upsilon}\left(u\right)\alpha)}.
\]
Using the above observation we can replace the sum over $\left(h_{u}\right)$
as follows
\begin{align*}
\kappa\left(\Upsilon\right) & =\sum_{\Gamma\subset\left[k\right]}\prod_{u\in\mathcal{U}}\left(1+\left(-1\right)^{(\gamma_{\Upsilon,\Gamma}\left(u\right)+\psi_{\Upsilon}\left(u\right)\alpha)}\right).
\end{align*}
Recall the definitions $\gamma_{\Upsilon,\Gamma}\left(u\right)=\sum_{i\in\Gamma}\Phi\left(u_{i},2\Upsilon_{i}\right)$
and $\psi_{\Upsilon}\left(u\right)=\sum_{i=1}^{k}\Phi\left(2\Upsilon_{i},u_{i}\right)$.

It follows from the definition of $\Phi$ and $\mathcal{M}(\widehat{T})$
for $\widehat{T}$ not complete bipartite that there is only one $\Gamma$
in the above sum for which $\gamma_{\Upsilon,\Gamma}\left(u\right)+\psi_{\Upsilon}\left(u\right)\alpha$
is even for all $u\in\mathcal{U}$. In particular if $\alpha$ is
even $\gamma_{\Upsilon,\Gamma}\left(u\right)$ is even for all $u\in\mathcal{U}$
only when $\Gamma\neq\varnothing$. If $\alpha$ is odd then $\gamma_{\Upsilon,\Gamma}\left(u\right)+\psi_{\Upsilon}\left(u\right)$
is even for all $u\in\mathcal{U}$ only when $\Gamma=\left\{ i\in\left[k\right]\mid u_{i}\text{ is even for all }u\in\mathcal{U}\right\} $.
In both cases we get $\kappa\left(\Upsilon\right)=2^{c_{\widehat{T}}}$
for all $\Upsilon$. Finally note that $\left|\prod_{j=1}^{k}r\left(j\right)\right|=\left|\widehat{T}\right|$.
This completes the proof.
\end{proof}
Next we rewrite the congruence conditions (\ref{eq:general cong cond})
in linear algebraic terms.

Let $\mathcal{U}\in\mathcal{M}(\widehat{T})$. Let $\left(h_{u}\right)\in\Lambda\left(\mathcal{U},N\right)$.
These congruence conditions can be encoded as vectors lying in a certain
coset of a subspace of $\mathbb{F}_{2}^{\left|\mathcal{U}\right|}$.
For any $\left(h_{u}\right)$ define the corresponding element $x\in\mathbb{F}_{2}^{\left|\mathcal{U}\right|}$
by 
\begin{equation}
x_{u}=\begin{cases}
1 & \text{if }h_{u}=-1\\
0 & \text{if }h_{u}=1.
\end{cases}\label{eq:congruence_correspondence}
\end{equation}
 Let $r_{k}=\sum_{j=1}^{k}r\left(j\right)$. Then there is a $r_{k}\times\left|\mathcal{U}\right|$
matrix $M_{\widehat{T}}$ over $\mathbb{F}_{2}$ and vector $\mathbf{h}\in\mathbb{F}_{2}^{r_{k}}$
such that the set of elements $x\in\mathbb{F}_{2}^{\left|\mathcal{U}\right|}$
corresponding to $\left(h_{u}\right)\in\Lambda\left(\mathcal{U},N\right)$
is equal to the set of solutions of the system $M_{\widehat{T}}x=\mathbf{h}$
(that is, equal to a coset of $\ker M_{T}$). 

\subsection{Factoring the main term}

We begin by investigating how the congruence conditions (\ref{eq:general cong cond})
behave under products of maximal unlinked sets.

Factor $\widehat{T}=\widehat{T}_{1}\times\widehat{T}_{2}$. Suppose
$\mathcal{U}\in\mathcal{M}(\widehat{T})$ factors as $\mathcal{U}_{1}\times\mathcal{U}_{2}$.
We can correspondingly factor $N=N_{1}\times N_{2}$. Define a map
$\lambda_{1}:\Lambda\left(\mathcal{U},N\right)\longrightarrow\Lambda\left(\mathcal{U}_{1},N_{1}\right)$
by

\[
[\lambda_{1}\left(\left(h_{u}\right)\right)]_{u_{1}}=\prod_{\substack{u\in\mathcal{U}_{1}\times\mathcal{U}_{2}\\
u=\left(u_{1},u_{2}\right)
}
}h_{u}
\]
 for each $u_{1}\in\mathcal{U}_{1}$. It follows from the conditions
(\ref{eq:general cong cond}) that this map is well-defined. Let $p_{1}$
be the projection onto the coordinates in $\left[k\right]$ corresponding
to $\widehat{T}_{1}$. If we let $x_{1}\in\mathbb{F}_{2}^{r_{1}}$
represent $\lambda_{1}\left(\left(h_{u}\right)\right)$ then $M_{\widehat{T}_{1}}x_{1}=p_{1}\left(\mathbf{h}\right)$.
Similarly define $\lambda_{2}:\Lambda\left(\mathcal{U},N\right)\longrightarrow\Lambda\left(\mathcal{U}_{2},N_{2}\right)$.
\begin{lem}
\label{lem:factor_lambda}In the above notation the map $\lambda=\lambda_{1}\times\lambda_{2}:\Lambda\left(\mathcal{U},N\right)\longrightarrow\Lambda\left(\mathcal{U}_{1},N_{1}\right)\times\Lambda\left(\mathcal{U}_{2},N_{2}\right)$
is a $2^{\left|\mathcal{U}_{1}\right|\left|\mathcal{U}_{2}\right|-\left|\mathcal{U}_{1}\right|-\left|\mathcal{U}_{2}\right|+1}$
to 1 function.
\end{lem}

\begin{proof}
Under the correspondence (\ref{eq:congruence_correspondence}) the
map $\lambda_{1}$ becomes
\[
[\lambda_{1}\left(x\right)]_{u_{1}}=\sum_{\substack{u\in\mathcal{U}_{1}\times\mathcal{U}_{2}\\
u=\left(u_{1},u_{2}\right)
}
}x_{u}.
\]
 Put an arbitrary ordering on the elements of $\mathcal{U}_{1}$ and
$\mathcal{U}_{2}$, and the lexicographic ordering on $\mathcal{U}_{1}\times\mathcal{U}_{2}$.
Then it is easy to see that $\lambda_{1}$ is given by a matrix $M_{1}$
of size $\left|\mathcal{U}_{1}\right|\times\left|\mathcal{U}\right|$
of full rank. Similarly $\lambda_{2}$ is given by $M_{2}$ of size
$\left|\mathcal{U}_{2}\right|\times\left|\mathcal{U}\right|$ and
full rank. Thus $\lambda$ is the matrix $M_{\lambda}=\left[M_{1}^{T},M_{2}^{T}\right]^{T}$
which has size $(\left|\mathcal{U}_{1}\right|+\left|\mathcal{U}_{2}\right|)\times\left|\mathcal{U}\right|$
and rank $\left|\mathcal{U}_{1}\right|+\left|\mathcal{U}_{2}\right|-1$.
It follows that $\dim\ker M_{\lambda}=\left|\mathcal{U}_{1}\right|\left|\mathcal{U}_{2}\right|-\left|\mathcal{U}_{1}\right|-\left|\mathcal{U}_{2}\right|+1$.

Finally we need to show that $\Lambda\left(\mathcal{U},N\right)+\ker M_{\lambda}\subset\Lambda\left(\mathcal{U},N\right)$
which is equivalent to $\ker M_{\lambda}\subset\ker M_{\widehat{T}}$
by definition of $\Lambda\left(\mathcal{U},N\right)$ as a coset of
$\ker M_{\widehat{T}}$. We have $\Lambda\left(\mathcal{U}_{i},N_{i}\right)=p_{i}\left(\mathbf{h}\right)+\ker M_{\widehat{T}_{i}}$.
This implies that $M_{\widehat{T}_{i}}M_{i}x=p_{i}\left(\mathbf{h}\right)=p_{i}\left(M_{\widehat{T}}x\right)$
for any $x\in\mathbb{F}_{2}^{\left|\mathcal{U}\right|}$.

Let $M'$ be the block diagonal matrix constructed from $M_{\widehat{T}_{1}}$
and $M_{\widehat{T}_{2}}$. Then for any $x\in\Lambda\left(\mathcal{U},N\right)$
and $y\in\ker M_{\lambda}$ we have $M_{\widehat{T}}y=M'M_{\lambda}y=0$.
Hence $\ker M_{\lambda}\subset\ker M_{\widehat{T}}$ as required.
This completes the proof.
\end{proof}
By defintion each $\mathcal{U}\in\mathcal{M}(\widehat{T})$ factors
into $\mathcal{U}=\mathcal{U}_{1}\times\mathcal{U}_{2}\in\mathcal{M}(\widehat{T})$.
We can factor $N=N_{1}\times N_{2}$ and let $x\in\Lambda\left(\mathcal{U},N\right)$.
Suppose $\chi\left(\mathcal{U},x\right)$ is a function satisfying
\begin{equation}
\chi\left(\mathcal{U},x\right)=\chi\left(\mathcal{U}_{1},\lambda_{1}\left(x\right)\right)\chi\left(\mathcal{U}_{2},\lambda_{2}\left(x\right)\right)\label{eq:chi character}
\end{equation}
 for all $\mathcal{U}\in\mathcal{M}(\widehat{T})$. Define 
\begin{equation}
\gamma\left(\mathcal{U},N\right)=\sum_{x\in\Lambda\left(\mathcal{U},N\right)}\chi\left(\mathcal{U},x\right)\label{eq:lambda as chi}
\end{equation}
 and $\gamma_{\mathcal{M}(\widehat{T}),N}=\sum_{\mathcal{U\in\mathcal{M}}(\widehat{T})}\gamma\left(\mathcal{U},N\right)$.
Then $\Gamma_{\mathcal{M}(\widehat{T})}=2^{c^{k}}\left|\widehat{T}\right|\sum_{N\in\mathcal{N}^{\pm}(\widehat{T})}\gamma_{\mathcal{M}(\widehat{T}),N}$.
\begin{lem}
\label{lem:factor_gamma}Let $\widehat{T}=\widehat{T}_{1}\times\widehat{T}_{2}$
and $N=N_{1}\times N_{2}$. Let $c_{\widehat{T}_{i}}$ be the cardinality
of any element of $\mathcal{M}(\widehat{T}_{i})$. Then
\begin{align*}
\gamma_{\mathcal{M}(\widehat{T}),N} & =2^{c_{\widehat{T}_{1}}c_{\widehat{T}_{2}}-c_{\widehat{T}_{1}}-c_{\widehat{T}_{2}}+1}\gamma_{\mathcal{M}(\widehat{T}_{1}),N_{1}}\gamma_{\mathcal{M}(\widehat{T}_{2}),N_{2}}.
\end{align*}
\end{lem}

\begin{proof}
Write $\mathcal{U}=\mathcal{U}_{1}\times\mathcal{U}_{2}$. We have
$\left|\mathcal{U}_{1}\right|=c_{\widehat{T}_{1}}$ and $\left|\mathcal{U}_{2}\right|=c_{\widehat{T}_{2}}$.
Let $\theta=c_{\widehat{T}_{1}}c_{\widehat{T}_{2}}-c_{\widehat{T}_{1}}-c_{\widehat{T}_{2}}+1$.
Hence by Lemma \ref{lem:factor_lambda} we can factor the resulting
sum
\begin{align*}
\gamma_{\mathcal{M}(\widehat{T}),N} & =\sum_{\mathcal{U}\in\mathcal{M}\left(\widehat{T}\right)}\sum_{x\in\Lambda\left(\mathcal{U},N\right)}\chi\left(\mathcal{U},x\right)\\
 & =2^{\theta}\sum_{\mathcal{U}=\mathcal{U}_{1}\times\mathcal{U}_{2}}\sum_{\left(x_{1},x_{2}\right)\in\Lambda\left(\mathcal{U}_{1},N_{1}\right)\times\Lambda\left(\mathcal{U}_{2},N_{2}\right)}\chi\left(\mathcal{U}_{1},x_{1}\right)\chi\left(\mathcal{U}_{2},x_{2}\right)\\
 & =2^{\theta}\sum_{\mathcal{U}=\mathcal{U}_{1}\times\mathcal{U}_{2}}\left(\sum_{x_{1}\in\Lambda\left(\mathcal{U}_{1},N_{1}\right)}\chi\left(\mathcal{U}_{1},x_{1}\right)\right)\left(\sum_{x_{2}\in\Lambda\left(\mathcal{U}_{2},N_{2}\right)}\chi\left(\mathcal{U}_{2},x_{2}\right)\right)\\
 & =2^{\theta}\gamma_{\mathcal{M}(\widehat{T}_{1}),N_{1}}\gamma_{\mathcal{M}(\widehat{T}_{2}),N_{2}}.
\end{align*}
\end{proof}

\subsection{Computing the main term}

Let $s\left(j\right)=s_{1}\left(j\right)+s_{2}\left(j\right)$ be
the number of connected components of $T_{j}$ where $s_{2}\left(j\right)$
is the number of components of size 1. Recall $\widehat{T}=\prod_{i=1}^{k}T_{i}$
and $\mathcal{M}(\widehat{T})=\prod_{j=1}^{k}\mathcal{T}_{j}$. We
recall some notation from previous sections. For $j\in\left[k\right]$
we denote by $C_{j}$ the set of fixed points for the action of $T_{j}$
on itself. Hence $\left|C_{j}\right|=s_{2}\left(j\right)$.

Let $\mathcal{C}_{j,i}$ be the $i$th connected component of $T_{j}$.
For $i=1.\ldots,s_{1}\left(j\right)$ we let $A_{j,i},B_{j,i}$ be
the $i$-th connected component of $A_{j},B_{j}$ respectively.

For any $\mathcal{U}\in\mathcal{M}(\widehat{T})$ and $x\in\Lambda\left(\mathcal{U},N\right)$
define $\chi\left(\mathcal{U},x\right)=\chi_{1}\left(\mathcal{U},x\right)\chi_{2}(\mathcal{U},x)$
(see \ref{eq:chi def}). Then $\chi$ satisfies (\ref{eq:chi character}).

We want to compute $\Gamma_{\mathcal{M}(\widehat{T})}=2^{c^{k}}\left|\widehat{T}\right|\sum_{N\in\mathcal{N}^{\pm}(\widehat{T})}\gamma_{\mathcal{M}(\widehat{T}),N}$
where
\[
\gamma_{\mathcal{M}(\widehat{T}),N}=\sum_{\mathcal{U}\in\mathcal{M}(\widehat{T})}\sum_{x\in\Lambda\left(\mathcal{U},N\right)}\chi\left(\mathcal{U},x\right).
\]

\begin{lem}
\label{prop:gamma_type2}Let $\omega=c_{\widehat{T}}-\sum_{j\in\left[k\right]}c_{T_{j}}+\sum_{j\in\left[k\right]}s\left(j\right)+\left(k-1\right)$.
For any $N$ let $m_{j,i}=\left|N_{j}\cap\mathrm{supp}B_{j,i}\right|$.
Then we have 
\[
\gamma_{\mathcal{M}(\widehat{T}),N}=\begin{cases}
2^{\omega} & \text{if \ensuremath{m_{j,i}\equiv0,1} mod4 for all \ensuremath{i,j}}\\
0 & \text{else}
\end{cases}.
\]
\end{lem}

\begin{proof}
For the remainder of the proof we will write $\widehat{T}=\prod_{j=1}^{k}T_{j}$,
and let $\widehat{T}_{i}=\prod_{j\ge i}T_{j}$. Any $\mathcal{U}\in\mathcal{M}(\widehat{T})$
factors as $\mathcal{U}=\prod_{j=1}^{k}\mathcal{W}_{j}$ with each
$\mathcal{W}_{j}$ maximal unlinked with full support on $T_{j}$.
Hence by repeated application of Proposition \ref{lem:factor_gamma}
we have
\begin{align}
\gamma_{\mathcal{M}(\widehat{T}),N} & =2^{c_{\widehat{T}}-\sum_{j}c_{T_{j}}+\left(k-1\right)}\prod_{j=1}^{k}\sum_{\mathcal{W}_{j}}\sum_{x_{j}\in\Lambda\left(\mathcal{W}_{j},N_{j}\right)}\chi\left(\mathcal{W}_{j},x_{j}\right)\label{eq:gamma2_firstline}
\end{align}
 where we are using that 
\begin{align*}
\sum_{i=1}^{k-1}c_{\widehat{T}_{i}}-c_{\widehat{T}_{i+1}}-c_{T_{i}}+1 & =c_{\widehat{T}}-\sum_{j}c_{T_{j}}+\left(k-1\right).
\end{align*}
 Note that $\left|\Lambda\left(\mathcal{W}_{j},N_{j}\right)\right|=2^{s_{2}\left(j\right)}$
but the value of $e\left(\mathcal{W}_{j},\cdot\right)+f_{N_{j}}\left(\mathcal{W}_{j},\cdot\right)$
is constant on $\Lambda\left(\mathcal{W}_{j},N_{j}\right)$ (since
$\Phi$ and $\lambda_{N_{j}}$ are zero on $C_{j}$).

Let $\mathcal{W}_{j}'=\mathcal{W}_{j}\backslash C_{j}$. Then $\left|\Lambda\left(\mathcal{W}_{j}',N_{j}\right)\right|=1$.
Hence (\ref{eq:gamma2_firstline}) becomes
\begin{equation}
2^{\omega_{0}}\prod_{j=1}^{k}\sum_{\mathcal{W}_{j}'}\chi\left(\mathcal{W}_{j}',x_{j}\right)\label{eq:gamma2_secondline}
\end{equation}
 where $\omega_{0}=2^{c_{\widehat{T}}-\sum_{j}c_{T_{j}}+\left(k-1\right)+\sum_{j}s_{2}\left(j\right)}$.

Note that for every $j$, if $u\in A_{j,i}$ then $\Phi\left(v,u\right)=0$
for all $v\in\mathcal{W}_{j}$. Furthemore for every $u,v\in\mathcal{W}_{j}$
we have $\Phi\left(u,v\right)=\Phi\left(v,u\right)$. Hence it follows
from the definition of the sets $\mathcal{W}_{j}=\left\{ \left(\cup_{i=1}^{s_{1}\left(j\right)}W_{i}\right)\cup C_{j}\mid W_{i}=A_{j,i}\text{ or }B_{j,i}\right\} $
that (\ref{eq:gamma2_secondline}) factors as

\[
2^{\omega_{0}}\prod_{j=1}^{k}\prod_{i=1}^{s_{1}\left(j\right)}\left(1+\chi\left(B_{j,i},x_{j}\right)\right)
\]
where by abuse of notation we are denoting by $x_{j}$ its projection
onto the coordinates corresponding to the $i$th connected component
$\mathcal{C}_{j,i}$.

We now turn to computing $\chi\left(B_{j,i},x_{j}\right)$ which is
determined by the pairity of 
\[
\sum_{\left\{ u,v\right\} \subset B_{j,i}}\Phi\left(u,v\right)x_{u}x_{v}+\sum_{u\in B_{j,i}}\lambda^{N_{j}}\left(u\right)x_{u}.
\]
 Note that $\Phi\left(u,v\right)=1$ for all $u,v\in B_{j}$ if $u\neq v$.
It follows from $\left(\ref{eq:congruenceconditions_k}\right)$ that
for any $u=2l\in B_{j,i}$ 
\[
x_{u}=\begin{cases}
1 & \text{if }l\in N_{j}\\
0 & \text{if }i\notin N_{j}.
\end{cases}
\]
It follows from this that $\sum_{\left\{ u,v\right\} \subset B_{j,i}}\Phi\left(u,v\right)x_{u}x_{v}=\binom{m_{j,i}}{2}$
which is even when $m_{j,i}\equiv0,1\mod4$. Furthermore for $u=2l\in B_{j,i}$
we have
\[
\lambda^{N_{j}}\left(u\right)=\begin{cases}
1 & \text{if }\left|N_{j}\cap S_{l}\right|\text{ odd}\\
0 & \text{else}
\end{cases}
\]
 and hence 
\begin{align*}
\sum_{u\in B_{j,i}}\lambda^{N_{j}}\left(u\right)x_{u} & =\begin{cases}
m_{j,i} & \text{if }m_{j,i}\text{ even}\\
0 & \text{else}
\end{cases}
\end{align*}
which is always even. Thus if $m_{j,i}\equiv0,1\mod4$ for all $i$
and $j$ then
\begin{align*}
\gamma_{\mathcal{M}(\widehat{T}),N} & =2^{\omega_{0}}\prod_{j=1}^{k}\prod_{i=1}^{s_{1}\left(j\right)}\left(1+\chi\left(B_{j,i},x_{j}\right)\right)\\
 & =2^{\omega_{0}+\sum_{j}s_{1}\left(j\right)}
\end{align*}
where $\omega_{0}=2^{c_{\widehat{T}}-\sum_{j}c_{T_{j}}+\left(k-1\right)+\sum_{j}s_{2}\left(j\right)}$.
If $m_{j,i}\neq0,1\mod4$ for any $i$ or $j$ then $\gamma_{\mathcal{M}(\widehat{T}),N}=0$.
Since $s\left(j\right)=s_{1}\left(j\right)+s_{2}\left(j\right)$ for
all $j$, this completes the proof.
\end{proof}
Recall for the next proposition that we defined $\mathcal{C}_{j,i}$
be the $i$th connected component of $T_{j}$.
\begin{prop}
\label{prop:Gamma type 2}Let $a=0,1$ in the case of positive (respectively
negative) discriminants. We have
\begin{equation}
\Gamma_{\mathcal{M}(\widehat{T})}=2^{\omega}\left|\widehat{T}\right|\prod_{j=1}^{k}\left[\sum_{m\equiv a\left(2\right)}^{r\left(j\right)}\sum_{\substack{m_{i}\equiv0,1\left(4\right)\\
\sum_{i=1}^{s\left(j\right)}m_{i}=m
}
}\prod_{i=1}^{s\left(j\right)}\binom{\left|\mathcal{C}_{j,i}\right|}{m_{i}}\right]\label{eq:gamma_2}
\end{equation}
where $\omega=2c_{\widehat{T}}-\sum_{j\in\left[k\right]}c_{T_{j}}+\sum_{j\in\left[k\right]}s\left(j\right)+\left(k-1\right)$.
\end{prop}

\begin{proof}
We begin with the formula from Lemma \ref{prop:gamma_type2} and sum
over $N\in\mathcal{N}^{\pm}(\widehat{T})$ satisfying the appropriate
conditions. For any $N$ let $m_{j,i}=\left|N_{j}\cap\mathrm{supp}B_{j,i}\right|$.
Then we have $\sum_{N}\gamma_{\mathcal{M}(\widehat{T}),N}=2^{\omega}\sideset{}{_{N}^{\prime}}\sum1$
where the prime indicates summation over $N$ satisfying $\left|N_{j}\right|\equiv a\left(2\right)$
and $m_{j,i}\equiv0,1\left(4\right)$. If we let $Q_{j}^{\pm}=\left|\left\{ N_{j}\subset T_{j}\mid\left|N_{j}\right|\equiv a\left(2\right),m_{i,j}\equiv0,1\left(4\right)\right\} \right|$
then we have $\sideset{}{_{N}^{\prime}}\sum1=\prod_{j=1}^{k}Q_{j}^{\pm}.$
One can further compute
\begin{equation}
Q_{j}^{\pm}=\sum_{m\equiv a\left(2\right)}^{r\left(j\right)}\sum_{\substack{m_{i}\equiv0,1\left(4\right)\\
\sum_{i=1}^{s\left(j\right)}m_{i}=m
}
}\prod_{i=1}^{s\left(j\right)}\binom{\left|\mathcal{C}_{i,j}\right|}{m_{i}}.\label{eq:Q}
\end{equation}
The outer summation corresponds to choices of possible sizes of the
set $N_{j}$ and the inner sum to the choices of elements of $T_{j}$
contained in $N_{j}$.
\end{proof}

\subsection{The $k$th moment}

We can combine the results of the previous two subsections.
\begin{prop}
\label{prop:T_moments}Let $\widehat{T}=\prod_{i=1}^{k}T_{i}$ be
a product of generating sets and suppose $T_{i}$ is not complete
bipartite for all $i$. Suppose $\mathbf{U}(\widehat{T})=\mathcal{M}(\widehat{T})$.
Then for all positive integers $k$
\[
\lim_{X\longrightarrow\infty}\frac{\sum_{d\in\mathcal{D}_{X,\alpha}^{\pm}}\frac{f_{\widehat{T}}\left(d\right)}{c_{\widehat{T}}^{\omega\left(d\right)}}}{\sum_{d\in\mathcal{D}_{X,\alpha}^{\pm}}1}=\frac{2^{k+\sum_{j\in\left[k\right]}\left(s\left(j\right)-c_{T_{j}}\right)-kn}\left|\widehat{T}\right|}{\left|\mathrm{Aut}_{H,\widehat{T}}\left(G/\Phi(G)\right)\right|}\prod_{j=1}^{k}Q_{j}^{\pm}
\]
where
\begin{itemize}
\item $Q_{j}^{\pm}$ is given by (\ref{eq:Q}),
\item $s\left(j\right)$ is the number of connected components of $T_{j}$.
\end{itemize}
Furthermore if $a>c_{\widehat{T}}$ then
\[
\lim_{X\longrightarrow\infty}\frac{\sum_{d<X}\frac{f_{\widehat{T}}\left(d\right)}{a^{\omega\left(d\right)}}}{\sum_{d<X}1}=0.
\]
\end{prop}

\begin{proof}
By Theorem \ref{thm:asymptotic_analysis} we have 
\[
\sum_{d<X}\frac{f_{\widehat{T}}\left(d\right)}{a^{\omega\left(d\right)}}=\frac{\Gamma_{\mathbf{U}(\widehat{T})}}{4^{\mathbf{c}}2^{kn}\left|\mathrm{Aut}_{H,\widehat{T}}\left(G/\Phi(G)\right)\right|}\frac{4}{\pi^{2}}X(\log X)^{(\mathbf{c}/a)-1}+o\left(X(\log X)^{(\mathbf{c}/a)-1}\right).
\]
By assumption $\Gamma_{\mathbf{U}(\widehat{T})}=\Gamma_{\mathcal{M}(\widehat{T})}$
which is given by Proposition \ref{prop:Gamma type 2}. We defined
$\mathbf{c}=\max_{\mathcal{U}\in\mathbf{U}\left(\widehat{T}\right)}\left|\mathcal{U}\right|$
which by assumption is given by $\mathbf{c}=c_{\widehat{T}}.$ Thus
setting $a=c_{\widehat{T}}$ and noting the well known formula for
the number of real quadratic fields $\sum_{d<X}1\sim\frac{2}{\pi^{2}}X$
gives
\begin{align*}
\lim_{X\longrightarrow\infty}\frac{\sum_{d\in\mathcal{D}_{X,\alpha}^{\pm}}\frac{f_{\widehat{T}}\left(d\right)}{a^{\omega\left(d\right)}}}{\sum_{d\in\mathcal{D}_{X,\alpha}^{\pm}}1} & =\frac{2^{\omega+1}\left|\widehat{T}\right|}{4^{c_{\widehat{T}}}2^{kn}\left|\mathrm{Aut}_{H,\widehat{T}}\left(G/\Phi(G)\right)\right|}\prod_{j=1}^{k}Q_{j}^{\pm}X(\log X)^{(\mathbf{c}/a)-1}\\
 & +o\left(X(\log X)^{(\mathbf{c}/a)-1}\right)
\end{align*}
 and $\omega+1-2c_{\widehat{T}}-kn=k+\sum_{j\in\left[k\right]}\left(s\left(j\right)-c_{T_{j}}\right)-kn$.
Clearly if $a>c_{\widehat{T}}$ then the above limit is 0.
\end{proof}

\section{\label{sec:Determining-the-distribution}Determining the distribution}

\subsection{\label{subsec:Preliminaries-for-distributions}Preliminaries for
distributions of functions on quadratic fields}

We start by recalling some preliminaries on measures and distributions.
In particular we specify what we mean by the distribution of a function
on the set of quadratic fields.

Let $\mathcal{K}$ be the set of real or imaginary quadratic fields.
Let $\Pi$ be the $\sigma$-algebra on $\mathcal{K}$ consisting of
all subsets and define the measurable space $\mathcal{M}=\left(\mathcal{K},\Pi\right)$.

For any positive $X\in\mathbb{R}$ let $\mu_{X}$ be the measure on
$\mathcal{M}$ with mass supported uniformly on $\left\{ K\in\mathcal{K}\mid\left|D_{K}\right|<X\right\} $.
Then for any measurable space $\mathcal{I}=\left(\mathcal{X},\Sigma\right)$
and function $g:\mathcal{M}\longrightarrow\mathcal{I}$ we have a
measure $\mu_{g,X}=g_{*}\mu_{X}$ on $\mathcal{I}$ (note any function
is measurable since $\Pi$ is the power set of $\mathcal{K}$). Thus
for any measurable set $U\in\Sigma$, $\mu_{g,X}\left(U\right)$ is
the proportion of quadratic fields in $K\in\mathcal{K}$ with $\left|D_{K}\right|<X$
such that $g\left(K\right)\in U$.

Now suppose $\mathcal{X}$ is also a metric space. Then for any function
$g:\mathcal{K}\longrightarrow\mathcal{X}$ we define 
\[
\mu_{g}=\lim_{X\longrightarrow\infty}\mu_{g,X}
\]
to be weak limit of measures, if it exists. If $\mathcal{X}$ has
the discrete topology this is equivalent to strong convergence: $\mu_{g,X}\left(A\right)\longrightarrow\mu_{g}\left(A\right)$
for all $A$, but this equivalence is not true in general. Note also
the limit $\mu_{g}$ may not be a probability measure (that is escape
of mass can occur). We call $\mu_{g}$ the distribution of $g$ if
$\mu_{g}$ is a probability measure on $\mathcal{I}$.
\begin{example*}
Let $\mathcal{X}$ be the set of finite abelian $p$-groups, and let
$\Sigma$ be the power set of $\mathcal{X}$. Let $g:\mathcal{K}\longrightarrow\mathcal{X}$
be defined by $g\left(K\right)=Cl_{K,p}^{2}$ (note squaring only
has an effect when $p=2$).

Then the Cohen-Lenstra conjectures state that $\mu_{g}$ exists and
describe what it should be equal to, for any $p$.
\begin{example*}
Let $\mathcal{X}$ be the set of finite abelian $p$-torsion groups,
and let $\Sigma$ be the power set of $\mathcal{X}$. In this case
any $A\in\mathcal{X}$ is determined by its order, and hence we can
identify $\mathcal{X}$ with the set $\left\{ p^{i}\mid i\in\mathbb{Z}_{\ge0}\right\} $.
Let $g:\mathcal{K}\longrightarrow\mathcal{X}$ be defined by $g\left(K\right)=\left|Cl_{K}^{2}\left[p\right]\right|$.
Fouvry and Kl\"uners computed $\mu_{g}$ in the case when $p=2$.

Since in this case $\mathcal{X}$ is discrete we have strong convergence
of the $\mu_{g,X}$. Since $\mathcal{X}\subset\mathbb{Q}$ it is easy
to see that this also implies convergence when the measures in question
are viewed on $\mathbb{Q}$ with its usual metric as follows. Let
$\mathcal{B}$ be the Borel $\sigma$-algebra on $\mathbb{Q}$ (the
$\sigma$-algebra generated by the open subsets of $\mathbb{Q}$).
Then we can equivalently consider $\mu_{g}$ on the measurable space
$\mathcal{I}'=\left(\mathbb{Q},\mathcal{B}\right)$. Letting $\mu_{g}'$
be the new measure on $\mathcal{I}'$ the connection is given by $\mu_{g}'=\sum_{i=0}^{\infty}\delta_{\left\{ p^{i}\right\} }\mu_{g}\left(p^{i}\right)$
where $\delta_{\left\{ p^{i}\right\} }$ is the points-mass with respect
to $\mathcal{B}$.
\end{example*}
\end{example*}
In our case the functions $f$ are rational valued, however taking
the discrete metric on the set of values we find that $\lim_{X\longrightarrow\infty}\mu_{f,X}\left(x\right)=0$
for all $x\in\mathbb{Q}$, that is $\mu_{f}$ exists but there is
escape of mass. This issue is corrected if we instead view the set
of values lying inside $\mathbb{Q}$ with its usual metric, as in
the second example above. We will see that in this case $\mu_{f}$
will be a probability measure.

We prove a lemma which will be needed throughout the next two sections.

\begin{lem}
\label{lem:distribution plus pointmass}Let $\mathcal{M},\mathcal{I}$
be measurable spaces with $\mathcal{I}=\left(\mathbb{R},\mathcal{B}\right)$.
Let $g,h:\mathcal{M}\longrightarrow\mathcal{I}$ measurable functions
such that $\mu_{g}$ exists and $\mu_{h}=\delta_{a}$ and let $t_{a}$
be translation by $a$. Then
\[
\mu_{g+h}=\left(t_{a}\right)_{*}\mu_{g}.
\]
\end{lem}

\begin{proof}
For any measure $\mu$ let $F_{\mu}$ denote its cumulative distribution
function. It is a fact that for a sequence of measures $\mu_{n}\longrightarrow\mu$
weakly if and only if $F_{\mu_{n}}\left(x\right)\longrightarrow F_{\mu}\left(x\right)$
at all points of continuity of $F_{n},F$.

First note that $F_{\mu_{g}}\left(x-a\right)=F_{\left(t_{a}\right)_{*}\mu_{g}}\left(x\right)$,
since for any $U\in\mathcal{B}$, $\left(t_{a}\right)_{*}\mu_{g}\left(U\right)=\mu_{g}\left(U-a\right)$.
Hence we need to show that 
\[
\lim_{X\longrightarrow\infty}F_{\mu_{g+h,X}}\left(x\right)=F_{\mu_{g}}\left(x-a\right).
\]
 By definition we have
\begin{align*}
F_{\mu_{g+h,X}}\left(x\right) & =\mu_{g+h,X}\left((-\infty,x]\right)\\
 & =\mu_{X}\left(\left(g+h\right)^{-1}\left((-\infty,x]\right)\right).
\end{align*}
Fix $\epsilon>0$. Let $\epsilon_{0}>0$ be small enough such that
$\left|F_{\mu_{g}}\left(x-a\right)-F_{\mu_{g}}\left(x-a\pm\epsilon_{0}\right)\right|<\epsilon$
(this is possible since $F_{\left(t_{a}\right)_{*}\mu_{g}}$ is continuous
at $x$ which implies $F_{\mu_{g}}$ is continuous at $x-a$) and
there exists $X$ large enough such that for $W=h^{-1}\left(a-\epsilon_{0},a+\epsilon_{0}\right)$
we have $\mu_{X}\left(W\right)\ge1-\epsilon$.

We have the set inclusions

\[
\left(g\right)^{-1}\left((-\infty,x-a-\epsilon_{0}]\right)\cap W\subset\left(h+g\right)^{-1}\left((-\infty,x]\right)\cap W,
\]
\[
\left(h+g\right)^{-1}\left((-\infty,x]\right)\cap W\subset\left(g\right)^{-1}\left((-\infty,x-a+\epsilon_{0}]\right)\cap W.
\]
 By choice of $X$ we have $\left|\mu_{X}\left(U\cap W\right)-\mu_{X}\left(U\right)\right|<\epsilon$
for any $U\in\mathcal{B}$. Hence 
\[
\left|F_{\mu_{g,X}}\left(x-a\pm\epsilon_{0}\right)-\mu_{X}(\left(g\right)^{-1}\left((-\infty,x-a\pm\epsilon_{0}]\right)\cap W)\right|<\epsilon,
\]
\[
\left|F_{\mu_{g+h,X}}\left(x\right)-\mu_{X}(\left(h+g\right)^{-1}\left((-\infty,x]\right)\cap W)\right|<\epsilon.
\]
Furthermore since $F_{\mu_{g,X}}\longrightarrow F_{\mu_{g}}$ we have
\[
\left|F_{\mu_{g,X}}\left(x-a\pm\epsilon_{0}\right)-F_{\mu_{g,X}}\left(x-a\right)\right|<2\epsilon.
\]
Putting the above together we obtain that for any $\epsilon>0$ there
exists $X$ large enough such that 
\[
F_{\mu_{g,X}}\left(x-a\right)-\epsilon_{1}\le F_{\mu_{g+h,X}}\left(x\right)\le F_{\mu_{g,X}}\left(x-a\right)+\epsilon_{1}
\]
 for some $\epsilon_{1}$ a multiple of $\epsilon$. Thus 
\begin{align*}
\lim_{X\longrightarrow\infty}F_{\mu_{g+h,X}}\left(x\right) & =\lim_{X\longrightarrow\infty}F_{\mu_{g,X}}\left(x-a\right)\\
 & =F_{\mu_{g}}\left(x-a\right).
\end{align*}
As stated at the beginning $F_{\mu_{g}}\left(x-a\right)=F_{\left(t_{a}\right)_{*}\mu_{g}}\left(x\right)$
which implies $\mu_{g+h,X}\longrightarrow\left(t_{a}\right)_{*}\mu_{g}$
weakly. This completes the proof.
\end{proof}
We are now interested in computing the distribution $\mu_{f}$ of
the function $f=\sum_{T_{i}}f_{T_{i}}$ where the sum is over generating
sets $T_{i}\subset T_{0}$.

We first consider the case when $T_{0}$ is not complete bipartite
and later apply this to handle the complete bipartite case.

\subsection{\label{subsec:The-non-complete}The non complete bipartite case}
\begin{thm}
\label{thm:k_moments pointmass}Let $a=\left|\mathrm{Aut}_{H}\left(G/\Phi(G)\right)\right|$.
Suppose $T_{0}$ is not complete bipartite. For all positive integers
$k$
\[
\lim_{X\longrightarrow\infty}\frac{\sum_{d\in\mathcal{D}_{X,\alpha}^{\pm}}\left(\frac{f\left(d\right)}{c^{\omega\left(d\right)}}\right)^{k}}{\sum_{d\in\mathcal{D}_{X,\alpha}^{\pm}}1}=\begin{cases}
\left(2^{c-s+n-1}a\right)^{-k}\cdot\left(\left|T_{0}\right|Q^{\pm}\right)^{k} & \text{if }T_{i}=T_{0}\text{ for all \ensuremath{i}}\\
0 & \text{else}
\end{cases}
\]
 where $Q^{\pm}$ is defined in (\ref{eq:Q}).
\end{thm}

\begin{proof}
Since $f=\sum_{T_{i}}f_{T_{i}}$, taking the $k$th power gives $f^{k}=\sum_{\widehat{T}}f_{\widehat{T}}$
where $\widehat{T}=\prod_{i=1}^{k}T_{i}$.

Fix $\widehat{T}$ and let $\mathbf{c}=\max_{\mathcal{U}\in\mathbf{U}(\widehat{T})}\left|\mathcal{U}\right|$.
First suppose$\widehat{T}\neq\prod_{i=1}^{k}T_{0}$ . By Theorem \ref{thm:The-maximal-unlinked}
$\mathbf{c}<c^{k}$. Then the second case of the statement follows
by Theorem \ref{thm:asymptotic_analysis} with $a=c^{k}$.

Now suppose $T_{i}=T_{0}\text{ for all \ensuremath{i}}$. Then by
Theorem \ref{thm:The-maximal-unlinked} $\mathbf{c}=c^{k}$ and $\mathbf{U}(\widehat{T})=\mathcal{M}(\widehat{T})$.
Then the first case of the statement follows by Proposition \ref{prop:T_moments}.
\end{proof}
\begin{cor}
\label{cor:distribution point mass}Let $a=\left|\mathrm{Aut}_{H}\left(G/\Phi(G)\right)\right|$.
Let $q=\left(2^{c-s+n-1}a\right)^{-k}\cdot\left(Q^{\pm}\right)^{k}$.
Suppose $T_{0}$ is not complete bipartite.

Then the distribution of $h\left(d\right)=\frac{f\left(d\right)}{c^{\omega\left(d\right)}}$
is $\mu_{h}=\delta_{q}$ where $\delta_{q}$ is the point-mass at
$q$.
\end{cor}

\begin{proof}
In the case where $T_{i}=T_{0}$ and $T_{0}$ is complete bipartite
the $k$th moment of $f\left(d\right)/(c)^{\omega\left(d\right)}$
is the $k$th power of some non-zero rational number by Theorem \ref{thm:k_moments pointmass}.
Thus the distribution determined is a point-mass.
\end{proof}

\subsection{\label{subsec:The-complete-bipartite}The complete bipartite case}

Suppose $T_{0}$ is complete bipartite. Fix a quadratic discriminant
$d$ of the quadratic field $K$. Let
\begin{equation}
g\left(d\right)=\sum_{\left(d_{1},\ldots,d_{r}\right)\in\mathcal{D}}\prod_{i=1}^{r}\prod_{p\mid d_{i}}\left(1+\left(\frac{\prod_{j\in S_{i}}d_{j}}{p}\right)^{\mathrm{ord}_{p}\left(d\right)}\right).\label{eq:g for T}
\end{equation}
 Note $g\left(d\right)=2^{n}\left|\mathrm{Aut}_{H}(G/\Phi\left(G\right))\right|f_{T_{0}}\left(d\right)$.
Let $\widehat{\mathcal{D}}=\left\{ \left(d_{1},\ldots,d_{r}\right)\mid\prod_{i=1}^{r}d_{i}=d\right\} $.
Recall we defined $\mathcal{D}\subset\widehat{\mathcal{D}}$ to be
the subset such that $d_{i}\neq1$ for all $i$. In the following
we will write $\mathcal{D}_{T}$ and $\widehat{\mathcal{D}}_{T}$
to keep track of tuples corresponding to any subset $T\subset T_{0}$
and $g_{\mathcal{D}_{T}}$ for the function in (\ref{eq:g for T})
with $\mathcal{D}=\mathcal{D}_{T}$. We have
\begin{align*}
g_{\widehat{\mathcal{D}}_{T_{0}}}\left(d\right) & =\sum_{T\subset T_{0}}g_{\mathcal{D}_{T}}\left(d\right)
\end{align*}
where we sum over all subsets $T\subset T_{0}$ (not necessarily generating
sets).

Note $g_{\mathcal{D}_{T_{0}}}/c^{\omega\left(d\right)}=2^{n}\left|\mathrm{Aut}_{H}(G/\Phi\left(G\right))\right|f_{T_{0}}\left(d\right)/c^{\omega\left(d\right)}$
is the function whose distribution we want to compute. We can write
\begin{equation}
\frac{g_{\mathcal{D}_{T_{0}}}\left(d\right)}{c^{\omega\left(d\right)}}=\frac{g_{\widehat{\mathcal{D}}_{T_{0}}}\left(d\right)}{c^{\omega\left(d\right)}}-\sum_{T\subset T_{0}}\frac{g_{\mathcal{D}_{T}}\left(d\right)}{c^{\omega\left(d\right)}}.\label{eq:rewriting g/c}
\end{equation}
 Now we consider the distribution of each of the functions on the
right-hand side. We start with $g_{\widehat{\mathcal{D}}_{T_{0}}}\left(d\right)/c^{\omega\left(d\right)}$.

For any $d'_{1},d_{2}'$ dividing $d$ define 
\[
P_{d_{1}'}\left(d_{2}'\right)=\prod_{p\mid d_{i}}\left(1+\left(\frac{d_{2}'}{p}\right)^{\mathrm{ord}_{p}\left(d\right)}\right).
\]
Now define the function 
\[
C\left(d\right)=\frac{1}{2^{\omega\left(d\right)}}\sum_{\substack{\left(d_{1}',d_{2}'\right)\\
d_{1}'d_{2}'=d
}
}P_{d_{1}'}\left(d_{2}'\right)P_{d_{2}'}\left(d_{1}'\right).
\]
 We will use this function to relate $g_{\widehat{\mathcal{D}}_{T_{0}}}/c^{\omega\left(d\right)}$
to the 4-torsion in $Cl_{K}$.
\begin{lem}
\label{lem:g in terms of C}For any quadratic field $K$ with discriminant
$d$
\[
\frac{g_{\widehat{\mathcal{D}}_{T_{0}}}\left(d\right)}{c^{\omega\left(d\right)}}=C\left(d\right).
\]
\end{lem}

\begin{proof}
Since $T_{0}$ is bipartite we have $T_{0}=\left(W_{1},W_{2}\right)$.
Let $T_{0,i}=T_{0}\cap W_{i}$. Then we have
\begin{align*}
\frac{g_{\widehat{\mathcal{D}}_{T_{0}}}\left(d\right)}{c^{\omega\left(d\right)}} & =\frac{1}{c^{\omega\left(d\right)}}\sum_{\left(d_{1},\ldots,d_{r}\right)\in\mathcal{D}}\prod_{i=1}^{r}\prod_{p\mid d_{i}}\left(1+\left(\frac{\prod_{j\in S_{i}}d_{j}}{p}\right)^{\mathrm{ord}_{p}\left(d\right)}\right)\\
 & =\frac{1}{c^{\omega\left(d\right)}}\sum_{\substack{\left(d_{1}',d_{2}'\right)\\
d_{1}'d_{2}'=d
}
}\left(\sum_{\left(d_{1},\ldots,d_{r/2}\right)\in\widehat{\mathcal{D}}_{T_{0,1}}\left(d_{1}'\right)}P_{d_{1}'}\left(d_{2}'\right)\right)\left(\sum_{\left(d_{r/2+1},\ldots,d_{r}\right)\in\mathcal{\widehat{D}}_{T_{0,2}}\left(d_{2}'\right)}P_{d_{2}'}\left(d_{1}'\right)\right).
\end{align*}
 Noting that the terms in the inner summations depend only on the
factors $d_{1}'$ and $d_{2}'$ we can write this as
\begin{align*}
 & \frac{1}{c^{\omega\left(d\right)}}\sum_{\substack{\left(d_{1}',d_{2}'\right)\\
d_{1}'d_{2}'=d
}
}P_{d_{1}'}\left(d_{2}'\right)P_{d_{2}'}\left(d_{1}'\right)\left(\sum_{\left(d_{1},\ldots,d_{r/2}\right)\in\widehat{\mathcal{D}}_{T_{0,1}}\left(d_{1}'\right)}1\right)\left(\sum_{\left(d_{r/2+1},\ldots,d_{r}\right)\in\mathcal{\widehat{D}}_{T_{0,2}}\left(d_{2}'\right)}1\right).
\end{align*}
Now $\left|\widehat{\mathcal{D}}_{T_{0,i}}\left(d_{i}'\right)\right|=\left(r/2\right)^{\omega\left(d_{i}'\right)}=c_{T_{0,i}}^{\omega\left(d_{i}'\right)}$
since this is the number of ways of factoring $d_{i}'$ into $r/2$
elements. Note that $c_{T_{0,i}}=c_{T_{0}}/2$. Thus we get
\begin{align*}
 & \frac{1}{c^{\omega\left(d\right)}}\sum_{\substack{\left(d_{1}',d_{2}'\right)\\
d_{1}'d_{2}'=d
}
}P_{d_{1}'}\left(d_{2}'\right)P_{d_{2}'}\left(d_{1}'\right)c_{T_{0,1}}^{\omega\left(d_{1}'\right)}c_{T_{0,2}}^{\omega\left(d_{2}'\right)}\\
 & =\frac{1}{2^{\omega\left(d\right)}}\sum_{\substack{\left(d_{1}',d_{2}'\right)\\
d_{1}'d_{2}'=d
}
}P_{d_{1}'}\left(d_{2}'\right)P_{d_{2}'}\left(d_{1}'\right).
\end{align*}
This completes the proof.
\end{proof}
\begin{lem}
\label{lem:g in terms of Cl2}Suppose $d$ is the discriminant of
a quadratic field $K$. Then

\begin{align*}
\frac{g_{\widehat{\mathcal{D}}_{T_{0}}}\left(d\right)}{c^{\omega\left(d\right)}} & =2\left|Cl_{K}\left[4\right]/Cl_{K}\left[2\right]\right|.
\end{align*}
\end{lem}

\begin{proof}
By Lemma \ref{lem:g in terms of C} $g_{\widehat{\mathcal{D}}_{T_{0}}}\left(d\right)/c^{\omega\left(d\right)}=C\left(d\right)$.

Let $G=D_{4}\cong C_{4}\rtimes C_{2}$, $H=C_{4}$ and $T$ be the
two non-trivial elements of $G/\Phi\left(G\right)\cong C_{2}\times C_{2}$
not in the projection of $H$. Then by Theorem \ref{thm:f_T_formula}
\begin{align}
C\left(d\right) & =\frac{2^{n}\left|\mathrm{Aut}_{H,T}(G/\Phi\left(G\right))\right|}{2^{\omega\left(d\right)}}f_{T}\left(d\right)+2\nonumber \\
 & =\frac{1}{2^{\omega\left(d\right)}}\left(8f_{T}\left(d\right)+2\cdot2^{\omega\left(d\right)}\right)\label{eq:C in terms of f}
\end{align}
where $n=2$ and $\left|\mathrm{Aut}_{H,T}(G/\Phi\left(G\right))\right|=2$.

Now $f_{T}$ is the number of $\left(G,H,T\right)$-extensions of
$K$. This is also the number of $\left(G,H\right)$-extensions since
this is the unique $T$ for the pair $\left(G,H\right)=\left(D_{4},C_{4}\right)$.

We claim that $f_{T}$ is in fact equal to the number of unramified
$C_{4}$-extensions of $K$. Clearly any $\left(G,H\right)$-extension
is an unramified $C_{4}$-extension of $K$. Conversely suppose $L/K$
is unramified and $\alpha:\mathrm{Gal}\left(L/K\right)\longrightarrow C_{4}$
is an isomorphism. Since $C_{4}$ is abelian $L$ is a subfield of
the Hilbert class field of $K$ hence Galois over $\mathbb{Q}$.
Then since $G\cong C_{4}\rtimes C_{2}$ and $\mathrm{Gal}\left(L/\mathbb{Q}\right)\cong\mathrm{Gal}\left(L/K\right)\rtimes\mathrm{Gal}\left(K/\mathbb{Q}\right)$
we can lift $\alpha$ to $\alpha':\mathrm{Gal}\left(L/\mathbb{Q}\right)\longrightarrow G$.
Hence $L/K$ is a $\left(G,H\right)$-extension.

For any such extension there are exactly 2 surjections from $Cl_{K}$
to $C_{4}$. Hence $2\cdot f_{T}=\left|\mathrm{Surj}\left(Cl_{K},C_{4}\right)\right|$.

Note $\left|\mathrm{Hom}\left(Cl_{K},C_{2^{m}}\right)\right|=\left|Cl_{K}\left[2^{m}\right]\right|$.
Furthermore 
\[
\left|\mathrm{Hom}\left(Cl_{K},C_{4}\right)\right|=\left|\mathrm{Surj}\left(Cl_{K},C_{4}\right)\right|+\left|\mathrm{Hom}\left(Cl_{K},C_{2}\right)\right|.
\]
 It is well known that $\left|Cl_{K}\left[2\right]\right|=2^{\omega\left(d\right)-1}$.
Combining the above facts with (\ref{eq:C in terms of f}) we get
\begin{align*}
C\left(d\right) & =\frac{1}{2^{\omega\left(d\right)}}\left(4\left|\mathrm{Surj}\left(Cl_{K},C_{4}\right)\right|+2\cdot2^{\omega\left(d\right)}\right)\\
 & =\frac{4}{2^{\omega\left(d\right)}}\left(\left|\mathrm{Surj}\left(Cl_{K},C_{4}\right)\right|+2^{\omega\left(d\right)-1}\right)\\
 & =2\left|Cl_{K}\left[4\right]\right|/\left|Cl_{K}\left[2\right]\right|.
\end{align*}
\end{proof}

Let $\mu_{CL}^{\pm}\left(H\right)$ be the Cohen-Lenstra probability
measure on finite abelian 2-groups $H$ which is proportional to $1/(\left|\mathrm{Aut}H\right|\cdot\left|H\right|^{a})$
where $a=0,1$ in the positive (resp. negative) case. Let $P_{CL}^{\pm}\left(i\right)$
be the probability with respect to the Cohen-Lenstra measure that
a finite abelian 2-group has rank $i$.

The distribution of the funtion $\left|Cl_{K}\left[4\right]/Cl_{K}\left[2\right]\right|$
over both imaginary and real quadratic fields was determined by Fouvry
and Kl\"uners in \cite{fk1} and \cite{fk2distribution}. Thus we
conclude the following.
\begin{lem}
\label{lem:cl dist}The function $h\left(d\right)=g_{\widehat{\mathcal{D}}_{T_{0}}}\left(d\right)/c^{\omega\left(d\right)}$
has distribution $\mu_{h}$ supported on the set $\left\{ 2^{i}\mid i\in\mathbb{Z}_{\ge1}\right\} $
given by
\[
\mu_{h}\left(2^{i}\right)=P_{CL}^{\pm}\left(i-1\right)
\]
 over all real (respectively imaginary) quadratic fields.
\end{lem}

\begin{proof}
Follows from Lemma \ref{lem:g in terms of Cl2} and Theorem 3 from
\cite{fk1}.
\end{proof}
Next we consider the distribution of $h_{T}\left(d\right)=g_{\mathcal{D}_{T}}\left(d\right)/c^{\omega\left(d\right)}$
for $T\subset T_{0}$. We will do this by computing the $k$th moments
of $h_{T}$.
\begin{lem}
\label{lem:lower order dist}If $T\in\left\{ W_{1},W_{2}\right\} $
then $\mu_{h_{T}}=\delta_{1}$. If $T\notin\left\{ W_{1},W_{2},T_{0}\right\} $
then $\mu_{h_{T}}=\delta_{0}$.
\end{lem}

\begin{proof}
Let $\mathbf{c}=\max_{\mathcal{U}\in\mathbf{U}\left(T^{k}\right)}\left|\mathcal{U}\right|$.
Since $T\subset T_{0}$ and $T_{0}$ is complete bipartite it is easy
to see that $c_{T}\le c_{T_{0}}$ with equality if and only if $T\in\left\{ W_{1},W_{2},T_{0}\right\} $.

Let $G'$ be the group generated by $T$, $H'=H\cap G$ and let $T_{0}'$
be the corresponding maximal generating set which lifts to order 2
elements in $G'-H'$. By Theorem \ref{thm:The-maximal-unlinked} applied
to $T_{0}'$ we see that $\mathbf{c}\le c_{T_{0}'}^{k}$ with equality
if and only if $T=T_{0}'$. By Lemma \ref{lem:unlinkedset_bound}
$c_{T_{0}'}\le c_{T_{0}}$.

If $T=T_{0}'$ then $\mathbf{c}=c_{T}^{k}=c_{T_{0}'}^{k}$. Putting
these facts together we see that $\mathbf{c}\le c_{T_{0}}^{k}$ with
equality if and only if $T\in\left\{ W_{1},W_{2},T_{0}\right\} $.

Hence if $T\notin\left\{ W_{1},W_{2},T_{0}\right\} $ by Theorem \ref{thm:asymptotic_analysis}
\[
\lim_{X\longrightarrow\infty}\frac{\sum_{d<X}\left(\frac{g_{\mathcal{D}_{T}}\left(d\right)}{c^{\omega\left(d\right)}}\right)^{k}}{\sum_{d<X}1}=0.
\]

If $T\in\left\{ W_{1},W_{2}\right\} $ then from the definition of
$g_{\mathcal{D}_{T}}$ we see that
\begin{align*}
h_{T}\left(d\right) & =\frac{1}{c^{\omega\left(d\right)}}\sum_{\left(d_{1},\ldots,d_{r/2}\right)\in\mathcal{D}_{T}}2^{\omega\left(d\right)}\\
 & =\frac{1}{c^{\omega\left(d\right)}}\left(\frac{r}{2}\right)^{\omega\left(d\right)}2^{\omega\left(d\right)}\\
 & =1.
\end{align*}
\end{proof}
We now put together the results of this section to determine the distribution
of $g_{\mathcal{D}_{T_{0}}}/c^{\omega\left(d\right)}$ and subsequently
$f$.
\begin{thm}
\label{thm:dist complete bipartite}Suppose $T_{0}$ is complete bipartite.
Let $q=2^{n-1}\left|\mathrm{Aut}_{H}(G/\Phi\left(G\right))\right|$.
Then $h\left(d\right)=f\left(d\right)/c^{\omega\left(d\right)}$ has
distribution $\mu_{h}$ supported on the set $\left\{ \left(2^{i}-1\right)/q\mid i\in\mathbb{Z}_{\ge0}\right\} $
given by
\[
\mu_{h}\left(\left(2^{i}-1\right)/q\right)=P_{CL}^{\pm}\left(i\right)
\]
 over the set of real (respectively imaginary) quadratic fields.
\end{thm}

\begin{proof}
Recall we have the relations 
\begin{equation}
f=\sum_{T\subset T_{0}}f_{T}\label{eq:expand f}
\end{equation}
 where the sum is over generating sets, 
\[
g_{\mathcal{D}_{T_{0}}}\left(d\right)=g_{\widehat{\mathcal{D}}_{T_{0}}}\left(d\right)-\sum_{T\subsetneqq T_{0}}g_{\mathcal{D}_{T}}\left(d\right)
\]
 where the sum is over all strict subsets of $T_{0}$, including non-generating
sets, and $g_{\mathcal{D}_{T_{0}}}\left(d\right)=2^{n}\left|\mathrm{Aut}_{H}(G/\Phi\left(G\right))\right|f_{T_{0}}\left(d\right)$.

By an argument similar to Lemma \ref{lem:lower order dist} we can
show that when $T\neq T_{0}$ the $k$th moments of $f_{T}/c^{\omega\left(d\right)}$
are all equal to 0 since in this case $T$ is a generating set and
hence $\max_{\mathcal{U}\in\mathbf{U}\left(T^{k}\right)}\left|\mathcal{U}\right|<c_{T_{0}}^{k}$.
Thus $\mu_{f_{T}/c^{\omega}}=\delta_{0}$. Combined with (\ref{eq:expand f})
this implies that $\mu_{h}=\mu_{f_{T_{0}}/c^{\omega}}$.

Let $h_{g}\left(d\right)=g_{\mathcal{D}_{T_{0}}}/c^{\omega\left(d\right)}$.
By Lemmas \ref{lem:cl dist} and \ref{lem:lower order dist} $\mu_{h_{g}}$
is supported on the set $\left\{ 2^{i}-2\mid i\in\mathbb{Z}_{\ge1}\right\} $
and
\[
\mu_{h_{g}}\left(2^{i}-2\right)=P_{CL}^{\pm}\left(i-1\right)
\]
 over real (respectively imaginary) quadratic fields.

Thus we conclude that $\mu_{h}$ is supported on $\left\{ \left(2^{i}-1\right)/q\mid i\in\mathbb{Z}_{\ge0}\right\} $
given by 
\[
\mu_{h}\left(\left(2^{i}-1\right)/q\right)=P_{CL}^{\pm}\left(i\right).
\]
\end{proof}
\begin{rem}
\label{rem:classgroup formula for f}From the above proof we see that
\begin{align}
f/c^{\omega\left(d\right)} & =\frac{1}{2^{n}\left|\mathrm{Aut}_{H}(G/\Phi\left(G\right))\right|}\left(2\left|Cl_{K}\left[4\right]/Cl_{K}\left[2\right]\right|-\sum_{T\subsetneqq T_{0}}g_{\mathcal{D}_{T}}\left(d\right)/c^{\omega\left(d\right)}\right)\nonumber \\
 & \qquad+\sideset{}{_{T\subsetneqq T_{0}}^{\prime}}\sum f_{T}/c^{\omega\left(d\right)}\nonumber \\
 & =\frac{1}{2^{n}\left|\mathrm{Aut}_{H}(G/\Phi\left(G\right))\right|}\left(2\left|Cl_{K}\left[4\right]/Cl_{K}\left[2\right]\right|-\sideset{}{_{T\subsetneqq T_{0}}^{\prime\prime}}\sum g_{\mathcal{D}_{T}}\left(d\right)/c^{\omega\left(d\right)}\right)\label{eq:class group formula for f}
\end{align}
 where the prime indicates a sum over subsets $T$ which are generating
sets, and the double-prime indicates a sum is over subsets $T$ which
are \textit{not} generating sets.

It is evident from this that when $T_{0}$ is complete bipartite the
distribution of $f\left(d\right)/c^{\omega\left(d\right)}$ is controlled
by the $4$-torsion of the class group. The values of the function
however depend on additional information, which comes from the other
unramified 2-extensions of the quadratic field $K$.
\end{rem}

Up until now we have been working with the formula for the function
$f_{T}$ which counts $\left(G,H,T\right)$-extensions unramified
away from infinity. For the remainder of the paper we switch notation
as follows. Let $f_{T}^{\prime}$ denote the function counting $\left(G,H,T\right)$-extensions
unramified away from infinity and let $f_{T}$ denote that counting
$\left(G,H,T\right)$-extensions unramified everywhere. Similarly
define $f^{\prime}$ and $f$. We now show that the distributions
and moments of these two functions are the same.
\begin{lem}
\label{lem:counting unram everywhere}Let $f_{T}^{\prime}\left(K\right)$
and $f_{T}\left(K\right)$ denote the number of $\left(G,H,T\right)$-extensions
of $K$ which are unramified away from infinity (resp. unramified
everywhere). Then $\mu_{f_{T_{0}}/c^{\omega}}=\mu_{f_{T_{0}}^{\prime}/c^{\omega}}$
and $E_{k}\left(f_{T_{0}}/c^{\omega}\right)=E_{k}\left(f_{T_{0}}^{\prime}/c^{\omega}\right)$
for all positive integers $k$.
\end{lem}

\begin{proof}
By the same proofs as in Propositions 11 and Lemma 12 of \cite{albertsklys}
we see that if we define $g=f_{T_{0}}/c^{\omega}-f_{T_{0}}^{\prime}/c^{\omega}$
then $E_{k}\left(g\right)=0$ for all $k$. This implies $\mu_{g}=\delta_{0}$
and hence by Lemma \ref{lem:distribution plus pointmass} $\mu_{f_{T_{0}}/c^{\omega}}=\mu_{f_{T_{0}}^{\prime}/c^{\omega}}$.
Then also clearly $E_{k}\left(f_{T_{0}}/c^{\omega}\right)=E_{k}\left(f_{T_{0}}^{\prime}/c^{\omega}\right)$.
\end{proof}

\section{\label{sec:Moments-and-Correlations}Moments and Correlations}

Using results from the previous section we can determine the moments
of $f$ for a fixed pair $\left(G,H\right)$ as well as correlations
of such functions for several different pairs $\left(G,H\right)$.

Using almost the same proof as in Lemma \ref{lem:distribution plus pointmass}
we can also prove
\begin{lem}
\label{lem:distribution times pointmass}Let $\mathcal{M},\mathcal{I}$
be measurable spaces with $\mathcal{I}=\left(\mathbb{R},\mathcal{B}\right)$.
Let $g,h:\mathcal{M}\longrightarrow\mathcal{I}$ measurable functions
such that $\mu_{g}$ exists and $\mu_{h}=\delta_{a}$ and let $m_{a}$
be multiplication by $a$. Then
\[
\mu_{g\cdot h}=\left(m_{a}\right)_{*}\mu_{g}.
\]
\end{lem}

We first need a lemma on correlations of the type of functions we
will encounter. For any measure $\mu$ let $E_{k}\left(\mu\right)$
denote the $k$th moment defined by $\int x^{k}d\mu$.
\begin{lem}
\label{lem:limiting moments}Let $\mathcal{M},\mathcal{I}$ be measurable
spaces with $\mathcal{I}=\left(\mathbb{R},\mathcal{B}\right)$. Let
$g,h:\mathcal{M}\longrightarrow\mathcal{I}$ be positive measurable
functions. Suppose $g$ has countable well-ordered image and $\mu_{g,X}\longrightarrow\mu_{g}$
strongly and $\mu_{h}=\delta_{a}$. Furthemore suppose $E_{k}\left(\mu_{g,X}\right)\longrightarrow E_{k}\left(\mu_{g}\right)$
and $E_{k}\left(\mu_{h,X}\right)\longrightarrow E_{k}\left(\mu_{h}\right)$
for $k=1,2$. Then $E_{1}\left(\mu_{g\cdot h,X}\right)\longrightarrow E_{1}\left(\mu_{g\cdot h}\right)$.
\end{lem}

\begin{proof}
Let $\left\{ y_{i}\right\} _{i=1}^{\infty}$ be the image of $g$
in increasing order. Then for each $k=1,2$ we have the equations
\[
E_{k}\left(\mu_{g}\right)=\sum_{i=0}^{\infty}\mu_{g}\left(y_{i}\right)y_{i}^{k}.
\]
Let $\epsilon_{1}>0$. Since the above series converges there exists
$N$ large such that $\int_{x>N}x^{k}d\mu_{g}<\epsilon_{1}$ for $k=1,2$.
Since $E_{k}\left(\mu_{g,X}\right)\longrightarrow E_{k}\left(\mu_{g}\right)$
and $\mu_{g,X}\longrightarrow\mu_{g}$ strongly we can make $X$ large
enough so that $\int_{x>N}x^{k}d\mu_{g,X}<\epsilon_{1}$ for $k=1,2$.

Let $\epsilon_{2},\delta>0$ and $I=\left(a-\delta,a+\delta\right)$
and let $J=\mathbb{R}\backslash I$. Since $\mu_{h}=\delta_{a}$ we
can find $X$ large enough such that $\mu_{h,X}\left(I\right)>1-\epsilon_{2}$
and $\mu_{h,X}\left(J\right)<\epsilon_{2}$. Since $E_{k}\left(\mu_{h,X}\right)\longrightarrow E_{k}\left(\mu_{h}\right)=a^{k}$
we can also find $X$ large enough such that $\left|a^{k}-\int x^{k}d\mu_{h,X}\right|<\epsilon_{2}$
for $k=1,2$. Hence 
\begin{align}
\left|\int_{J}x^{k}d\mu_{h,X}\right| & <\epsilon_{2}+\left|a^{k}-\int_{I}x^{k}d\mu_{h,X}\right|\label{eq:correlation ineq1}
\end{align}
and the right side goes to 0 with $\epsilon_{2}$ and $\delta$. Also
choose $\epsilon_{2}$ such that $\epsilon_{2}N^{2}$ is small. So
we can choose $\epsilon_{1},\epsilon_{2},\delta$ arbitrarily small
and $X$ large enough such that all of the above holds.

Now fix such an $X$ and let $W_{1}=h^{-1}\left(I\right)\cap\mathcal{D}_{X}$
and $W_{1}^{c}=\mathcal{D}_{X}\backslash W_{1}$. Note $\mu_{X}\left(W_{1}\right)=\mu_{h,X}\left(I\right)>1-\epsilon_{2}$
and $\mu_{X}\left(W_{1}^{c}\right)=\mu_{h,X}\left(J\right)<\epsilon_{2}$.
By definition

\begin{align*}
E_{1}\left(\mu_{g\cdot h,X}\right) & =\int_{\mathcal{D}_{X}}\left(g\cdot h\right)\left(x\right)d\mu_{X}\\
 & =\int_{W_{1}}\left(g\cdot h\right)\left(x\right)d\mu_{X}+\int_{W_{1}^{c}}\left(g\cdot h\right)\left(x\right)d\mu_{X}
\end{align*}

For $k=1,2$ we have
\begin{align}
\int_{W_{1}^{c}}g\left(x\right)^{k}d\mu_{X} & =\int_{W_{1}^{c}\cap g^{-1}(\left[0,N\right])}g\left(x\right)^{k}d\mu_{X}+\int_{W_{1}^{c}\cap g^{-1}(N,\infty)}g\left(x\right)^{k}d\mu_{X}\label{eq:correlation ineq2}\\
 & \le\epsilon_{2}N^{k}+\int_{x>N}x^{k}d\mu_{g,X}\nonumber \\
 & \le\epsilon_{2}N^{2}+\epsilon_{1}.\nonumber 
\end{align}
 Setting $k=1$ this implies that $\left|E_{1}\left(\mu_{g}\right)-\int_{W_{1}}g\left(x\right)d\mu_{X}\right|<\left|E_{1}\left(\mu_{g}\right)-E_{1}\left(\mu_{g,X}\right)\right|+\epsilon_{2}N^{2}+\epsilon_{1}$.
Then it follows from the definition of $W_{1}$ and the previous fact
that 
\[
\left(a-\delta\right)\int_{W_{1}}g\left(x\right)d\mu_{X}\le\int_{W_{1}}\left(g\cdot h\right)\left(x\right)d\mu_{X}\le\left(a+\delta\right)\int_{W_{1}}g\left(x\right)d\mu_{X}
\]
 so $\int_{W_{1}}\left(g\cdot h\right)\left(x\right)d\mu_{X}$ can
be made arbitrarily close to $aE_{1}\left(\mu_{g}\right)=E_{1}\left(\mu_{g\cdot h}\right)$
(see Lemma \ref{lem:distribution times pointmass}).

Finally it follows from Cauchy-Schwarz that 
\[
\int_{W_{1}^{c}}\left(g\cdot h\right)\left(x\right)d\mu_{X}\le\left(\int_{W_{1}^{c}}g\left(x\right)^{2}d\mu_{X}\right)^{1/2}\left(\int_{W_{1}^{c}}h\left(x\right)^{2}d\mu_{X}\right)^{1/2}.
\]
 Note $\int_{W_{1}^{c}}h\left(x\right)^{2}d\mu_{X}=\int_{J}x^{2}d\mu_{h,X}$.
Hence by (\ref{eq:correlation ineq1}) and (\ref{eq:correlation ineq2})
with $k=2$ we see that $\int_{W_{1}^{c}}\left(g\cdot h\right)\left(x\right)d\mu_{X}$
goes to 0 with $\epsilon_{1},\epsilon_{2}$ and $\delta$.
\end{proof}
\begin{rem}
\label{rem:limiting moments}A similar easier proof can be used to
show Lemma \ref{lem:limiting moments} holds when $g$ satisfies the
same conditions as $h$ and $\mu_{g}=\delta_{b}$ for some non-zero
$b\in\mathbb{R}$.
\end{rem}

 Considering the explicit form of the function $f$ when $T_{0}$
is complete bipartite we obtain the following result on correlations.
\begin{thm}
\label{thm:correlations}Let $\left(G_{1},H_{1}\right),\ldots,\left(G_{k},H_{k}\right)$
be a sequence of admissible pairs. Let $T_{i,0}$ be the maximal admissible
generating set corresponding to $\left(G_{i},H_{i}\right)$. Let $f_{i}$
be the function counting $\left(G_{i},H_{i}\right)$-extensions. Let
$Y\subset\left[k\right]$ be the set of indices for which $T_{i,0}$
is complete bipartite. Then
\[
\lim_{X\longrightarrow\infty}\frac{\sum_{d\in\mathcal{D}_{X}^{\pm}}\prod_{i=1}^{k}\left(\frac{f_{i}\left(d\right)}{c_{T_{i,0}}^{\omega\left(d\right)}}\right)}{\sum_{d\in\mathcal{D}_{X}^{\pm}}1}=\left(\prod_{j=1}^{k}q_{j}^{-1}\right)\left(\prod_{j\in\left[k\right]\backslash Y}P_{j}^{\pm}\right)\left(\sum_{i=0}^{\left|Y\right|}\left(-1\right)^{i}\binom{\left|Y\right|}{i}M^{\pm}\left(\left|Y\right|-i\right)\right)
\]
 where
\begin{itemize}
\item $q_{j}=2^{n_{j}-1}\left|\mathrm{Aut}_{H_{j}}(G_{j},\Phi\left(G_{j}\right))\right|$
\item $M^{-}\left(j\right)=N\left(j\right)$ and $M^{+}\left(j\right)=2^{-j}\left(N\left(j+1\right)-N\left(j\right)\right)$
\item $P_{j}^{\pm}=2^{c_{T_{j,0}}-s_{j}+n_{j}-1}\left|T_{j,0}\right|Q_{j}^{\pm}$
where $s_{j}$ is the number of connected components of $T_{j,0}$
and $Q_{j}^{\pm}$ is defined in (\ref{eq:Q}).
\end{itemize}
\end{thm}

\begin{proof}
From Remark \ref{rem:classgroup formula for f} we see that if $T_{i,0}$
is complete bipartite then $f_{i}\left(d\right)/c_{T_{i,0}}^{\omega\left(d\right)}$
can be written as a sum of $(\left|Cl_{K}\left[4\right]/Cl_{K}\left[2\right]\right|-1)/q_{i}$
and functions $g_{i,j}$ satisfying $\mu_{g_{i,j}}=\delta_{0}$ and
$E_{k}\left(\mu_{g_{i,j},X}\right)\longrightarrow E_{k}\left(\mu_{g_{i,j}}\right)$.
Let $f\left(K\right)=\left|Cl_{K}\left[4\right]/Cl_{K}\left[2\right]\right|-1$.
Note $f^{k}\left(K\right)=\sum_{i=0}^{k}\left(-1\right)^{i}\binom{k}{i}\left|Cl_{K}\left[4\right]/Cl_{K}\left[2\right]\right|^{i}$.
By Fouvry-Kl\"uners
\[
\lim_{X\longrightarrow\infty}\frac{\sum_{K\in\mathcal{D}_{X}^{\pm}}\left(\left|Cl_{K}\left[4\right]/Cl_{K}\left[2\right]\right|\right)^{k}}{\sum_{K\in\mathcal{D}_{X}^{\pm}}1}=M^{\pm}\left(k\right).
\]
 Hence 
\[
\lim_{X\longrightarrow\infty}\frac{\sum_{d\in\mathcal{D}_{X}^{\pm}}\left(f\left(K\right)\right)^{k}}{\sum_{d\in\mathcal{D}_{X}^{\pm}}1}=\sum_{i=0}^{k}\left(-1\right)^{i}\binom{k}{i}M^{\pm}\left(k-i\right)
\]
and $\mu_{f}\left(2^{i}-1\right)=P_{CL}^{\pm}\left(i\right)$ for
$i\in\mathbb{Z}_{\ge0}$. Note $f\left(K\right)^{k}$ satisfies the
conditions of Lemma \ref{lem:limiting moments}.

By Theorem \ref{thm:k_moments pointmass} and Corollary \ref{cor:distribution point mass}
we know the limiting distribution and moments of $f_{i}\left(d\right)/c_{T_{i,0}}^{\omega\left(d\right)}$
for each $i\notin Y$. Let $g\left(d\right)=\prod_{i\in\left[k\right]\backslash Y}\left(f_{i}\left(d\right)/c_{T_{i,0}}^{\omega\left(d\right)}\right)$.
By applying Remark \ref{rem:limiting moments} iteratively to products
of powers of the functions $f_{i}\left(d\right)/c_{T_{i,0}}^{\omega\left(d\right)}$
for $i\in\left[k\right]\backslash Y$ we obtain that $\mu_{g}$ is
a point-mass supported at $\prod_{j\in\left[k\right]\backslash Y}P_{j}^{\pm}/q_{j}$
and

\[
\lim_{X\longrightarrow\infty}\frac{\sum_{d\in\mathcal{D}_{X}^{\pm}}g\left(d\right)}{\sum_{d\in\mathcal{D}_{X}^{\pm}}1}=\prod_{j\in\left[k\right]\backslash Y}P_{j}^{\pm}/q_{j}.
\]
Let $h\left(d\right)=\prod_{i\in Y}(f_{i}\left(d\right)/c_{T_{i,0}}^{\omega\left(d\right)})$.
Then similarly by Remark \ref{rem:limiting moments} applied to the
functions $g_{i,j}$ for $i\in Y$ and Lemma \ref{lem:limiting moments}
to products of $f$ and the $g_{i,j}$ we see
\[
\lim_{X\longrightarrow\infty}\frac{\sum_{d\in\mathcal{D}_{X}^{\pm}}h\left(d\right)}{\sum_{d\in\mathcal{D}_{X}^{\pm}}1}=\sum_{i=0}^{k}\left(-1\right)^{i}\binom{k}{i}M^{\pm}\left(k-i\right)
\]
 and $\mu_{h}\left(2^{i}-1\right)=P_{CL}^{\pm}\left(i\right)$. Finally
applying Lemma \ref{lem:limiting moments} to $g$ and $h$ once more
completes the proof.

\end{proof}
\begin{cor}
\label{cor:compositum density}Let $\left(G_{1},H_{1},T_{0,1}\right),\ldots,\left(G_{k},H_{k},T_{0,k}\right)$
be a set of admissible tuples. Let $\mathcal{K}_{0}\subset\mathcal{K}$
be the set of all quadratic fields $K$ whose maximal unramified extension
$K^{un}/\mathbb{Q}$ contains a $\left(G_{i},H_{i}\right)$ extension
for all $i$.

If $H_{i}$ is non-abelian for all $i$ then $\mathcal{K}_{0}$ has
density 1. Otherwise $\mathcal{K}_{0}$ has density $1-P_{CL}\left(0\right)$.
\end{cor}

\begin{proof}
Let $f_{i}$ and $c_{i}$ be the counting function and normalizing
constant corresponding to $\left(G_{i},H_{i},T_{0,i}\right)$. Since
in the first case for each $i$ the function $h_{i}=f_{i}/c_{i}^{\omega}$
is distributed as a point mass at some non-zero value this implies
that $h_{i}$ evaluates to be non-zero at one hundred percent of quadratic
fields. Clearly the field $K^{un}$ contains the desired compositum
if $\prod_{i=1}^{k}h_{i}$ is non-zero at $K$.

On the other hand if $H_{i}$ is abelian for at least one $i$ then
$\prod_{i=1}^{k}h_{i}$ will be zero with density $P_{CL}\left(0\right)$
(since for such $i$ the functions $h_{i}$ are maximally correlated).
\end{proof}

\section{\label{sec:Additional remarks}Additional remarks}

\subsection{\label{subsec:Malle-Bhargava-heuristics}Malle-Bhargava heuristics}

As further evidence for the refined Conjecture \ref{conj:Let--be-1}
we obtain an asymptotic for $\sum_{K,0<\pm D_{K}<X}f_{T}\left(K\right)$
by giving a minor modification of an argument of Wood \cite{Wood}
who computed an asymptotic for $\sum_{K,0<\pm D_{K}<X}f\left(K\right)$.
The argument employs the field counting heuristics put forth by Bhargava
\cite{Bhargavamassformula}. In the setting of counting $G$ extensions
$L/\mathbb{Q}$ with $D_{L}<X$ the heuristics say that the asymptotic
is determined by the product of the local masses at each prime $p$:
\[
m_{p}=\frac{1}{\left|G\right|}\sum_{\phi\in S_{\mathbb{Q}_{p},G}}\frac{1}{p^{c\left(\phi\right)}}
\]
 where $S_{\mathbb{Q}_{p},G}$ is the set of homomorphisms $\mathrm{Gal}\left(\overline{\mathbb{Q}_{p}}/\mathbb{Q}_{p}\right)\longrightarrow G$,
and $c\left(\phi\right)$ is a function giving the power of $p$ in
$D_{L_{p}}$ where $L_{p}/\mathbb{Q}_{p}$ is the extension corresponding
to $\phi$.

Since we are interested only in fields counted by $f_{T}$ we modify
the above local mass as follows. The sum is restricted to morphisms
which map the absolute inertia group $I_{p}\subset\mathrm{Gal}\left(\overline{\mathbb{Q}_{p}}/\mathbb{Q}_{p}\right)$
to an element $x$ which projects to an element of $T$ in $G/\Phi\left(G\right)$.
Furthermore since we are counting extensions $L/K/\mathbb{Q}$ with
$D_{K}<X$ (rather than $D_{L}<X$) we let $c\left(\phi\right)$ be
the function giving the power of $p$ in $D_{K_{p}}$ where $K_{p}$
is the ramified quadratic subfield of $L_{p}/\mathbb{Q}_{p}$ the
field corresponding to $\phi$.

With these considerations, for each odd prime $p$ the modified local
mass is easily shown to be
\begin{align}
m_{p}^{'} & =\frac{1}{\left|G\right|}\sum_{c}\sum_{x\in c}\sum_{\left\langle x\right\rangle \le D\le G}\left|\mathrm{Aut}{}_{\left\langle x\right\rangle }\left(D\right)\right|\left|\left\{ F/\mathbb{Q}_{p}\mid\mathrm{Gal}\left(F/\mathbb{Q}_{p}\right)\cong D\right\} \right|\cdot p^{-1}\nonumber \\
 & +\sum_{D\le G}\left|\mathrm{Aut}{}_{\left\langle x\right\rangle }\left(D\right)\right|\left|\left\{ F/\mathbb{Q}_{p}\mid\mathrm{Gal}\left(F/\mathbb{Q}_{p}\right)\cong D\right\} \right|\label{eq:localmass}
\end{align}
 where the first sum is over conjugacy classes of order 2 elements
lifting $T$ and the third sum is over possible decomposition groups
with inertia group $\left\langle x\right\rangle $ and the sum in
the second term is over cyclic subgroups, since this term corresponds
to the unramified extensions.
\begin{prop}
With the above notation
\[
m_{p}^{'}=1+\frac{c_{T}}{p}.
\]
\end{prop}

\begin{proof}
Note that since $D/\left\langle x\right\rangle $ has to be cyclic
and $x$ cannot be a square in $G$ (if $x$ is a square then $x\in\Phi\left(G\right)$)
$D$ has to be of the form $C_{2}\times C_{2^{m}}$ where $m\ge0$.

For any $c$ and $m$ let 
\[
r_{p}\left(c,m\right)=\sum_{x\in c}\sum_{\substack{\left\langle x\right\rangle \le D\le G,\\
D\cong C_{2}\times C_{2^{m}}
}
}\left|\mathrm{Aut}{}_{\left\langle x\right\rangle }\left(D\right)\right|\left|\left\{ F/\mathbb{Q}_{p}\mid\mathrm{Gal}\left(F/\mathbb{Q}_{p}\right)\cong D\right\} \right|
\]
 so that by switching summations the first term in $\left(\ref{eq:localmass}\right)$
becomes
\[
\frac{1}{\left|G\right|}\sum_{c}\sum_{m\ge0}r_{p}\cdot p^{-1}.
\]

We now compute the values of $r_{p}\left(c,m\right)$. Note that $\left|\mathrm{Aut}{}_{\left\langle x\right\rangle }\left(D\right)\right|=2^{m}$.
Also, for any odd prime there is exactly 1 ramified extension $F/\mathbb{Q}_{p}$
with $\mathrm{Gal}\left(F/\mathbb{Q}_{p}\right)\cong C_{2}\times C_{2^{m}}$
for $m>0$ and exactly 2 ramified extensions for $m=0$. 

If $D\cong C_{2}$ (that is $m=0$) then
\begin{align*}
r_{p}\left(c,0\right) & =\sum_{x\in c}2\\
 & =2\left|c\right|.
\end{align*}
 Now suppose $m\ge1$. Then
\[
r_{p}\left(c,m\right)=2^{m}\sum_{x\in c}\sum_{\substack{\left\langle x\right\rangle \le D\le G,\\
D\cong C_{2}\times C_{2^{m}}
}
}1
\]
 and
\[
\sum_{\substack{\left\langle x\right\rangle \le D\le G,\\
D\cong C_{2}\times C_{2^{m}}
}
}1=\frac{\left|C_{G}\left(x\right)\left[2^{m}\right]\right|-\left|C_{G}\left(x\right)\left[2^{m-1}\right]\right|-\delta\left(m\right)}{\left|\mathrm{Aut}{}_{\left\langle x\right\rangle }\left(D\right)\right|}
\]
 where $\delta\left(m\right)=1$ if $m=1$ and 0 otherwise. We subtract
$\delta\left(m\right)$ to account for the element $x\in C_{G}\left(x\right)\left[2\right]$
which does not contribute to the count of $D\cong C_{2}\times C_{2}$
in the case $m=1$.

Thus for any $m\ge1$ we get 
\begin{align*}
r_{p}\left(c,m\right) & =\sum_{x\in c}\left|C_{G}\left(x\right)\left[2^{m}\right]\right|-\left|C_{G}\left(x\right)\left[2^{m-1}\right]\right|\\
 & =\left|c\right|\left(\left|C_{G}\left(x\right)\left[2^{m}\right]\right|-\left|C_{G}\left(x\right)\left[2^{m-1}\right]-\delta\left(m\right)\right|\right).
\end{align*}
Noting that $\left|C_{G}\left(x\right)\left[2^{m}\right]\right|=\left|C_{G}\left(x\right)\right|$
for large enough $m$ and $\left|C_{G}\left(x\right)\left[2^{0}\right]\right|+\delta\left(1\right)=2$
we get 
\begin{align*}
\frac{1}{\left|G\right|}\sum_{c}\sum_{m\ge0}r_{p}\left(c,m\right) & =\frac{1}{\left|G\right|}\sum_{c}\left|c\right|\left|C_{G}\left(x\right)\right|\\
 & =\sum_{c}1\\
 & =c_{T}.
\end{align*}
In the second equality we are using the orbit stabilizer theorem which
says $\left|G\right|=\left|c\right|\left|C_{G}\left(x\right)\right|$. 

By a similar argument one can show that 
\[
\sum_{D\le G}\left|\mathrm{Aut}{}_{\left\langle x\right\rangle }\left(D\right)\right|\left|\left\{ F/\mathbb{Q}_{p}\mid\mathrm{Gal}\left(F/\mathbb{Q}_{p}\right)\cong D\right\} \right|=\left|G\right|
\]
 where the sum is over all cyclic subgroups (recall this term corresponds
to the unramified extensions).

Thus the local mass at $p$ is 
\[
1+\frac{c_{T}}{p}.
\]
 This translates to an asymptotic of $X\left(\log X\right)^{c_{T}-1}$
for the global count, in agreement with Conjecture \ref{conj:Let--be-1}.
\end{proof}

\subsection{\label{subsec:Examples}Examples}

We list all the examples of admissible pairs $\left(G,H\right)$ along
with their graphs and asymptotics for $n=2$ and 3.

When $n=2$ the only example is $\left(G,H\right)=\left(D_{4},C_{4}\right)$
and $T_{0}=\left\{ t_{1},t_{2}\right\} $ which form an edge in $\mathcal{G}\left(T_{0}\right)$.
Hence $c=2.$

The following example shows that $f_{\left(G,H\right)}$ can have
different asymptotics (that is different values $c_{T_{0}}$) for
different subgroups $H$ of the same group $G$.

Let $n=3$. There are two possibilities for $G$ which is an admissible
central extension of $\mathbb{F}_{2}^{3}$ by $\mathbb{F}_{2}$.

Let $G$ be the central product of $D_{4}$ and $C_{4}$ (the quotient
of $D_{4}\oplus C_{4}$ obtained by identifying their Frattini subgroups).
It is isomorphic to both $Q_{8}\rtimes C_{2}$ and $\left(C_{4}\times C_{2}\right)\rtimes C_{2}$
with the appropriate actions.

If we let $H=Q_{8}$ then $T_{0}=\left\{ t_{1},t_{2},t_{3}\right\} $
such that $\mathcal{G}\left(T_{0}\right)$ is complete on its vertices.
Hence $c=3$.

If we let $H=C_{4}\times C_{2}$ then $T_{0}=\left\{ t_{1},t_{2},t_{3},t_{1}t_{2}t_{3}\right\} $
such that $\mathcal{G}\left(T_{0}\right)$ is complete bipartite with
partitions $W_{1}=\left\{ t_{1},t_{2}\right\} $ and $W_{2}=\left\{ t_{3},t_{1}t_{2}t_{3}\right\} $.
Hence $c=4$.

Let $G=D_{4}\times C_{2}$. The only possibility is $H=D_{4}$ in
which case $T_{0}=\left\{ t_{1},t_{2},t_{3}\right\} $ such that $\left(t_{1},t_{2}\right)\in\mathcal{G}\left(T_{0}\right)$
and $t_{3}$ is disconnected. Hence $c=4$.

\subsection{\label{subsec:Comparison-with-Fouvry-Klners}Comparison with Fouvry-Kl\"uners}

In the case $\left(G,H\right)=\left(D_{4},C_{4}\right)$ the formula
for our counting function $f$ in Theorem \ref{thm:f_T_formula} closely
resembles the expression computed by Fouvry and Kl\"uners for $\left|Cl_{K}\left[4\right]/Cl_{K}\left[2\right]\right|$.
They showed that the function $h\left(d\right)=\left|Cl_{K}\left[4\right]/Cl_{K}\left[2\right]\right|$
has distribution $\mu_{h}\left(2^{i}\right)=P_{CL}^{\pm}\left(i\right)$
for all $i\in\mathbb{Z}_{\ge0}$. We make explicit the relationship.

In this case the only admissible generating set $T_{0}$ is the maximal
one consisting of two elements, hence $f=f_{T_{0}}$ and $c_{T_{0}}=2$.
Hence (using the notation of Section \ref{subsec:The-complete-bipartite})
we have

\begin{align*}
g_{\mathcal{D}_{T_{0}}}\left(d\right) & =g_{\widehat{\mathcal{D}}_{T_{0}}}\left(d\right)-\sum_{T\subsetneqq T_{0}}g_{\mathcal{D}_{T}}\left(d\right)\\
 & =g_{\widehat{\mathcal{D}}_{T_{0}}}\left(d\right)-2\cdot2^{\omega\left(d\right)}
\end{align*}
 where the sum in the first line is over all strict subsets of $T_{0}$,
including non-generating sets. Furthermore in this case
\begin{align*}
g_{\mathcal{D}_{T_{0}}}\left(d\right) & =2^{n}\left|\mathrm{Aut}_{H}(G/\Phi\left(G\right))\right|f_{T_{0}}\left(d\right)\\
 & =8f_{T_{0}}\left(d\right).
\end{align*}
By Lemma \ref{lem:g in terms of Cl2} 
\[
\frac{g_{\widehat{\mathcal{D}}_{T_{0}}}\left(d\right)}{2^{\omega\left(d\right)}}=2\left|Cl_{K}\left[4\right]/Cl_{K}\left[2\right]\right|
\]
 so by the above it follows that 
\begin{align*}
\frac{f\left(d\right)}{2^{\omega\left(d\right)}} & =\frac{1}{8}\left(2\left|Cl_{K}\left[4\right]/Cl_{K}\left[2\right]\right|-2\right)\\
 & =\frac{\left|Cl_{K}\left[4\right]/Cl_{K}\left[2\right]\right|-1}{4}.
\end{align*}
 Thus in this case the function $f$ is supported exactly on the same
set as its distribution $\mu_{f/2^{\omega}}$ which by Theorem \ref{thm:dist complete bipartite}
(or directly from the above result of Fouvry-Kl\"uners) is
\[
\mu_{h}\left(\left(2^{i}-1\right)/4\right)=P_{CL}^{\pm}\left(i\right)
\]
 for all $i\in\mathbb{Z}_{\ge0}$ (we reiterate that for any pair
$\left(G,H\right)\neq\left(D_{4},C_{4}\right)$ $f$ will be supported
away from this set with positive density).

\bibliographystyle{alpha}
\bibliography{2groups}

\textsc{~}

\textsc{\small{}Department of Mathematics and Statistics, University
of Calgary, Canada}{\small\par}

{\small{}$\mathtt{jack.klys@ucalgary.ca}$}{\small\par}
\end{document}